\documentclass[graybox]{svmult}
\usepackage{mathptmx}       % selects Times Roman as basic font

\usepackage{helvet}         % selects Helvetica as sans-serif font
\usepackage{courier}        % selects Courier as typewriter font
\usepackage{makeidx}         % allows index generation
\usepackage{graphicx}        % standard LaTeX graphics tool
\usepackage{multicol}        % used for the two-column index
\usepackage[bottom]{footmisc}% places footnotes at page bottom

\makeindex             % used for the subject index
\usepackage{hyperref}
\usepackage{float}
\usepackage{caption}
\usepackage{subcaption}
\usepackage{bm}
\usepackage{amsmath,amssymb,amsfonts}        % AMS math packages
\hypersetup{
    colorlinks=true,
    linkcolor=black,
    filecolor=black,      
    urlcolor=black,
    citecolor=black,
    }

\allowdisplaybreaks

\newcommand{\cP}{\mathcal{P}}

\newcommand{\N}{\mathbb N}
\newcommand{\R}{\mathbb R}

\renewcommand{\E}{\mathbb E}
\DeclareMathOperator{\Tr}{Tr}
\DeclareMathOperator{\NC}{NC}
\DeclareMathOperator{\PS}{S}

\begin{document}
\title*{Corrections to Classical Matrix Ensemble Moments, Non-Crossing Annular Pairings, and Ribbon Graphs}
\titlerunning{Ribbon Graphs and Non-Crossing Annular Pairings}
\author{Anas A. Rahman, Daniel Munoz George, and James A. Mingo}
\institute{Anas A. Rahman \at Department of Mathematics, The University of Hong Kong, Pok Fu Lam, Hong Kong, \email{aarahman@hku.hk}
\and Daniel Munoz George \at Department of Mathematics, The University of Hong Kong, Pok Fu Lam, Hong Kong, \email{danielmg@hku.hk}
\and James A. Mingo \at Department of Mathematics and Statistics, Jeffery Hall, Queen's University, Kingston, ON K7L 3N6, Canada,  \email{james.mingo@queensu.ca}}
\maketitle

\abstract{
We elucidate a bijection between ribbon graphs on the real projective plane and non-crossing annular pairings that relate to the $1/N$ correction term of the GOE and LOE spectral moments. We also derive analogous objects for the $1/N^2$ correction terms of said moments and their equivalents for the GUE and LUE.
}

\section{Introduction}

Since its inception, random matrix theory has steadily gained attention due to its \textit{unreasonable effectiveness} in studying heavy atom spectra \cite{Wig55,Dys62b,Wig67}, wireless communication channels \cite{TV04}, random quantum states \cite{CN16}, log-gases \cite{Fo10}, and quantum transport \cite{Be97} (to give a non-exhaustive list), with a recent surge in interest stemming from advances in the fields of high-dimensional data and machine learning. The most famous result in the field is undoubtedly due to the pioneering work \cite{Wig55,Wig58} of Wigner, who showed that in the limit of large matrix size and with appropriate scaling, the limiting distribution of the eigenvalues of so-called Wigner matrices is the celebrated \textit{Wigner semicircle law},
\begin{equation}
\rho(x) = \frac{1}{2\pi}\sqrt{4-x^2}.
\end{equation}
To be precise, let $X$ be a real symmetric or complex Hermitian $N\times N$ random matrix with independently distributed entries $X_{ij}$ ($1\le i\le j\le N$)  of mean zero, identical variance for $i=j$, variance $\E|X_{ij}|^2=1$ for $i\ne j$, and all other moments having upper bound independent of $i,j$. Then, $X$ is said to be a Wigner matrix and it is precisely the eigenvalues of $\widetilde{X}:=\frac{1}{\sqrt{N}}X$ that follow the above semicircular distribution in the large $N$ limit.

The seminal result of Wigner recounted above was obtained through the \textit{method of moments}, where one uses the fact that the normalised moments
\begin{equation*}
\widetilde{m}_n=\int_{\R}x^n\rho_N(x)\,\mathrm{d}x\quad(n=0,1,2,\ldots)
\end{equation*}
of the density $\rho_N(x)$ of the eigenvalues of $\widetilde{X}$ fully characterise $\rho_N(x)$ when Carleman's condition \cite[Sec.~88]{Wal48} is satisfied. Wigner studied said moments by noting that they are are given by the averages with respect to the distribution of $\widetilde{X}$,
\begin{equation*}
\widetilde{m}_n=\frac{1}{N}\E\Tr\widetilde{X}^n,
\end{equation*}
which he then showed via combinatorial arguments to be given in the large $N$ limit by the $(n/2)\textsuperscript{th}$ Catalan number when $n$ is even (the moments vanish for odd $n$ by symmetry) --- these characterise the Wigner semicircle law, hence $\rho_N(x)$ converges weakly to $\rho(x)$. Since Wigner's application of the method of moments, moments of random matrices of the form $\E\Tr M^n$ have been interpreted as counts of graphs in a variety of ways, with the physics community largely using diagrammatic theories \textit{\`a la} Feynman to simplify matrix integrals \cite{Hoo74,BIPZ78,BIZ80,YK83,VWZ84}, while combinatorialists and algebraic geometers have dually used matrix integrals to enumerate graphs \cite{HZ86,Kon92,GJ97,GHJ01}. In particular, when matrix entries are simple sums and products of Gaussian variables, one has access to the Isserlis--Wick theorem \cite{Iss18,Wic50}, which can be used to show that the moments $\E\Tr M^n$ are given by genus expansions in $1/N$ whose coefficients are counts of ribbon graphs. We review this connection in \S\ref{s1.1} for the Gaussian and Laguerre orthogonal and unitary ensembles.

Moving onto another perspective, it has been shown in recent years that the infinitesimal moments, equivalently $1/N$ correction to the spectral moments, of the Gaussian orthogonal ensemble (GOE) enumerates a certain class of non-crossing annular pairings \cite{mingo2019}. In this paper we show this can be extended to the $N^{-2}$ terms in the genus expansion of the GOE and the GUE, despite the difficulty in enumeration of genus $1$ permutations (see Cori and Hetyei \cite{ch14}).  We expect this extension to be relevant to the problem of finding the correction to Fevrier's definition of higher order infinitesimal freeness; according to the definition of \cite{fev12}, independent GUEs are not asymptotically infinitesimally free of higher order as it can be seen from the tables of Harer and Zagier \cite{HZ86} that $\kappa_4'' = 1$, whereas infinitesimal freeness of second order requires that $\kappa_4''= 0$. Moreover, the recursion of \cite{HZ86} still lacks an elementary topological rationale. There is now evidence that  higher order freeness \cite{collins2007second} is necessary for both said rationale and the correction to Fevrier's higher order infinitesimal freeness, and this theory is based on non-crossing permutations on a multi-annulus. These are exactly the objects of consideration in the present paper. 

The goal of this paper is to highlight the relationship between ribbon graphs and non-crossing annular pairings in the context of low order corrections to the spectral moments of the Gaussian and Laguerre orthogonal and unitary ensembles. Thus, in the remainder of this introduction, we review the relevant theory of ribbon graphs and non-crossing annular pairings. Then, in \S\ref{s2.1}, we give a bijection between the ribbon graphs and non-crossing annular pairings enumerated by the infinitesimal moments of the GOE, while in \S\ref{s2.2} and \S\ref{s2.3}, we derive non-crossing annular pairings corresponding to the $1/N^2$ corrections to the spectral moments of the Gaussian unitary ensemble (GUE) and GOE, respectively. We give analogous arguments for the corresponding Laguerre ensembles (LOE and LUE) in Section \ref{s3}.

From the topological viewpoint, one motivation for studying these new classes of non-crossing annular pairings is that they give a tractable, planar classification of the corresponding ribbon graphs. In particular, we give planar descriptions of the ribbon graphs that can be embedded into the torus, real projective plane, and Klein bottle. These ribbon graphs are special because, by the classification theorem for closed surfaces, every other non-planar ribbon graph involved in the computation of the moments of interest can be constructed by gluing a finite number of torii to them --- we leave it to future work to understand how gluing torii to ribbon graphs should be interpreted in the setting of non-crossing annular pairings.

\subsection{Matrix Moments and Ribbon Graphs} \label{s1.1}

We begin by introducing the random matrix ensembles studied in this paper, which are constructed from the real and complex Ginibre ensembles.

\begin{definition}[GOE, GUE, LOE, and LUE]
The $M\times N$ real and complex \textit{Ginibre ensembles} \cite{Gin65} are respectively the sets of $M\times N$ real ($\beta=1$) and complex ($\beta=2$) matrices with probability density functions (p.d.f.s)
\begin{equation*}
P(G)=\pi^{-MN\beta/2}\exp\left(-\Tr\,G^\dagger G\right)=\prod_{i=1}^M\prod_{j=1}^N\pi^{-\beta/2}\exp\left(-|G_{i,j}|^2\right)
\end{equation*}
defined with respect to the Lebesgue measure
\begin{equation*}
\mathrm{d}G=\prod_{i=1}^M\prod_{j=1}^N\prod_{s=1}^\beta\mathrm{d}G_{i,j}^{(s)},
\end{equation*}
where we write $G_{i,j}^{(1)}=\mathrm{Re}(G_{i,j})$ and $G_{i,j}^{(2)}=\mathrm{Im}(G_{i,j})$. We say that
\begin{itemize}
\item $H=\tfrac{1}{2}(G^{\intercal}+G)$ belongs to the $N\times N$ \textit{Gaussian orthogonal ensemble} when $G$ is drawn from the $N\times N$ real Ginibre ensemble;
\item $H=\tfrac{1}{2}(G^\dagger+G)$ belongs to the $N\times N$ \textit{Gaussian unitary ensemble} when $G$ is drawn from the $N\times N$ complex Ginibre ensemble;
\item $W=G^{\intercal}G$ belongs to the $(M,N)$ \textit{Laguerre orthogonal ensemble} when $G$ is drawn from the $M\times N$ real Ginibre ensemble;
\item $W=G^\dagger G$ belongs to the $(M,N)$ \textit{Laguerre unitary ensemble} when $G$ is drawn from the $M\times N$ complex Ginibre ensemble.
\end{itemize}
\end{definition}

\begin{remark} \label{rmk1}
The entries of real Ginibre matrices are independent centred normal variables of variance $1/2$. For $G$ a complex Ginibre matrix, the entries are independent complex normal variables such that $\E[G_{i,j}^2]=0$, but $\E[G_{i,j}\overline{G}_{i,j}]=1$. 
\end{remark}

The key property of the ensembles defined above is that the entries of Ginibre matrices are Gaussian, so for $M$ drawn from one of said ensembles, the moments $\E\Tr M^n$ can be calculated through the Isserlis--Wick theorem \cite{Iss18,Wic50}.

\begin{theorem}[Isserlis--Wick]
For $n\in\N$, let $x_1,\ldots,x_n$ be centred normal variables. Then, $\E[x_1\cdots x_n]$ vanishes for $n$ odd, while for $n$ even, it decomposes as
\begin{equation}
\E[x_1\cdots x_n] = \sum_{\pi\in \cP_2(n)}\prod_{\{u,v\}\in\pi}\E[x_u x_v],
\end{equation}
where $\cP_2(n)$ denotes the set of pairings of $[n]:=\{1,\ldots,n\}$, equivalently the partitions $\pi$ of $[n]$ whose blocks $\{u,v\}\in\pi$ are all of size two (see \S\ref{s1.2}).
\end{theorem}

We now give a brief overview of the ribbon graphs relevant to this paper;  refer to \cite[Sec.~3.3]{Rah22} and references therein for details. We begin with the GUE.

\subsubsection*{GUE Moments and Orientable Ribbon Graphs}
Fix $n\in\mathbb{N}$ and let $H$ be an $N\times N$ GUE matrix. As the entries of $H$ are centred normal variables, the Isserlis--Wick theorem says that $\E\Tr H^n$ vanishes for $n$ odd, while for $n$ even, we have that
\begin{align}
\E\Tr H^n&=\sum_{i_1,\ldots,i_n=1}^N\E\left[H_{i_1,i_2}H_{i_2,i_3}\cdots H_{i_n,i_1}\right] \nonumber
\\&=\sum_{i_1,\ldots,i_n=1}^N\sum_{\pi\in\cP_2(n)}\prod_{\{u,v\}\in\pi}\E\left[H_{i_u,i_{u+1}}H_{i_v,i_{v+1}}\right], \label{eq3}
\end{align}
where we define $i_{n+1}\equiv i_1$. Now, we have by Remark \ref{rmk1} that
\begin{equation} \label{eq4}
\E\left[H_{i_u,i_{u+1}}H_{i_v,i_{v+1}}\right]=\tfrac{1}{2}\chi_{i_u=i_{v+1}}\chi_{i_v=i_{u+1}},
\end{equation}
where the indicator function $\chi_A$ equals one when $A$ is true and vanishes otherwise. Thus, interchanging the order of summation in equation \eqref{eq3} above shows that
\begin{equation} \label{eq5}
\E\Tr H^n=2^{-n/2}\sum_{\pi\in\cP_2(n)}w(\pi),
\end{equation}
where we define the weight of a pairing $\pi$ to be
\begin{equation} \label{eq6}
w(\pi)=\sum_{i_1,\ldots,i_n=1}^N\prod_{\{u,v\}\in\pi}\chi_{i_u=i_{v+1}}\chi_{i_v=i_{u+1}}.
\end{equation}
The product of indicator functions in the above vanishes unless certain summation indices $i_1,\ldots,i_n$ are identified, in which case it equals one. Thus, $w(\pi)$ is a sum of unity over the indices that are unconstrained by the indicator functions and is hence given by $N^V$, where $V$ is the number of such free summation indices.

It is at this point that ribbon graphs present themselves as a convenient method for computing $w(\pi)$. We proceed by drawing an $n$-gon to represent $\Tr H^n$, with the $u\textsuperscript{th}$ side labelled clockwise by $H_{i_u,i_{u+1}}$. Then, the vertices of this $n$-gon inherit a clockwise labelling by the indices $i_1,\ldots,i_n$. We then represent a pairing $\pi$ by connecting edges labelled by $H_{i_u,i_{u+1}}$ and $H_{i_v,i_{v+1}}$ by an untwisted ribbon whenever $\{u,v\}\in\pi$, with the result being called a ribbon graph. Note that the ribbons identify vertices of our $n$-gon in compliance with the indicator functions present in $w(\pi)$ so that $w(\pi)=N^V$, where $V$ is now the number of unique vertices of our $n$-gon after the ribbon-identification process. It is equivalent to say that $V$ is the number of boundaries of the ribbon graph representing $\pi$. See figures \ref{fig1a}--\ref{fig1c} for an example.
\begin{figure}
    \centering
    \begin{subfigure}{0.32\textwidth}
        \centering
        \captionsetup{justification=centering}
        \includegraphics[width=0.7\textwidth]{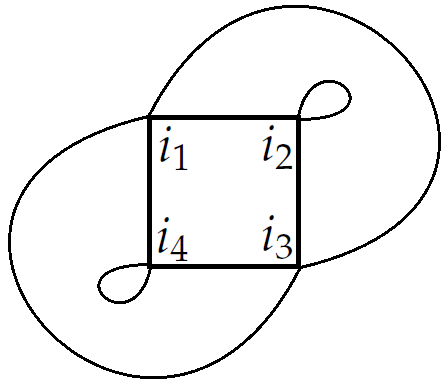}
        \caption{$i_1\equiv i_3$} \label{fig1a}
    \end{subfigure}\hfill
    \begin{subfigure}{0.32\textwidth}
        \centering
        \captionsetup{justification=centering}
        \includegraphics[width=0.7\textwidth]{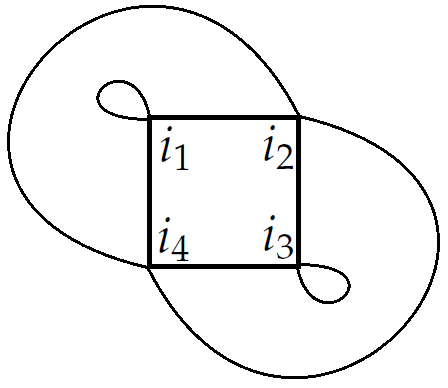}
        \caption{$i_2\equiv i_4$}
    \end{subfigure}\hfill
    \begin{subfigure}{0.32\textwidth}
        \centering
        \captionsetup{justification=centering}
        \includegraphics[width=0.7\textwidth]{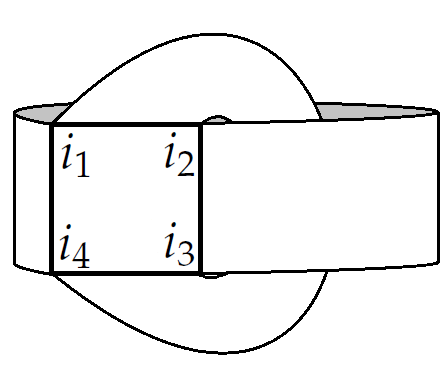}
        \caption{$i_1\equiv i_2\equiv i_3\equiv i_4$} \label{fig1c}
    \end{subfigure}
    \caption{These ribbon graphs encode how one must identify summation indices when computing $\E\Tr H^4$ for the GUE. Their weights are respectively $N^3$, $N^3$, and $N$.}
\end{figure}

After the ribbon-identification process described above, our $n$-gon is homeomorphic to an orientable manifold with cell decomposition having one unique face, $n/2$ unique edges, and $V$ unique vertices. Thus, by Euler's formula, its Euler characteristic is $\chi=F-E+V=1-n/2+V$. The Euler characteristic is also given by $\chi=2-2g$, where $g\in\N$ is the genus of our ribbon graph. Thus, we see that $V=n/2+1-2g$ and may rewrite equation \eqref{eq5} as a genus expansion.

\begin{proposition}[GUE moments] \label{prop1}
Let $n\in\mathbb{N}$ be even and define $a_g(n)$ to be the set of orientable ribbon graphs of genus $g$ built from an $n$-gon. Then, for $H$ an $N\times N$ GUE matrix,
\begin{equation}
m_n^{(\mathrm{GUE})}:=\E\Tr H^n=\left(\frac{N}{2}\right)^{n/2}\sum_{g=0}^{\lfloor n/4\rfloor}N^{1-2g}|a_g(n)|.
\end{equation}
\end{proposition}

\begin{remark}
As per Wigner's semicircle law, $|a_0(n)|$ is the $(n/2)\textsuperscript{th}$ Catalan number.
\end{remark}

\subsubsection*{GOE Moments and Locally Orientable Ribbon Graphs}
Most of the above arguments follow identically in the GOE case. The first point of difference is that equation \eqref{eq4} must be replaced with 
\begin{equation*}
\E\left[H_{i_u,i_{u+1}}H_{i_v,i_{v+1}}\right]=\tfrac{1}{4}\left(\chi_{i_u=i_{v+1}}\chi_{i_v=i_{u+1}}+\chi_{i_u=i_v}\chi_{i_{u+1}=i_{v+1}}\right)
\end{equation*}
so that the GOE analogues of equations \eqref{eq5} and \eqref{eq6} are
\begin{align}
\E\Tr H^n&=2^{-n}\sum_{\pi\in\cP_2(n)}w(\pi),
\\ w(\pi)&=\sum_{i_1,\ldots,i_n=1}^N\prod_{\{u,v\}\in\pi}\left(\chi_{i_u=i_{v+1}}\chi_{i_v=i_{u+1}}+\chi_{i_u=i_v}\chi_{i_{u+1}=i_{v+1}}\right).
\end{align}
We yet again compute $w(\pi)$ by constructing ribbon graphs as before. This time, the term $\chi_{i_u=i_{v+1}}\chi_{i_v=i_{u+1}}$ is represented by an untwisted ribbon connecting $H_{i_u,i_{u+1}}$ to $H_{i_v,i_{v+1}}$, but the term $\chi_{i_u=i_v}\chi_{i_{u+1}=i_{v+1}}$ represents the same ribbon twisted once in the fashion of a M\"obius band. The genus expansion follows as before, but we must now consider the Euler genus $k$, which relates to the Euler characteristic of a non-orientable ribbon graph according to $\chi=2-k$. See Figure \ref{fig2} for examples.
\begin{figure}
    \centering
    \includegraphics[width=0.8\linewidth]{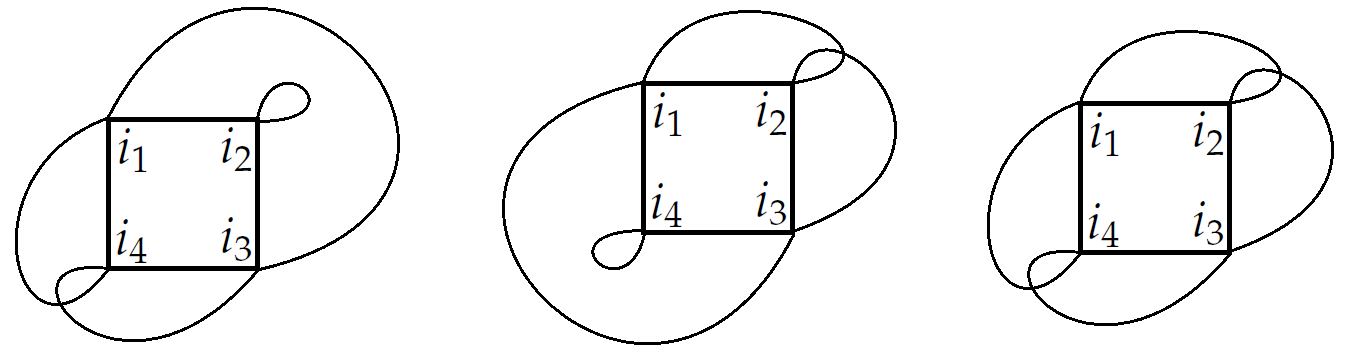}
    \caption{These ribbon graphs contribute to $\E\Tr H^4$ for the GOE. They are respectively of Euler genus $1$, $1$, and $2$.}
    \label{fig2}
\end{figure}
\begin{proposition}[GOE moments] \label{prop2}
Let $n\in\N$ be even and define $b_k(n)$ to be the set of non-orientable ribbon graphs of Euler genus $k$ built from an $n$-gon. Then, for $H$ an $N\times N$ GOE matrix,
\begin{equation}
m_n^{(\mathrm{GOE})}:=\E\Tr H^n=\left(\frac{N}{4}\right)^{n/2}\sum_{k=0}^{n/2}N^{1-k}\left(|a_{k/2}(n)|\chi_{k/2\in\N}+|b_k(n)|\right).
\end{equation}
\end{proposition}

\subsubsection*{LUE and LOE Moments and Bipartite Ribbon Graphs}
Moving onto the Laguerre ensembles, let us first consider $W=G^\dagger G$ drawn from the $(M,N)$ LUE so that $G$ is an $M\times N$ complex Ginibre matrix. Then, the LUE analogue of equation \eqref{eq3} is
\begin{align*}
\E\Tr W^n&=\sum_{i_1,\ldots,i_n=1}^N\sum_{j_1,\ldots,j_n=1}^M\E\left[G_{i_1,j_1}^\dagger G_{j_1,i_2}\cdots G_{i_n,j_n}^\dagger G_{j_n,i_1}\right]
\\&=\sum_{i_1,\ldots,i_n=1}^N\sum_{j_1,\ldots,j_n=1}^M\sum_{\pi\in \PS_n}\prod_{u=1}^n\E\left[\overline{G}_{j_u,i_u}G_{j_{\pi(u)},i_{\pi(u)+1}}\right],
\end{align*}
where we recall that $i_{n+1}\equiv i_1$ and write $\PS_n$ for the set of permutations on $[n]$. The equivalents of equations \eqref{eq5} and \eqref{eq6} are then
\begin{align}
\E\Tr W^n&=\sum_{\pi\in \PS_n}w(\pi),
\\ w(\pi)&=\sum_{i_1,\ldots,i_n=1}^N\sum_{j_1,\ldots,j_n=1}^M\prod_{u=1}^n\chi_{i_u=i_{\pi(u)+1}}\chi_{j_u=j_{\pi(u)}}.
\end{align}
To compute $w(\pi)$, we construct a ribbon graph representation of $\pi$ by drawing a $2n$-gon with vertices labelled clockwise by the summation indices $i_1,j_1,i_2,j_2,\ldots,i_n,j_n$ and then connect edges pairwise with untwisted ribbons such that vertices labelled by $i$-indices can only be identified with each other and likewise for those labelled by $j$-indices.
\begin{figure}
    \centering
    \includegraphics[width=0.7\linewidth]{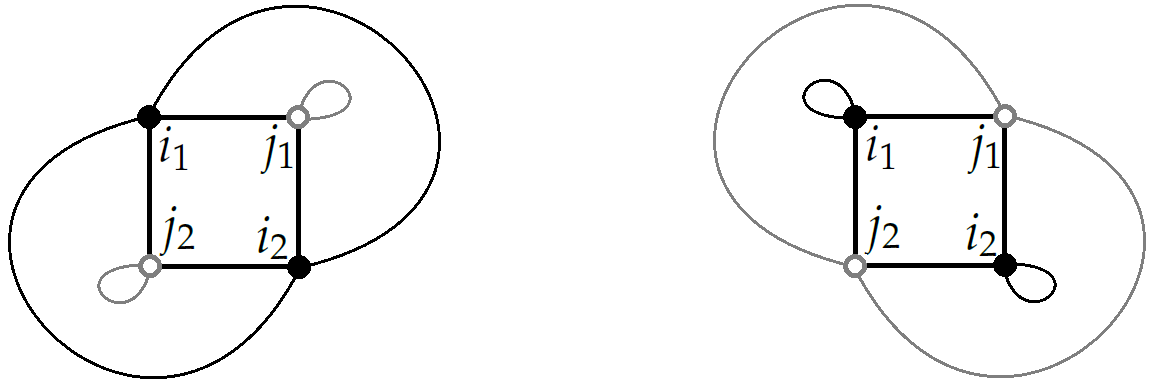}
    \caption{These ribbon graphs contribute to $\E\Tr W^2$ for the LUE. Assigning the colour black (white) to vertices labelled by $i_1,i_2$ ($j_1,j_2$) and letting $M=cN$, we see that these graphs are respectively weighted $NM^2=c^2N^3$ and $N^2M=cN^3$.}
    \label{fig3}
\end{figure}
\begin{proposition}[LUE moments] \label{prop3}
Let $n\in\N$ and consider a $(2n)$-gon with vertices alternately coloured black and white. Define $\widetilde{a}_{g,p}(n)$ to be the set of orientable ribbon graphs of genus $g$ built from said $(2n)$-gon that have $p$ white boundaries. Then, taking $M=cN$ for simplicity, the moments of a matrix $W$ drawn from the $(M,N)$ LUE admit the genus expansion
\begin{equation}
m_n^{(\mathrm{LUE})}:=\E\Tr W^n=N^n\sum_{g=0}^{\lfloor n/2\rfloor}\sum_{p=1}^nc^pN^{1-2g}|\widetilde{a}_{g,p}(n)|.
\end{equation}
\end{proposition}

In the LOE case, we must allow M\"obius twisted ribbons that represent variances of the form
\begin{equation*}
\E\left[G_{i_u,j_u}^{\intercal}G_{i_v,j_v}^{\intercal}\right]=\E\left[G_{j_u,i_{u+1}}G_{j_v,i_{v+1}}\right]=\tfrac{1}{2}
\end{equation*}
in addition to the untwisted ribbons representing the covariance
\begin{equation*}
\E\left[G_{i_u,j_u}^{\intercal}G_{j_v,i_{v+1}}\right]=\tfrac{1}{2}.
\end{equation*}
Nonetheless, the resulting ribbon graphs have consistently coloured boundaries when the underlying $(2n)$-gons have alternately coloured vertices.
\begin{proposition}[LOE moments] \label{prop4}
Let $n\in\N$ and define $\widetilde{b}_{k,p}(n)$ to be the set of non-orientable ribbon graphs of Euler genus $k$ with $p$ white boundaries built from the same $(2n)$-gons as the ribbon graphs of Proposition \ref{prop3}. Then, taking $M=cN$ yet again, the moments of a matrix $W$ drawn from the $(M,N)$ LOE have the genus expansion
\begin{equation}
m_n^{(\mathrm{LOE})}:=\E\Tr W^n=\left(\frac{N}{2}\right)^n\sum_{k=0}^n\sum_{p=1}^nc^pN^{1-k}\left(|\widetilde{a}_{k/2,p}(n)|\chi_{k/2\in\N}+|\widetilde{b}_{k,p}(n)|\right).
\end{equation}
\end{proposition}

\subsection{Non-Crossing Annular Pairings} \label{s1.2}

Parallel to the viewpoint of ribbon graphs, matrix moments and cumulants are also related to counts of non-crossing annular pairings and related objects. Let us proceed with a brief introduction to the definitions and notation used in this paper regarding partitions and permutations of a set. For detailed discussion of these concepts, the standard references are \cite{NicaSpeicher,MingoSpeicher2017}. 

\begin{definition}[Partitions and pairings]
A \textit{partition} of a set $A$ is a disjoint collection of subsets of $A$, called the \textit{blocks} of the partition, whose union is $A$. We denote the set of partitions of a set $A$ by $\cP(A)$ and define the set of \textit{pairings} $\cP_2(A)$ of $A$ to be the subset of $\cP(A)$ whose partitions all have blocks of size exactly two.
\end{definition}

For finite sets $A$, it is well known that $\cP(A)$ is a partially ordered set (poset), in fact a lattice, with respect to the refinement ordering. This means in particular that every pair of partitions in $\cP(A)$ has a well-defined join.
\begin{definition}[Join]
Let $A$ be a finite set and let $\pi,\sigma\in\cP(A)$. We henceforth endow $\cP(A)$ with the refinement ordering $\leq$ and consider the poset $(\cP(A),\leq)$. Under this ordering, we write $\pi\le\sigma$ whenever each and every block of $\pi$ is contained within some block of $\sigma$. The \textit{join} $\pi\vee\sigma$ of $\pi,\sigma$ is the smallest partition in $\cP(A)$ such that $\pi,\sigma\le\pi\vee\sigma$.
\end{definition}
For $A$ a finite set, the smallest and largest partitions in $\cP(A)$ are respectively
\begin{align*}
0_A&:=\{\{a\}\}_{a\in A},
\\ 1_A&:=\{A\}.
\end{align*}

In this paper, we will also consider permutations, writing $\PS_A$ for the set of permutations on a finite set $A$, with $\PS_n$ and $\PS_{\pm n}$ reserved for the special sets $[n]$ and $\pm[n]:=\{-n,\ldots,-1\}\cup[n]$, respectively. Any permutation can be seen as a partition by forgetting the ordering of the cycles of the permutation and interpreting them as the (unordered) blocks of the equivalent partition. Conversely, every partition can canonically be seen as a permutation by assigning each block increasing order. Moving forward, we will interchange notation for permutations and partitions, with their meaning being clear from context.

\begin{definition}[Non-crossing property] \label{def4}
Given a set $A$ of size $n$ and $\gamma\in \PS_A$, we say that $\pi\in \PS_A$ is \textit{non-crossing with respect to} $\gamma$ if
\begin{enumerate}
    \item $\pi\vee \gamma=1_A$,
    \item $\#(\pi)+\#(\pi^{-1}\gamma)+\#(\gamma)=n+2$,
\end{enumerate}
where we take this opportunity to define $\#(\sigma)$ to be the number of cycles (blocks) of a permutation (partition) $\sigma$. A partition $\pi\in\cP(A)$ is said to be non-crossing with respect to $\gamma\in\cP(A)$ if their canonical permutation analogue is non-crossing, according to the above, with respect to the permutation corresponding to $\gamma$.
\end{definition}

\begin{remark} \label{rmk3}
Defintion \ref{def4} generalises the standard notion of non-crossing partitions (on $[n]$ with respect to $1_n=(1,\ldots,n)$), where a partition $\pi\in\cP(n)$ is said to be non-crossing if none of its blocks cross. A block $A\in\pi$, in turn, is said to be \textit{crossing} if there exists another block $B\in\pi$ such that there exist $a,a'\in A$ and $b,b'\in B$ with $a<b<a'<b'$.
\end{remark}

\begin{figure}[H]
    \centering
    \includegraphics[width=0.4\linewidth]{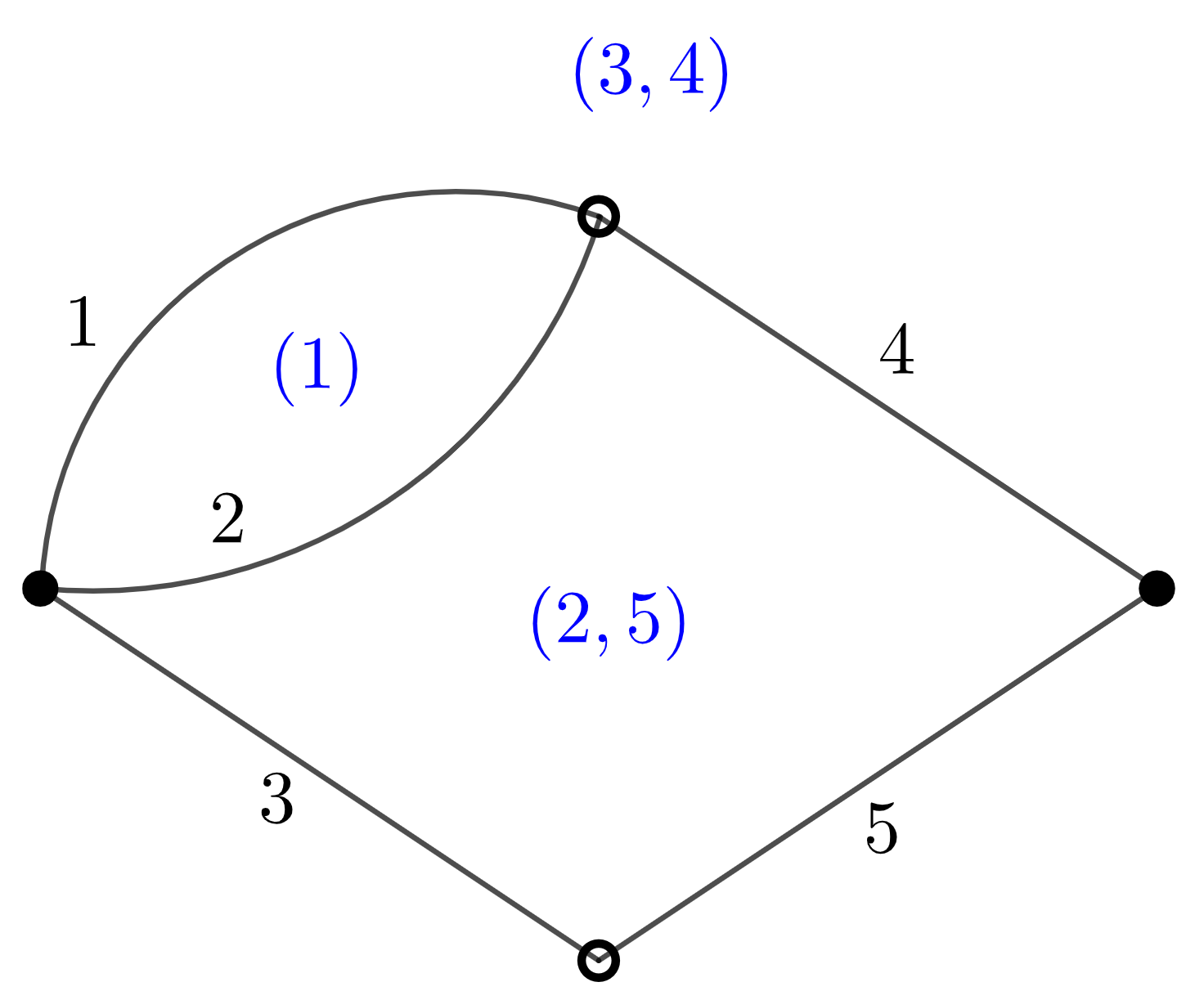}
    \caption{This is a hypermap on $A=[5]$ with vertex permutation $\gamma=(1,2,3)(4,5)$ and hyperedge permutation $\pi=(1,2,4)(3,5)$. Equivalently, it is a bipartite map with $|A|=n=5$ edges where the cycles of $\gamma$ represent the black vertices and the cycles of $\pi$ represent white vertices. Note that $\pi$ is non-crossing with  respect to $\gamma$ because this hypermap is connected and represents a cell decomposition of the genus zero sphere with faces given by the cycles of $\pi^{-1}\gamma=(1)(2,5)(3,4)$.}
    \label{fig:hypermap}
\end{figure}

\begin{remark} \label{rmk4}
The reader may recognize the triple $(A,\gamma,\pi)$ as an orientable combinatorial hypermap or bipartite map (else, see Section \ref{s3}). Then, condition (1) above says that the group $\langle \pi,\gamma\rangle$ generated by $\pi,\gamma$ is transitive on $A$, meaning that its orbit is exactly $A$ and the related hypermap is connected. Condition (2) on the other hand ensures a hypermap of genus one: $\#(\pi)$ counts the number of hyperedges or white vertices of our bipartite map, $\#(\pi^{-1}\gamma)$ counts the number of faces, and $\#(\gamma)$ counts the number of (black) vertices. Thus, our bipartite map has $F=\#(\pi^{-1}\gamma)$ faces, $E=n$ edges, and $V=\#(\pi)+\#(\gamma)$ vertices, so by Euler's formula $\chi=F-E+V$, we have that the Euler characteristic is $\#(\pi)+\#(\pi^{-1}\gamma)+\#(\gamma)-n$. As this is also equal to $2-2g$, one recovers condition (2) above exactly when setting the genus $g$ to zero. See Figure \ref{fig:hypermap} for an example.
\end{remark}

\begin{definition}[Notation]\label{def5}
Let $n\in\N$, $A$ be a finite set, $\gamma\in \PS_A$, $1_n:=(1,\ldots,n)\in \PS_n$, $\widetilde{1}_n:=(1,\ldots,n)(-n,\ldots,-1)\in \PS_{\pm n}$, and $\widetilde{\gamma}\in\PS_{\pm n}$. We define
\begin{itemize}
\item $\NC(\gamma)\subseteq \PS_A$ to be the set of non-crossing permutations of                                                                        $A$ with respect to $\gamma$;
\item $\NC_2(\gamma)\subseteq \NC(\gamma)$ to be the set of non-crossing pairings of $A$ with respect to $\gamma$;
\item $\NC^\delta(\widetilde{\gamma})\subset\PS_{\pm n}$ to be the set of permutations $\pi$ in $\NC(\widetilde{\gamma})$ such that, for $r\in[n]$, $(-r,r)$ is never a cycle of $\pi$ and for every cycle, say $(r_1,\dots, r_s)$, of $\pi$, $\pi$ must also contain the cycle $(-r_s,\dots,-r_1)$;
\item $\NC_2^\delta(\widetilde{\gamma})\subset\PS_{\pm n}$ to be the set of pairings in $\NC^\delta(\widetilde{\gamma})$;
\item $\NC(n):=\NC(1_n)$ to be the set of non-crossing (disk) permutations of $[n]$;
\item $\NC_2(n):=\NC_2(1_n)$ to be the set of non-crossing (disk) pairings of $[n]$;
\item the set $\NC^{\delta}(n,-n)\subset \PS_{\pm n}$ of \textit{symmetric non-crossing annular permutations of} $\pm[n]$ to be $\NC(\widetilde{1}_n)$;
\item the set $\NC_2^{\delta}(n,-n)\subset \PS_{\pm n}$ of \textit{symmetric non-crossing annular pairings of} $\pm[n]$ to be the set of pairings in $\NC^{\delta}(n,-n)$.
\end{itemize}
\end{definition}

In this paper, we will mostly deal with non-crossing disk and annular permutations, which correspond respectively to the cases $\#(\gamma)=1,2$. Non-crossing disk permutations are of particular interest in free probability and random matrix theory, see for instance \cite{biane1997some,mingo2024asymptotic,male2022joint,NicaSpeicher}. Non-crossing annular permutations were first studied in \cite{mingo2004annular} and have since then been used to define new notions of independence such as second order freeness \cite{mingo2007secondPart1}, which was later generalised to higher order freeness in the series of papers \cite{mingo2007second,collins2007second}.
\begin{example}\label{Example:non crossing pairings}
Set $n=6$ and let $\gamma=1_6\in \PS_6$. Then, $\pi=(1,2)(3,6)(4,5)$ is a non-crossing disk pairing since $\pi^{-1}\gamma=(1)(2,6)(3,5)(4)$ and we have
\begin{equation*}
\#(\pi)+\#(\pi^{-1}\gamma)+\#(\gamma)=6+2.
\end{equation*}
If instead, $n=8$ and $\gamma=(1,2,3,4)(5,6,7,8)\in \PS_8$, then $\pi=(1,8)(2,3)(4,5)(6,7)$ is a non-crossing annular pairing with respect to $\gamma$ because (checking Definition \ref{def4})
\begin{enumerate}
\item $\pi$ contains the cycle $(1,8)$, which intersects with both cycles of $\gamma$, so $\pi\vee\gamma=1_{[8]}$;
\item $\pi^{-1}\gamma=(1,3,5,7)(2)(4,8)(6)$, so
\begin{equation*}
\#(\pi)+\#(\pi^{-1}\gamma)+\#(\gamma)=8+2.
\end{equation*}
\end{enumerate}
\end{example}

\begin{figure}
    \centering
    \includegraphics[width=0.5\linewidth]{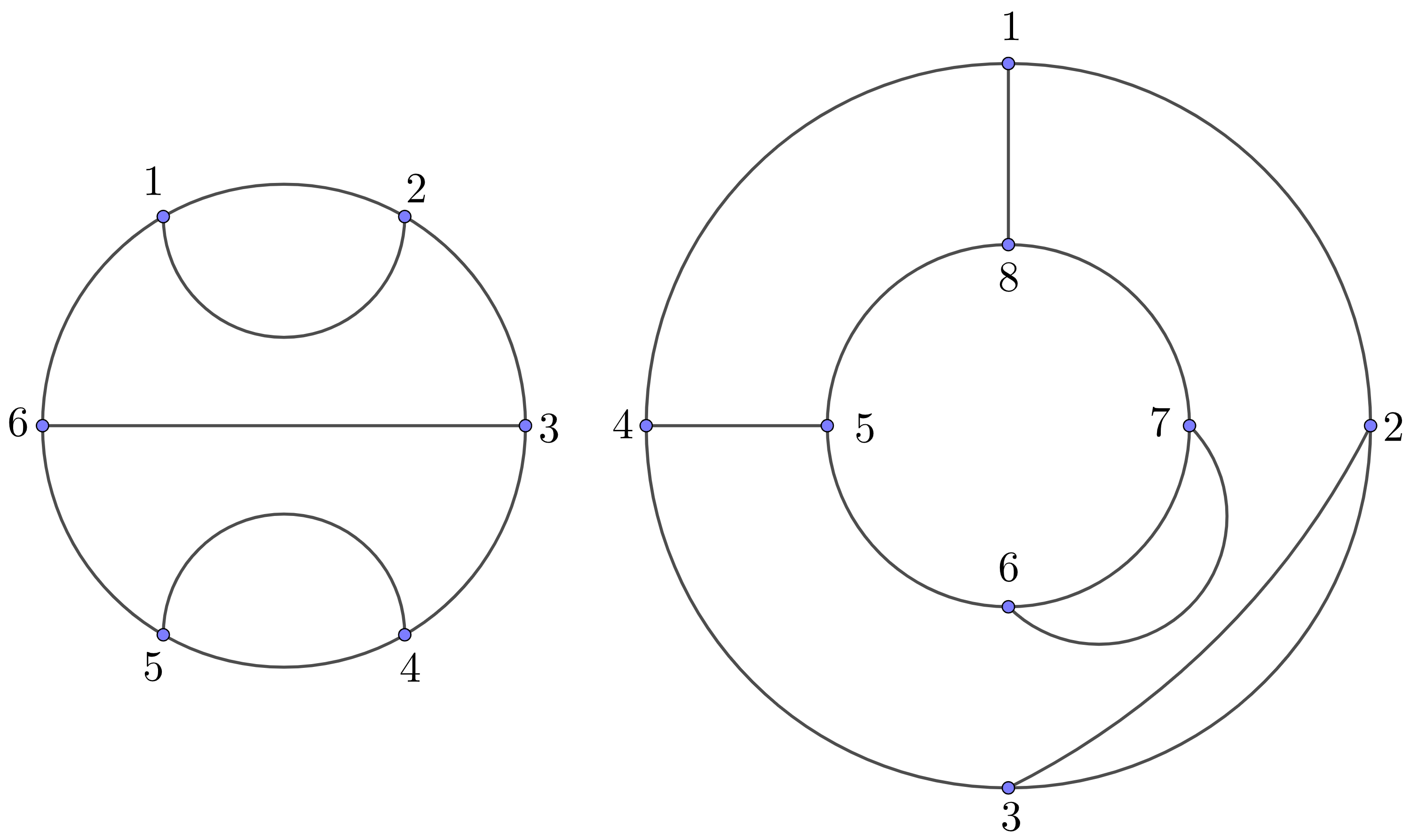}
    \includegraphics[width=0.3\linewidth]{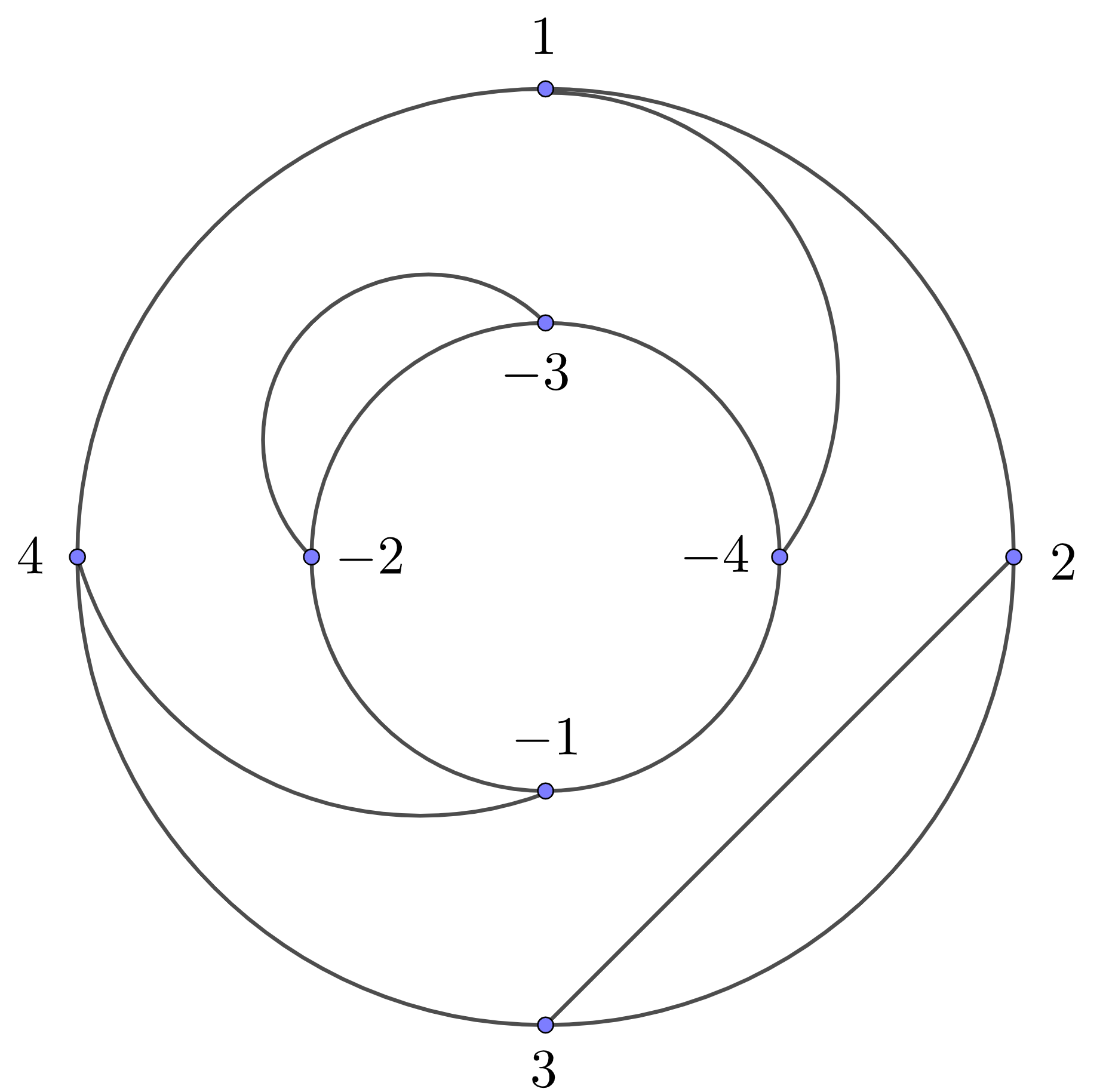}
    \caption{The non-crossing disk and annular pairings with respect to $\gamma=(1,\dots,6)$ (left) and $\gamma=(1,2,3,4)(5,6,7,8)$ (middle) of Example \ref{Example:non crossing pairings} with the symmetric non-crossing annular pairing $\pi=(1,-4)(4,-1)(2,3)(-2,-3)\in\NC_2^\delta(4,-4)$ (right). The cycles $(1,8)$, $(4,5)$, $(1,-4)$, and $(4,-1)$ connecting the inner and outer circles of the annuli are called \textit{through strings}.}
    \label{fig:noncrossingpairings}
\end{figure}

\section{Non-Crossing Annular Pairings and Moments of Gaussian Ensembles} \label{s2}

In propositions \ref{prop1} and \ref{prop2} of \S\ref{s1.1}, we saw that the moments $m_n=\E\Tr H^n$ of the GUE and GOE are given by genus expansions whose coefficients are enumerations of particular ribbon graphs. In this section, we show that these ribbon graphs are in bijection with certain classes of non-crossing annular pairings. This is done most clearly through the language of permutations. Thus, let us slightly abuse notation and redefine the sets of ribbon graphs $a_g(n)$ and $b_k(n)$ of propositions \ref{prop1} and \ref{prop2} as the combinatorial maps that they canonically biject to.

\begin{definition}[Combinatorial maps] \label{def6}
Following \cite{Tut84}, \cite[\S17.10]{GR01}, a \textit{combinatorial map} is a quadruple $(E_Q,\tau_0,\tau_1,\tau_2)$ where
\begin{enumerate}
\item $E_Q:=\pm[n]=\{-n,\ldots,-1,1,\ldots,n\}$ and $n$ is even,
\item $\tau_0,\tau_1,\tau_2$ are fixed-point free involutions on $E_Q$,
\item $\tau_0\tau_1=\tau_1\tau_0$ and $\tau_0\tau_1$ is also fixed-point free,
\item the group $\langle\tau_0,\tau_1,\tau_2\rangle$ generated by $\tau_0,\tau_1,\tau_2$ is transitive, meaning that it has one orbit, that being all of $E_Q$.
\end{enumerate}
The elements of $E_Q$ are called \textit{quarter-edges}. The orbits of $\tau_0$, $\langle\tau_0,\tau_1\rangle$, $\langle\tau_0,\tau_2\rangle$, and $\langle\tau_1,\tau_2\rangle$ are respectively referred to as the \textit{half-edges}, \textit{edges}, \textit{vertices}, and \textit{faces} of the combinatorial map.
\end{definition}

\begin{example}
The ribbon graph of Figure \ref{fig6} below corresponds to, and by our abuse of notation is identified with, the combinatorial map $(E_Q,\tau_0,\tau_1,\tau_2)$ with
\begin{equation*}
\begin{array}{ll}E_Q=\{1,-1,2,-2,3,-3,4,-4\},&\tau_0=(1,-1)(2,-2)(3,-3)(4,-4),
\\ \tau_1=(1,-2)(-1,2)(3,-4)(-3,4),&\tau_2=(1,-4)(-1,2)(-2,3)(-3,4).\end{array}
\end{equation*} 
\end{example}
\begin{figure}[H]
        \centering
        \includegraphics[width=0.3\textwidth]{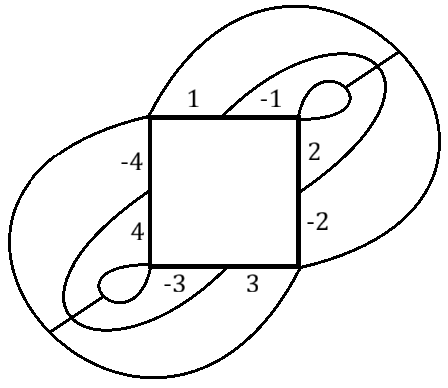}
        \caption{We divide the ribbons of a ribbon graph into quarters that are labelled $1,-1,2,-2,3,-3,4,-4$ in clockwise order. This yields a canonical combinatorial description of the ribbon graph.} \label{fig6}
\end{figure}

As we are studying unmixed, one-point moments, our ribbon graphs are built from exactly one polygon, so one says there is one vertex, according to the above terminology. Then, the canonical combinatorial description of a ribbon graph has
\begin{align}
\tau_0&=(1,-1)(2,-2)\cdots(n,-n), \label{tau0}
\\\tau_2&=(-1,2)(-2,3)\cdots(-n,1). \label{tau2}
\end{align}
Moreover, in the orientable case, we have for all $1\le a,b\le n$ that $\tau_1(a)<0$ and that $\tau_1(b)=-a$ whenever $\tau_1(a)=-b$. Thus, the orientable case can be simplified by removing the redundancy of $\tau_0$ and simply considering pairings $\pi\in\cP_2(n)$.

\begin{definition}[Ribbon graphs]\label{def7}
Let $n\in\N$ be even. The set of genus $g$ orientable ribbon graphs with $n$ half-edges introduced in Proposition \ref{prop1} is
\begin{equation*}
a_g(n)=\left\{\pi\in\cP_2(n)\,\middle\vert\,\#(\pi^{-1}1_n)=n/2+1-2g\right\},
\end{equation*}
where we recall that $1_n=(1,2,\ldots,n)$.

The set of non-orientable ribbon graphs with $n$ half-edges and Euler genus $k$ introduced in Proposition \ref{prop2} is
\begin{equation*}
b_k(n)=\left\{\tau_1\in\cP_2(\pm n)\,\middle\vert\,\begin{array}{l}\#(\tau_2\tau_1)=n+2(1-k),\,\tau_0\tau_1=\tau_1\tau_0,\\\tau_0\tau_1\textrm{ is fixed-point free, and there}\\\textrm{exists }a\in[n]\textrm{ such that }\tau_1(a)\in[n]\end{array}\right\},
\end{equation*}
where $\cP_2(\pm n)$ is the set of pairings of $\pm[n]$ and $\tau_0,\tau_2$ are as in equations \eqref{tau0}, \eqref{tau2}.
\end{definition}

\begin{remark}
The conditions on $\#(\pi^{-1}1_n)$ and $\#(\tau_2\tau_1)$ in the above definition derive from the fact that the Euler characteristic of a ribbon graph is given by $\chi=F-E+V$, where $F=1$ counts the unique face of the $n$-gon at the centre of the ribbon graph, $E=n/2$ is the number of edges of the $n$-gon after pairwise identification, and $V$ is the number of boundaries of the ribbon graph, which are identify the vertices of the $n$-gon. In the orientable case, $V=\#(\pi^{-1}1_n)$ and the Euler characteristic relates to the genus according to $\chi=2-2g$. In the non-orientable case, each boundary of the ribbon graph is described by \textit{two} cycles of $\tau_2\tau_1$, accounting for the potential clockwise and anticlockwise local choices of orientation. Moreover, the Euler genus $k$ and characteristic $\chi$ are related by $\chi=2-k$.

The other conditions for $b_k(n)$ descend from the third condition of Definition \ref{def6} (the remaining conditions are satisfied due to $\tau_1$ being a pairing and our choice of $\tau_0,\tau_2$) and the requirement that a ribbon graph must have at least one twisted ribbon to be non-orientable.
\end{remark}

In \S\ref{s2.1} below, we prove the bijection
\begin{equation*}
b_1(n)\cong\NC_2^\delta(n,-n),
\end{equation*}
while in \S\ref{s2.2} and \S\ref{s2.3}, we derive analogous sets of non-crossing annular pairings that are in bijection with $a_1(n)$ and $b_2(n)$, respectively. We recall that these ribbon graphs are fundamental in the sense that, by the classification theorem for closed surfaces, every other non-planar ribbon graph contributing to $m_n^{(\mathrm{GUE})}$ and $m_n^{(\mathrm{GOE})}$ can be constructed by gluing a finite number of torii to a ribbon graph of $a_1(n),b_1(n)$, or $b_2(n)$. The general strategy in the following is to use topological arguments to guide our intuition before giving algebraic proofs involving the various subsets of $\PS_n$ and $\PS_{\pm n}$ discussed thus far.

\subsection{Order \texorpdfstring{$1/N$}{1/N} Corrections to the GOE Moments} \label{s2.1}

Let $H$ be an $N\times N$ GOE matrix. As discussed in Proposition \ref{prop1}, the moment $m_n^{(\mathrm{GOE})}=\E\Tr H^n$ ($n\in\N$ even) is a polynomial of the form
\begin{equation} \label{mngoe1}
m_n^{(\mathrm{GOE})}=\frac{|a_0(n)|}{2^n}N^{n/2+1}+\frac{|b_1(n)|}{2^n}N^{n/2}+\mathrm{O}(N^{n/2-1}).
\end{equation}
In other words, the normalised trace of the scaled matrix $\widetilde{H}=\frac{2}{\sqrt{N}}H$ has terminating $1/N$ expansion of the form
\begin{equation} \label{mngoe2}
\frac{1}{N}\E\Tr\widetilde{H}^n=|a_0(n)|+\frac{|b_1(n)|}{N}+\mathrm{O}(N^{-2}).
\end{equation}
The set $a_0(n)$ is simply the set of planar ribbon graphs, which is counted by the Catalan numbers. We study the correction to this leading order term, which is equal to the number of non-orientable ribbon graphs of $n$ half-edges and Euler genus $1$. This correction term has also been explored from the side of free probability theory in \cite{mingo2019}, where it was proved that the coefficient of interest counts the symmetric non-crossing annular pairings of $\NC_2^\delta(n,-n)$. We now unify these theories by showing how the relevant ribbon graphs can be seen as non-crossing annular pairings. 

We begin with a well known lemma that is often used as an alternative definition of the (Euler) genus of a ribbon graph.
\begin{lemma}[Ribbon graphs and surfaces] \label{lemma1}
A (non-)orientable ribbon graph is of (Euler) genus $k$ ($g$) exactly when it can be embedded into a (non-)orientable surface of equal (Euler) genus without any crossings or twists and excising the ribbon graph from said surface results in a finite collection of disks.
\end{lemma}
Thus, $b_1(n)$ is the set of ribbon graphs with $n$ half-edges that can be drawn on the real projective plane without any crossings or twists. For tractability, we draw our ribbon graphs on the fundamental polygon of the real projective plane, which is a disk with its boundary identified antipodally. Since the ribbons of our drawing has no crossings or twists, we are able to shrink our ribbons down to edges, removing all of the negative labels. Note that at least one edge must cross the antipodally-identified boundary of the fundamental polygon (an odd number of times to avoid trivial crossings), otherwise excising the ribbon graph from the real projective plane would not result in a collection of disks; see Figure \ref{fig7a} below.

\begin{figure}[H]
    \centering
    \begin{subfigure}{0.49\textwidth}
        \centering
        \captionsetup{justification=centering}
        \includegraphics[width=0.69\textwidth]{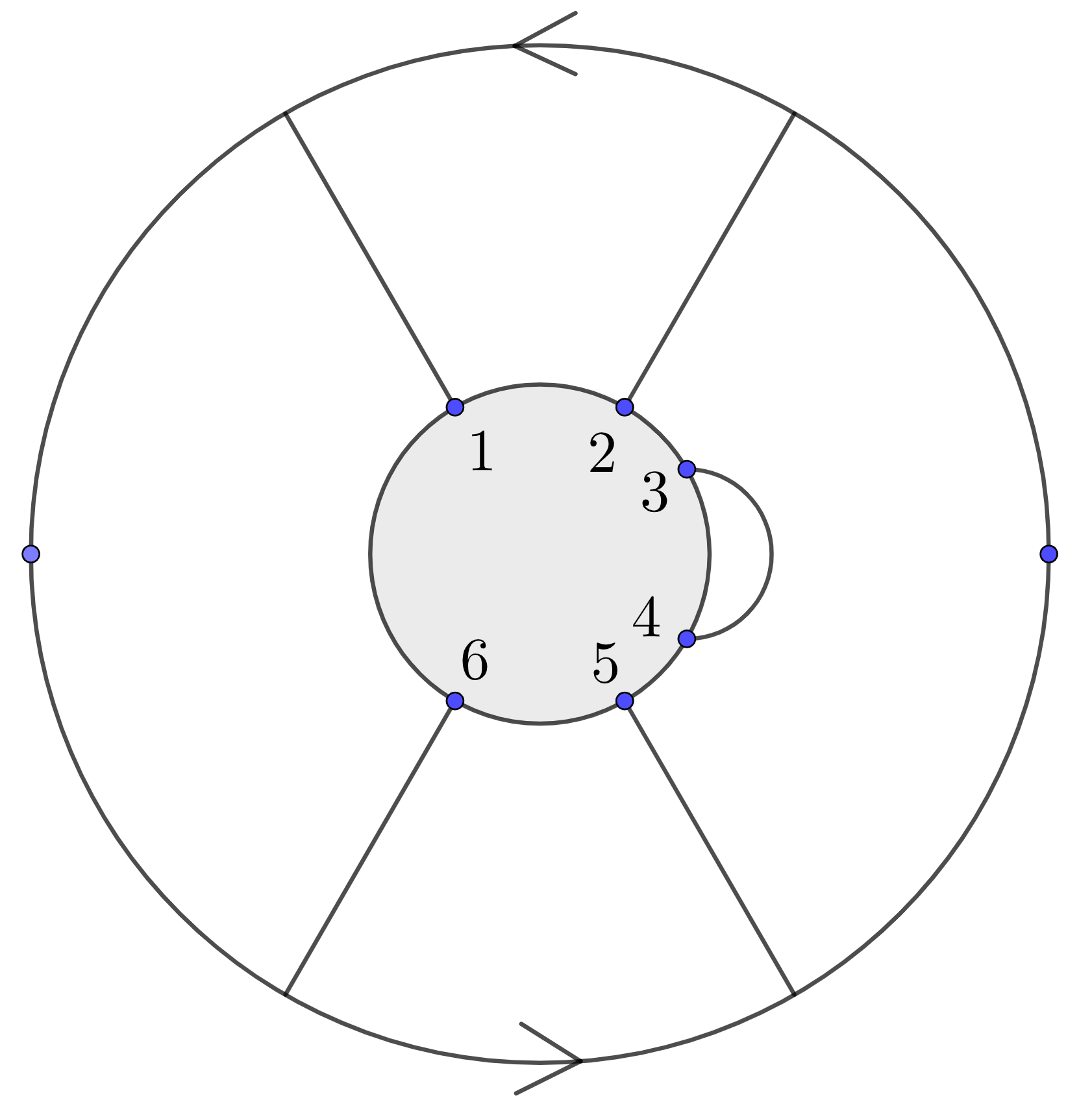}
        \caption{Graph given by the restriction $\tau_1\vert_{E_Q/\tau_0}$.} \label{fig7a}
    \end{subfigure}\hfill
    \begin{subfigure}{0.49\textwidth}
        \centering
        \captionsetup{justification=centering}
        \includegraphics[width=0.69\textwidth]{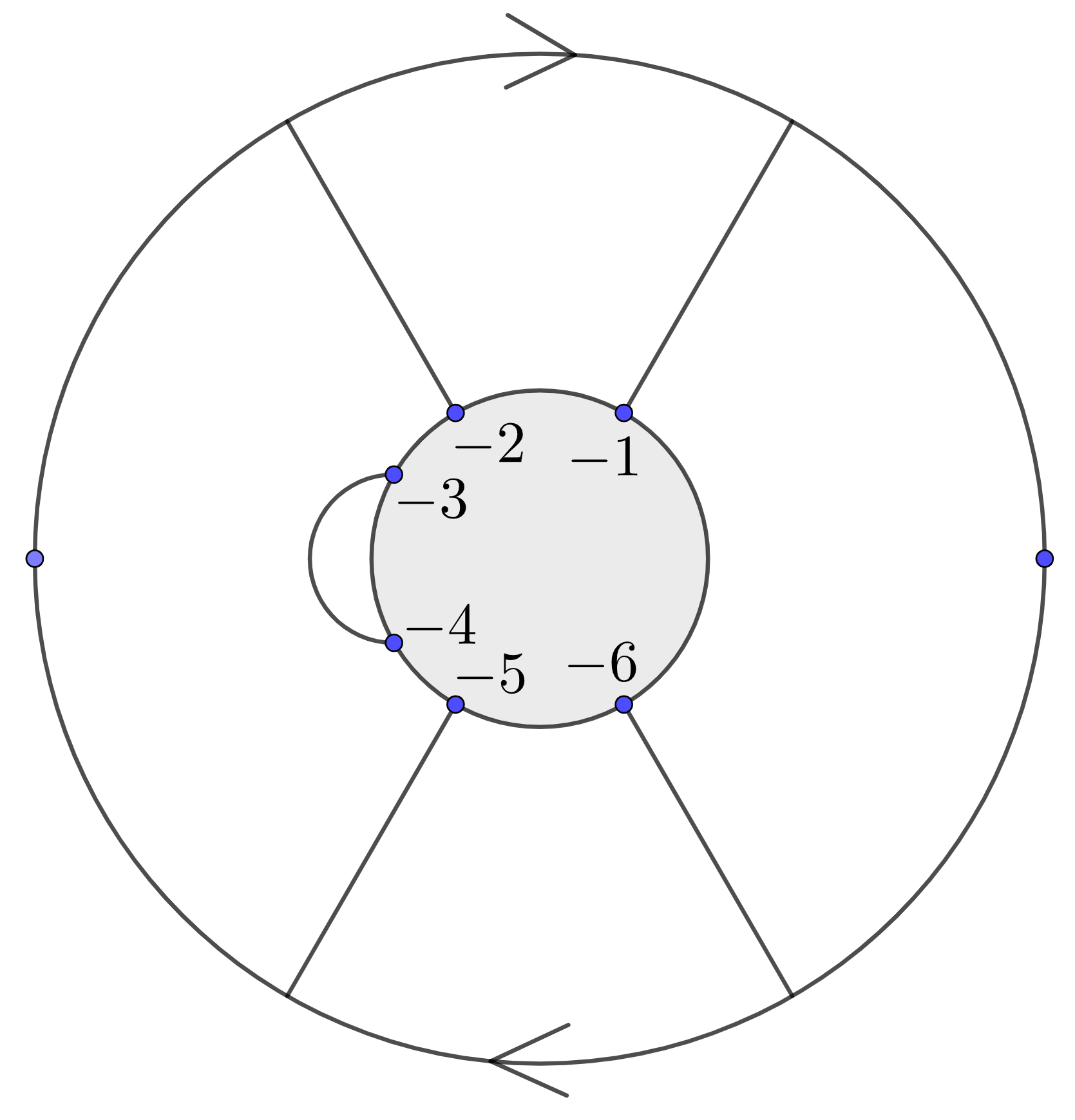}
        \caption{Inverted copy of the graph.} \label{fig7b}
    \end{subfigure}

    \begin{subfigure}{0.49\textwidth}
        \centering
        \captionsetup{justification=centering}
        \includegraphics[width=0.69\textwidth]{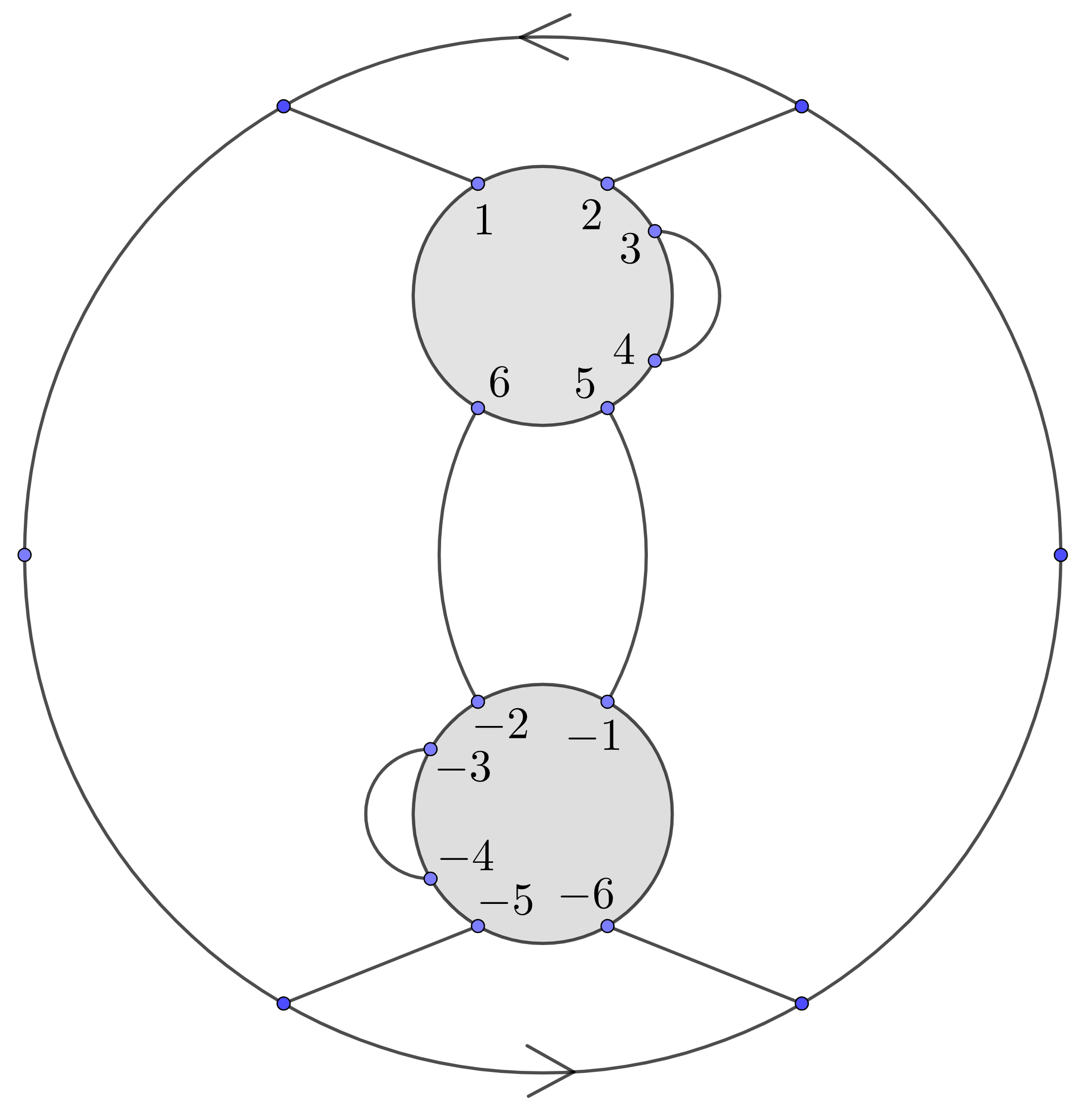}
        \caption{Two-vertex graph on a sphere.} \label{fig7c}
    \end{subfigure}\hfill
    \begin{subfigure}{0.49\textwidth}
        \centering
        \captionsetup{justification=centering}
        \includegraphics[width=0.69\textwidth]{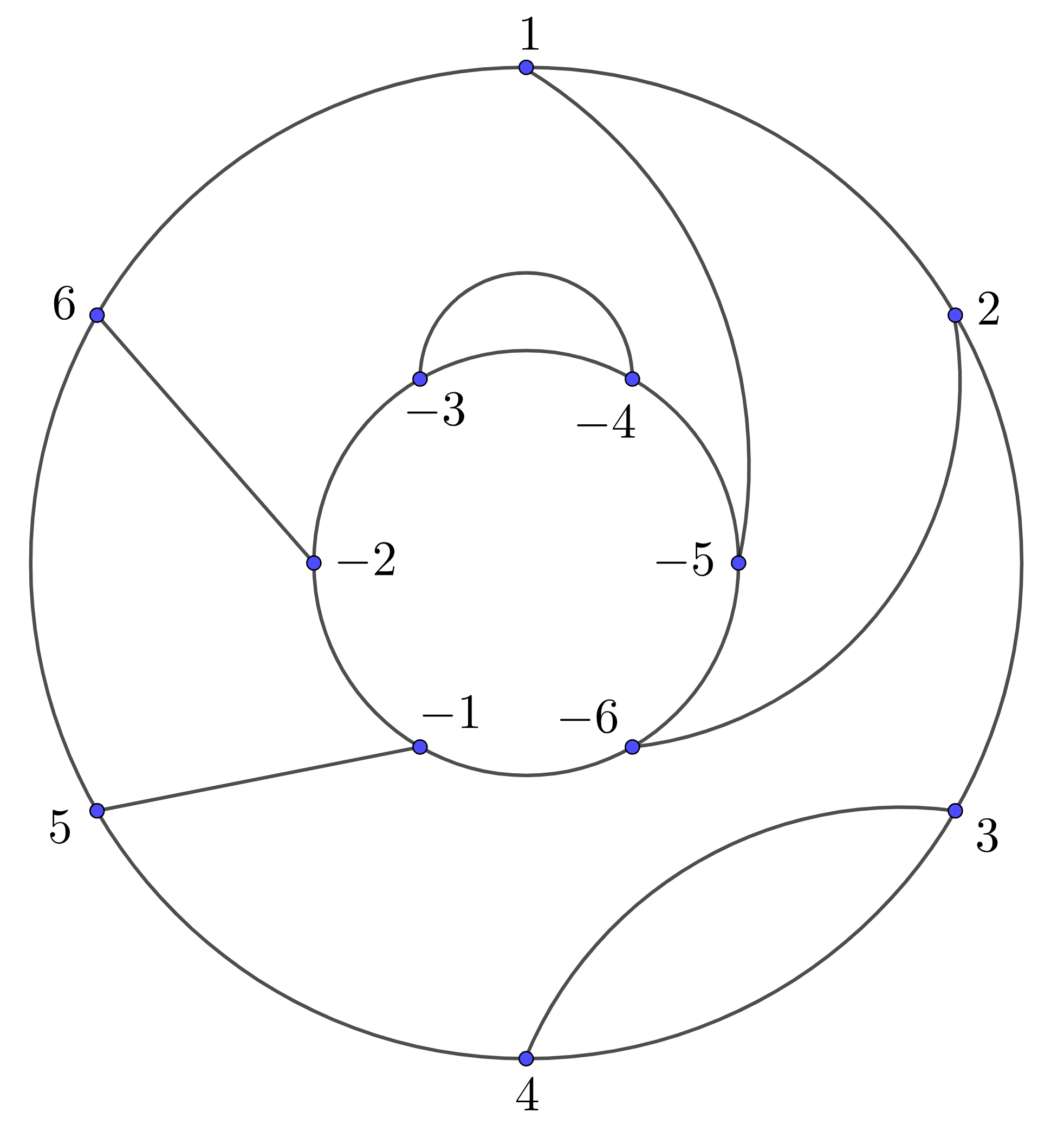}
        \caption{The non-crossing annular pairing.} \label{fig7d}
    \end{subfigure}
    \caption{Given $\tau_1\in b_1(6)$, consider its restriction to the positively-labelled half-edges obtained by quotienting the set of quarter-edges $\pm[6]$ by the action of $\tau_0$; this is (a). To obtain a symmetric non-crossing annular pairing from this graph, proceed by drawing an inverted (reflected) copy of it with labels replaced by their negatives, (b), then glue the bottom half of the boundary of (a) to the top half of the boundary of (b) to produce (c), a two-vertex ribbon graph on a sphere. Gluing the boundary together and excising said vertices gives the symmetric non-crossing annular pairing (d).} \label{fig7}
\end{figure}

Now, to describe our bijection, we first recall that the sphere is a double cover of the real projective plane, where the double covering map projects antipodal points on the sphere to the same point on the real projective plane. Thus, the real projective plane can be visualised as the northern hemisphere of the sphere with the equator identified antipodally. Taking the reverse viewpoint, gluing the fundamental polygon of the real projective plane to a copy of itself along half of the boundary produces its canonical double cover. 

Hence, given a ribbon graph drawn on the fundamental polygon of a real projective plane, the algorithm to produce the desired symmetric non-crossing annular pairing is as follows:
\begin{enumerate}
\item Draw an inverted (i.e., reflected along some axis) copy of the ribbon graph with all labels replaced by their negatives in order to distinguish the two copies.
\item Glue the original ribbon graph and its inverted copy together along half of their boundaries. The result is a graph with two vertices on a sphere.
\item Removing the vertices and labelling the resulting boundaries by the half-edges connecting to them results in the desired annular pairing.
\end{enumerate}

For the inverse mapping, one simply glues disks acting as ribbon graph vertices to the boundaries of the annulus to form a sphere, moves said vertices to the north and south poles of the sphere while ensuring that every edge is antipodal to another, cuts along the equator, discards the hemisphere with negative labels, and finally identifies the boundary of the remaining hemisphere in an antipodal way.

We give the combinatorial statement of our bijection.

\begin{proposition}[Bijection for $b_1(n)$]\label{prop5}
Let $n\in\N$ be even and recall the specifications of $\NC_2^\delta(n,-n)$ and $b_k(n)$ given in definitions \ref{def5} and \ref{def7}. Then, the mapping
\begin{align*}
\varphi_1:b_1(n)&\to\NC_2^\delta(n,-n),
\\\tau_1&\mapsto\tau_1\tau_0
\end{align*}
is a bijection, where we recall that $\tau_0=(1,-1)(2,-2)\cdots(n,-n)$.
\end{proposition}

\begin{proof}
We first show that the image of $\varphi_1$ indeed lies within $\NC_2^\delta(n,-n)$. Thus, we check the relevant conditions of definitions \ref{def4} and \ref{def5}:
\begin{enumerate}
\item $\tau_1\tau_0$ is a pairing because, by Definition \ref{def7}, it is fixed-point free and $\tau_0\tau_1=\tau_1\tau_0$ implies that $(\tau_1\tau_0)^2$ is the identity, so $\varphi_1(\tau_1)=\tau_1\tau_0$ is an involution.
\item There does not exist $r\in[n]$ such that $(-r,r)$ is a cycle of $\varphi_1(\tau_1)=\tau_1\tau_0$ because if there were such a cycle, $\tau_1$ would have $r$ and $-r$ as fixed points, which would be in contradiction with it being a pairing.
\item For every $r,s\in\pm[n]$ such that $(r,s)$ is a cycle of $\tau_1\tau_0$, we must have that
\begin{equation*}
\tau_1(-r)=\tau_1\tau_0(r)=s\,\textrm{ and }\tau_1(-s)=\tau_1\tau_0(s)=r,
\end{equation*}
so by Definition \ref{def7},
\begin{equation*}
\tau_1\tau_0(-r)=\tau_0\tau_1(-r)=-s\,\textrm{ and }\tau_1\tau_0(-s)=\tau_0\tau_1(-s)=-r.
\end{equation*}
Hence, $(-s,-r)$ is also a cycle of $\tau_1\tau_0$.
\item Note that $\widetilde{1}_n=\tau_2\tau_0$ and recall that, by the definition of $b_1(n)$, $\tau_1$ is such that $\#(\tau_2\tau_1)=n$. Then, as the conjugation and inverse of a permutation has the same number of cycles as the original permutation, we see that
\begin{equation*}
\#(\varphi_1(\tau_1)^{-1}\widetilde{1}_n)=\#(\tau_0\tau_1\tau_2\tau_0)=\#(\tau_1\tau_2)=\#(\tau_2\tau_1)=n.
\end{equation*}
As $\varphi_1(\tau_1)=\tau_1\tau_0$ is a pairing on $\pm[n]$, we thus have
\begin{equation*}
\#(\varphi_1(\tau_1))+\#(\varphi_1(\tau_1)^{-1}\widetilde{1}_n)+\#(\widetilde{1}_n)=2n+2.
\end{equation*}
\item As $\tau_1\in b_1(n)$, there exists $a\in[n]$ such that $\tau_1(a)\in[n]$, which we denote $b$. Then, $\tau_1\tau_0(-a)=b$, so $\varphi_1(\tau_1)=\tau_1\tau_0$ contains the cycle $(-a,b)$ for $a,b\in[n]$. Hence, $\varphi_1(\tau_1)\vee\widetilde{1}_n=1_{\pm[n]}$.
\end{enumerate}

Conditions (4) and (5) ensure that $\tau_1\tau_0$ is non-crossing with respect to $\widetilde{1}_n$, so is an element of $NC(\widetilde{1}_n)$. Then, conditions (2) and (3) place $\tau_1\tau_0$ in $\NC^\delta(n,-n)$ and, finally, condition (1) refines this to $\varphi_1(\tau_1)=\tau_1\tau_0\in\NC_2^\delta(n,-n)$.

Since $\tau_0$ is an involution, $\varphi_1$ is injective, so it remains to prove that it is also surjective. Thus, let $\pi\in\NC_2^\delta(n,-n)$ and check that $\varphi_1^{-1}(\pi)=\pi\tau_0$ satisfies the conditions of $b_1(n)$ given in Definition \ref{def7}:
\begin{enumerate}
\item Let $r\in\pm[n]$ and $s=\pi(r)$. Then, by Definition \ref{def5}, $s\ne r,-r$ and both $(r,s)$ and $(-s,-r)$ are cycles of $\pi$. Thus,
\begin{equation*}
\pi\tau_0(r)=\pi(-r)=-s\,\textrm{ and }\pi\tau_0(-s)=\pi(s)=r,
\end{equation*}
so $(r,-s)$ is a cycle of $\pi\tau_0$. As this holds for any $r\in\pm[n]$, $\pi\tau_0$ is a fixed-point free involution on $\pm[n]$ and $\varphi_1^{-1}(\pi)=\pi\tau_0\in\cP_2(\pm n)$.
\item Take $r\in\pm[n]$ and $s=\pi(r)$ as above. Then, both $(r,s)$ and $(-s,-r)$ are cycles of $\pi$ and we see that
\begin{equation*}
\tau_0\pi\tau_0(r)=\tau_0\pi(-r)=\tau_0(-s)=s.
\end{equation*}
As this holds for all $r\in\pm[n]$, we have $\tau_0\pi\tau_0=\pi$, so $\tau_0\varphi_1^{-1}(\pi)=\varphi_1^{-1}(\pi)\tau_0$.
\item By the above, $\tau_0\varphi_1^{-1}(\pi)=\tau_0\pi\tau_0=\pi$, which is a pairing, hence fixed-point free.
\item We recall that $\widetilde{1}_n=\tau_2\tau_0$ and that the conjugation and inverse of a permutation has the same number of cycles as the original permutation. Thus,
\begin{equation*}
\#(\tau_2\varphi_1^{-1}(\pi))=\#(\tau_0\tau_2\pi)=\#(\pi^{-1}\tau_2\tau_0)=\#(\pi^{-1}\widetilde{1}_n)=n,
\end{equation*}
where the last equality follows from Definition \ref{def4}.
\item As $\pi\vee\widetilde{1}_n=1_{\pm[n]}$, there exists $a\in[n]$ such that $\pi(-a)\in[n]$, which we denote $b$. Then,
\begin{equation*}
\varphi_1^{-1}(\pi)(a)=\pi\tau_0(a)=\pi(-a)=b\in[n].
\end{equation*}
\end{enumerate}
Having shown that $\varphi_1$ is injective and surjective, we are done.
\end{proof}

\subsection{Order \texorpdfstring{$1/N^2$}{1/N2} Corrections to the GUE Moments} \label{s2.2}

We now consider $H$ an $N\times N$ GUE matrix. Going by Proposition \ref{prop2}, we have for even $n\in\N$ the polynomial expression
\begin{equation*}
m_n^{(\mathrm{GUE})}=\E\Tr H^n=\frac{|a_0(n)|}{2^{n/2}}N^{n/2+1}+\frac{|a_1(n)|}{2^{n/2}}N^{n/2-1}+\mathrm{O}(N^{n/2-3}),
\end{equation*}
equivalently the terminating $1/N$ expansion for $\widetilde{H}=\sqrt{\frac{2}{N}}H$,
\begin{equation*}
\frac{1}{N}\E\Tr\widetilde{H}^n=|a_0(n)|+\frac{|a_1(n)|}{N^2}+\mathrm{O}(N^{-4}).
\end{equation*}
At leading order, we have the Catalan numbers, in exact agreement with the equivalent moments of the GOE. However, the next to leading order correction is now of order $1/N^2$ rather than $1/N$. In the spirit of \cite{mingo2019}, we now define a new set of non-crossing annular pairings and then prove that they are in bijection with $a_1(n)$, thus characterising said correction.

\begin{definition}[Toroidal annular pairings] \label{def8}
Let $n\in\N$ be even and $1\le u<v< n$. Recalling Definition \ref{def5}, we define
\begin{align}
\NC_{2,u,v}^{\rm T}(n)&:=\left\{\pi\in\NC_2\left(1_n(u-1,v)\right)\,\middle\vert\,\begin{array}{l}(u,v)\in\pi\textrm{ and for all }a\in[u-1],\\\pi(a)\in[n]\setminus\{u,\ldots,v\}\end{array}\right\},
\\\NC_2^{\rm T}(n)&:=\bigcup_{1\le u<v< n}\NC_{2,u,v}^{\rm T}(n),
\end{align}
where we say $(0,v)\equiv(n,v)$ and $[0]\equiv\emptyset$.
\end{definition}
As $1_n(u-1,v)=(u,\ldots,v)(1,\ldots,u-1,v+1,\ldots,n)$, we may see $\NC_{2,u,v}^{\rm T}(n)$ as the set of non-crossing pairings on an annulus with the outer circle labelled clockwise by $u,\ldots,v$ and the inner circle labelled $1,\ldots,u-1,v+1,\ldots,n$ in an anticlockwise fashion. The second condition in the definition of $\NC_{2,u,v}^{\rm T}(n)$ simply means that there are no \textit{through strings} connecting the labels $1,\ldots,u-1$ on the inner circle of the annulus to any of the labels on the outer circle; see Figure \ref{fig8c} for an example.

The main result of this subsection is that, as subsets of $\cP_2(n)$, $a_1(n)$ and $\NC_2^{\rm T}(n)$ are \textit{equal} --- it is only a matter of interpretation of the pairings. To see this, we first recall from Lemma \ref{lemma1} that the genus one orientable ribbon graphs of $a_1(n)$ can be embedded into a torus without any crossings. In terms of the fundamental polygon of the torus, the prescription then is to represent the torus as a square with opposite sides identified without any twists and then draw the ribbon graph on this square, exploiting traversal through the boundaries to avoid crossings. Note that since the ribbon graphs of interest are of genus one, at least two edges of the ribbon graph must traverse the boundaries of the fundamental polygon of the torus in a non-trivial way --- one through the top-bottom boundary and the other through the left-right boundary. Moreover, letting $u\in[n]$ be the minimal label involved in such an edge, we are able to draw the edge $(u,\pi(u))$ as a vertical line (see Figure \ref{fig8a}). Note that this $u$ coincides with the minimal element of $[n]$ such that $(u,\pi(u))$ crosses with another cycle, i.e., there exists some $u'\in[n]$ such that $u<u'<\pi(u)<\pi(u')$.

Now, the key idea behind our mapping from $a_1(n)$ to $\NC_2^{\rm T}(n)$ is that cutting a torus along a non-trivial cycle produces a cylinder, equivalently annulus, and that the edge $(u,\pi(u))$ described above identifies such a cycle. Thus, we begin by drawing a toroidal ribbon graph as described above, then cut vertically immediately on the left of the $(u,\pi(u))$ edge, splitting the fundamental polygon of the torus and indeed the vertex or $n$-gon of the ribbon graph into two. Then, we glue along the left and right sides of the fundamental polygon of the torus so that the cut now serves as the new left and right sides. Identifying the top and bottom then forms the desired non-crossing annular pairing; see Figure \ref{fig8} for an example.

\begin{figure}
    \centering
    \begin{subfigure}{0.32\textwidth}
        \centering
        \captionsetup{justification=centering}
        \includegraphics[width=0.95\textwidth]{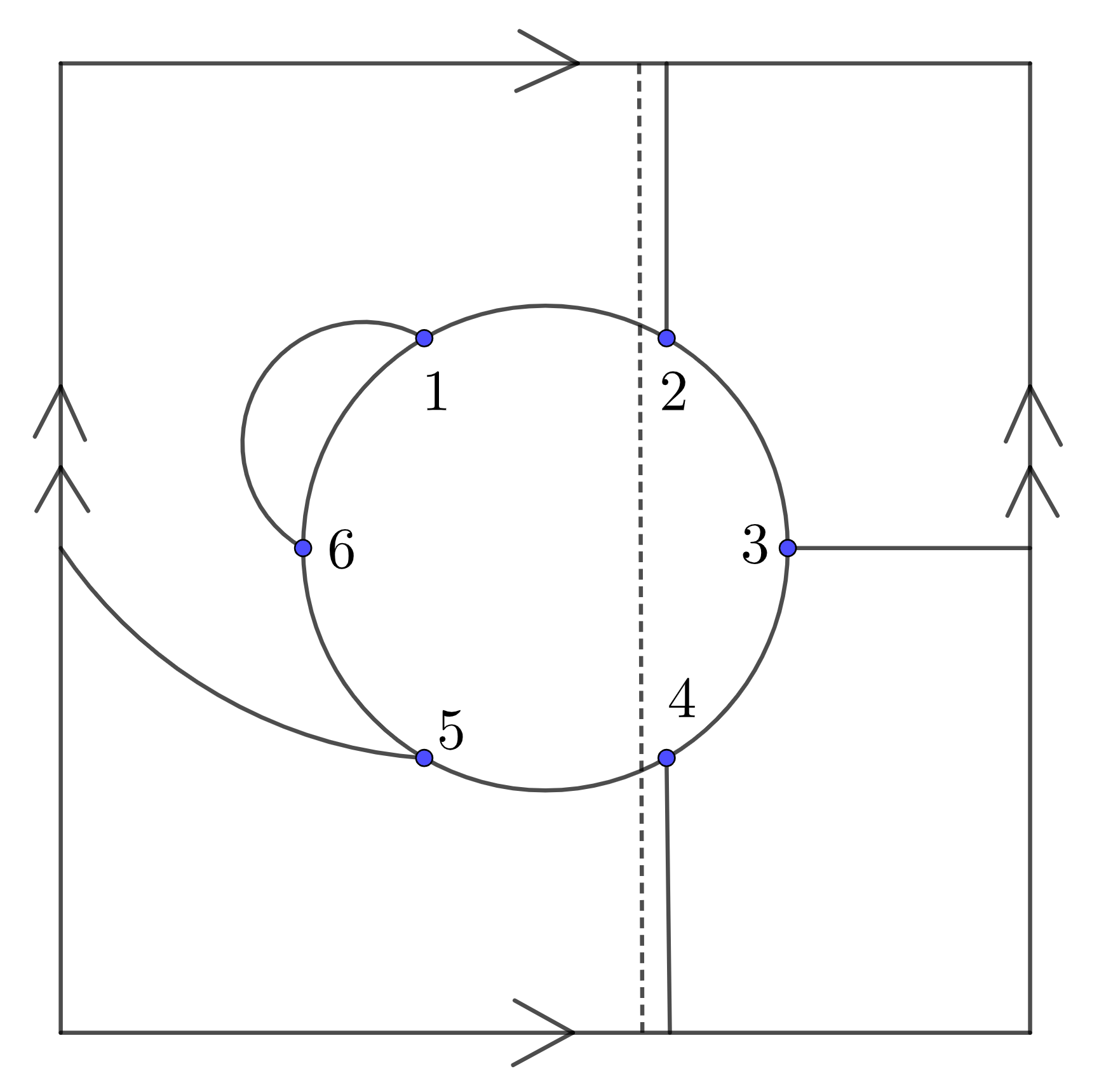}
        \caption{A ribbon graph on the torus with a cut along $(2,4)$.} \label{fig8a}
    \end{subfigure}\hfill
    \begin{subfigure}{0.32\textwidth}
        \centering
        \captionsetup{justification=centering}
        \includegraphics[width=0.95\textwidth]{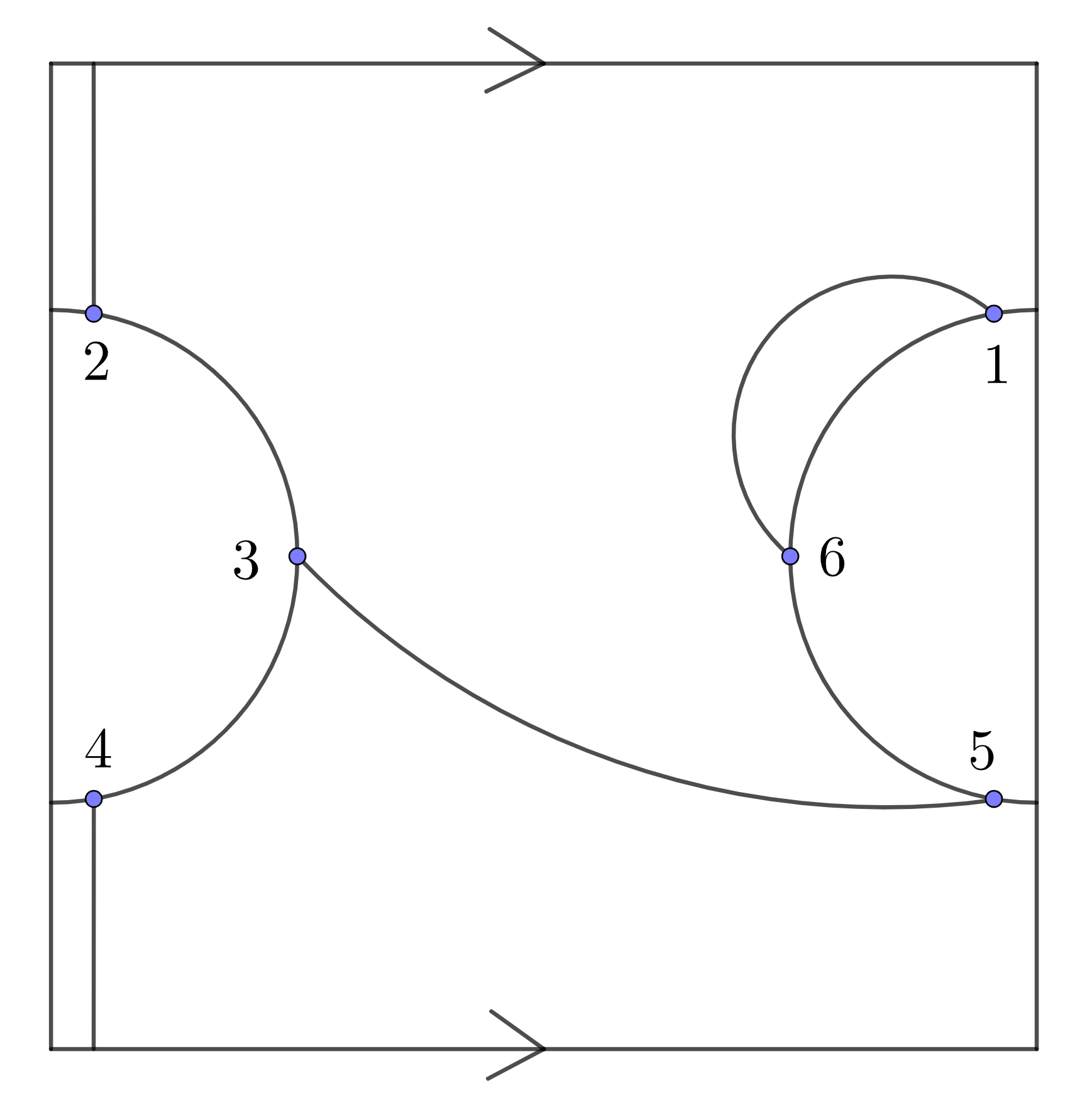}
        \caption{After gluing along the left-right boundary.} \label{fig8b}
    \end{subfigure}\hfill
    \begin{subfigure}{0.32\textwidth}
        \centering
        \captionsetup{justification=centering}
        \includegraphics[width=0.95\textwidth]{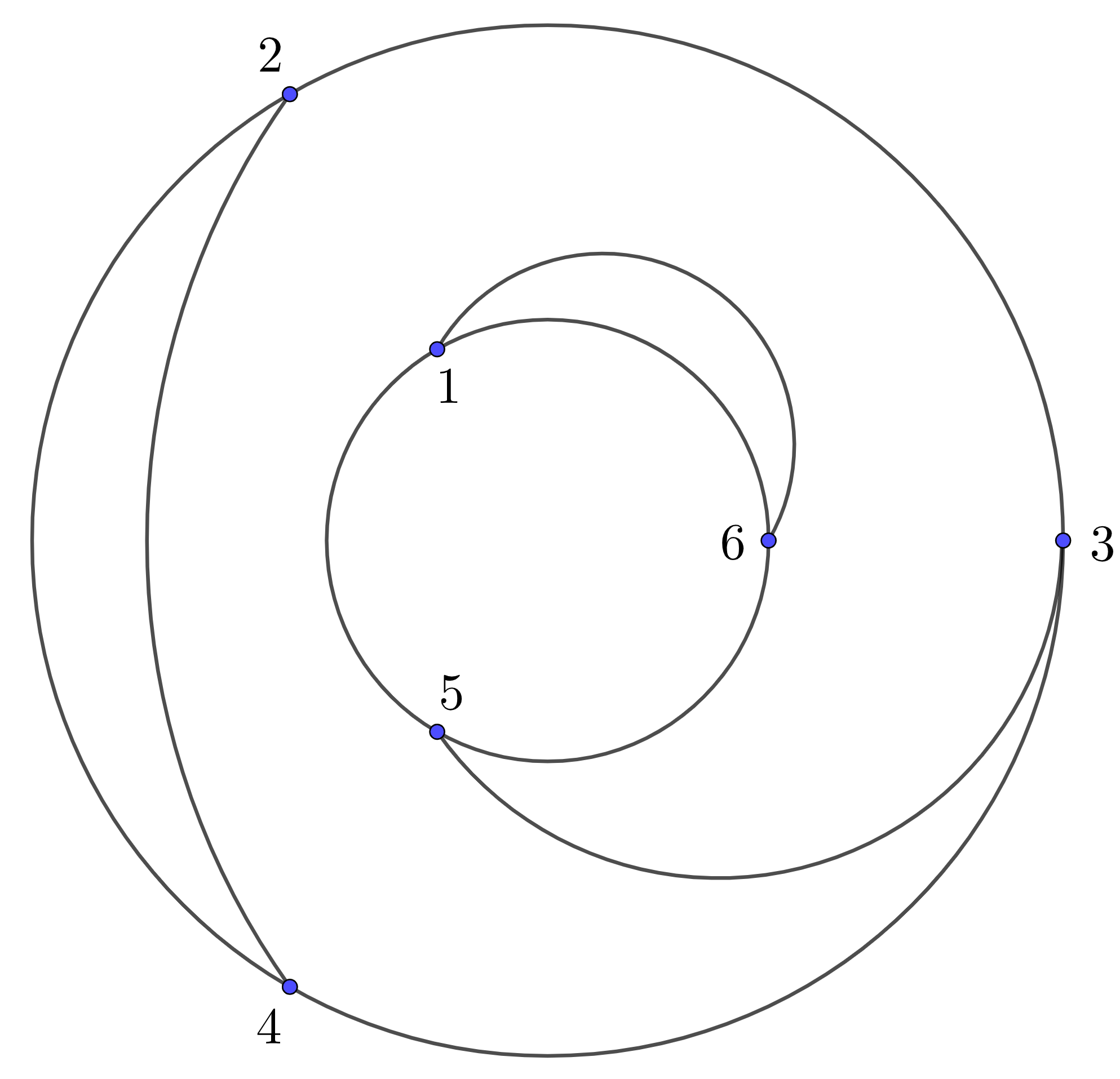}
        \caption{A non-crossing annular pairing in $\NC_{2,2,4}^{\rm T}(6)$.} \label{fig8c}
    \end{subfigure}
    \caption{Given $\tau_1\in a_1(6)$, we may draw it on the fundamental polygon of the torus as in (a), taking care to draw the edge $(2,4)$ vertically, as $2\in[6]$ is the minimal label involved in a crossing. Cutting along the left of the $(2,4)$ edge, as indicated by the dotted line, and then gluing along the left and right boundaries of the torus produces (b). Identifying the top and bottom then yields the non-crossing annular pairing (c).} \label{fig8}
\end{figure}

For the inverse mapping, we cut the annulus along a line connecting a point on the outer circle between $u$ and $v$, the minimal and maximal labels on the outer circle, with a point on the inner circle between $u-1$ and $v+1$. Then, identifying the edges of the resulting rectangle that were the boundaries of the original annulus produces a torus. Grouping the now-adjacent labels results in the desired ribbon graph.

We give the combinatorial proof of our mapping being a bijection.
\begin{proposition}[Bijection for $a_1(n)$] \label{prop6}
Let $n\in\N$ be even and recall the specifications of $a_1(n)$ and $\NC_2^{\rm T}(n)$ given in definitions \ref{def7} and \ref{def8}. Then, we have the equality of subsets of $\cP_2(n)$,
\begin{equation*}
a_1(n)=\NC_2^{\rm T}(n).
\end{equation*}
\end{proposition}

\begin{proof}
Let us first prove that $a_1(n)\subseteq\NC_2^{\rm T}(n)$. Thus, letting $\pi\in a_1(n)$, we observe the following:
\begin{enumerate}
    \item By the definition of $a_1(n)$,
    \begin{equation*}
    \#(\pi^{-1}1_n)=n/2-1\implies\#(\pi)+\#(\pi^{-1}1_n)=n-1.
    \end{equation*}
    Hence, by Definition \ref{def4}, $\pi$ is crossing with respect to $1_n$. Therefore, according to Remark \ref{rmk3}, there exist $1\le u<u'<v<v'\le n$ such that $(u,v),(u',v')\in\pi$. Let $u\in[n]$ to be minimal such that $(u,\pi(u))$ is crossing and let $u'\in[n]$ be such that $u<u'<\pi(u)<\pi(u')$. Let us also set $v=\pi(u)$ and $v'=\pi(u')$.
    \item Then, $u'\in \{u,\dots,v\}$ and $v'\in [n]\setminus \{u,\dots,v\}$, so $u',v'$ belong to distinct cycles of $1_n(u-1,v)$. Thus, $(u',v')$ is a through string of $\pi$ and $\pi \vee 1_n(u-1,v) =1_{[n]}$.
    \item As $\pi^{-1}1_n(u-1)=\pi^{-1}(u)=v$, $u-1$ and $v$ are in the same cycle, say $\sigma$, of $\pi^{-1}1_n$. It is straightforward to check that multiplying $\sigma$ by $(u-1,v)$ splits $\sigma$ into a product of two cycles. It follows that $\#(\pi^{-1}1_n(u-1,v))=\#(\pi^{-1}1_n)+1$, consequently (using again the definition of $a_1(n)$),
    \begin{equation} \label{equ:torusid}
    \#(\pi)+\#(\pi^{-1}1_n(u-1,v))=\#(\pi)+\#(\pi^{-1}1_n)+1=n.
    \end{equation}
    \item For all $a\in[u-1]$, $\pi(a)\in[n]\setminus\{u,\ldots,v\}$, as the existence of $a\in[u-1]$ with $\pi(a)\in\{u,\ldots,v\}$ would contradict the definition of $u$.
\end{enumerate}
Upon recalling that $1_n(u-1,v)$ contains two cycles, conditions (2) and (3) say that $\pi\in \NC_2(1_n(u-1,v))$, in accordance with definitions \ref{def4} and \ref{def5}, while condition (4) says that $\pi\in \NC^{\rm T}_{2,u,v}(n)$, hence $\pi\in \NC_2^{\rm T}(n)$. 

We now prove that $\NC_2^{\rm T}(n)\subseteq a_1(n)$. Thus, taking $\pi\in \NC_2^{\rm T}(n)$, we check the conditions of Definition \ref{def7} pertaining to $a_1(n)$:
\begin{enumerate}
    \item As $\pi\in\NC_2^{\rm T}(n)$, there exists $(u,v)\in\pi$ such that $\pi\in \NC^{\rm T}_{2,u,v}(n)$.
    \item Thus, we have that $\pi\in\NC_2(1_n(u-1,v))$, so by definitions \ref{def4} and \ref{def5},
    \begin{equation*}
        \#(\pi)+\#(\pi^{-1}1_n(u-1,v))=n,
    \end{equation*}
    where we recall that $1_n(u-1,v)$ has two cycles. By the logic of condition (3) above, we have that $\#(\pi^{-1}1_n(u-1,v))=\#(\pi^{-1}1_n)+1$, so
    \begin{align*}
    \#(\pi)+\#(\pi^{-1}1_n) & =\#(\pi)+\#(\pi^{-1}1_n(u-1,v))-1=n-1 \\
    & \implies\#(\pi^{-1}1_n)=n/2-1.
    \end{align*}
\end{enumerate}
These conditions are equivalent to the definition of $a_1(n)$, so we have shown that $a_1(n)\subseteq\NC_2^{\rm T}(n)\subseteq a_1(n)$ and we are done.
\end{proof}

\subsection{Order \texorpdfstring{$1/N^2$}{1/N2} Corrections to the GOE Moments} \label{s2.3}

Returning to the setting of \S\ref{s2.1} concerning the GOE, one has from Proposition \ref{prop2} that the refinement of the finite expansions \eqref{mngoe1} and \eqref{mngoe2} are respectively
\begin{align*}
m_n^{(\mathrm{GOE})}&=\frac{|a_0(n)|}{2^n}N^{n/2+1}+\frac{|b_1(n)|}{2^n}N^{n/2}
\\&\quad+\frac{|a_1(n)|+|b_2(n)|}{2^n}N^{n/2-1}+\mathrm{O}(N^{n/2-2}),
\\\frac{1}{N}\E\Tr\widetilde{H}^n&=|a_0(n)|+\frac{|b_1(n)|}{N}+\frac{|a_1(n)|+|b_2(n)|}{N^2}+\mathrm{O}(N^{-2}).
\end{align*}
As we have derived sets of non-crossing annular pairings that are in bijection with $a_1(n)$ and $b_1(n)$, we may now turn towards characterising the $1/N^2$ corrections to the GOE moments as the enumerations of non-cross annular pairings. It remains to construct a class of such pairings in bijection with $b_2(n)$, which is the set of ribbon graphs of $n$ half-edges that can be drawn on the Klein bottle without any crossings or twists. It turns out that this can be done by combining the ideas of the previous two subsections. Indeed, the key observation is that the torus is a double cover of the Klein bottle. Thus, let us define a new class of non-crossing annular pairings.

\begin{definition}[Klein bottle annular pairings]\label{def9}
Let $n\in\N$ be even and $1\le u<v<n$. Recalling that $\widetilde{1}_n=(1,\ldots,n)(-n,\ldots,-1)\in \PS_{\pm n}$, define
\begin{align*}
\widetilde{1}_{n,u,v}&:=\widetilde{1}_n(-u,v-1)(-v,u-1)
\\&\;=(u,1-v)(v,1-u)\widetilde{1}_n
\\&\;=(u,\ldots,v-1,1-u,\ldots,-1,-n,\ldots,-v)(v,\ldots,n,1,\ldots,u-1,1-v,\ldots,-u),
\end{align*}
where we write $(-v,0)\equiv(-v,n)$ and $(v,0)\equiv(v,-n)$. The analogues of $\NC_{2,u,v}^{\rm T}(n)$ and $\NC_2^{\rm T}(n)$ of Definition \ref{def8} are
\begin{align}
\NC_{2,u,v}^{\rm K}(n)&:=\left\{\pi\in\NC_2^\delta(\widetilde{1}_{n,u,v})\,\middle\vert\,\begin{array}{l}(u,-v)\in\pi\textrm{ and for all}\\a\in[u-1],\,\pi(a)\in[n]\end{array}\right\},
\\ \NC_2^{\rm K}(n)&:=\bigcup_{1\le u<v<n}\NC_{2,u,v}^{\rm K}(n),
\end{align}
where we retain the convention $[0]\equiv\emptyset$.
\end{definition}

We proceed by starting at Lemma \ref{lemma1}, as before. Thus, draw the fundamental polygon of the Klein bottle, which is a square with the left and right sides identified without a twist and the top and bottom sides identified with a twist. Then, given a ribbon graph $\tau_1$ of $b_2(n)$, we draw it on the fundamental polygon of the Klein bottle, ensuring that, for $u\in[n]$ minimal such that $\tau_1(u)\in[n]$, the edge $(u,\tau_1(u))$ is drawn vertically and there is at least one edge traversing the left and right sides of the fundamental polygon at hand. Then, to produce an element of $\NC_2^{\rm K}(n)$, we go through the following algorithm:
\begin{enumerate}
\item Draw an inverted copy of the ribbon graph with all labels replaced by their negatives.
\item Glue the bottom of the original ribbon graph to the top of its inverted copy in order to form its double cover, a ribbon graph on the torus.
\item Now, cut immediately on the left of the half-edges labelled $u,-\tau(u)$ and immediately on the right of the half-edges labelled $-u,\tau(u)$, with these cuts being connected via cuts through the vertices of the ribbon graph.
\item Glue the left and right sides of the torus together so that the cut of the previous step serves as the new left and right sides.
\item Identify the top and bottom to form an annulus with the outer and inner circles of the annulus inheriting labels from the ribbon graph vertices that we now remove.
\end{enumerate}

\begin{figure}
    \centering
    \begin{subfigure}{0.32\textwidth}
        \centering
        \captionsetup{justification=centering}
        \includegraphics[width=0.9\textwidth]{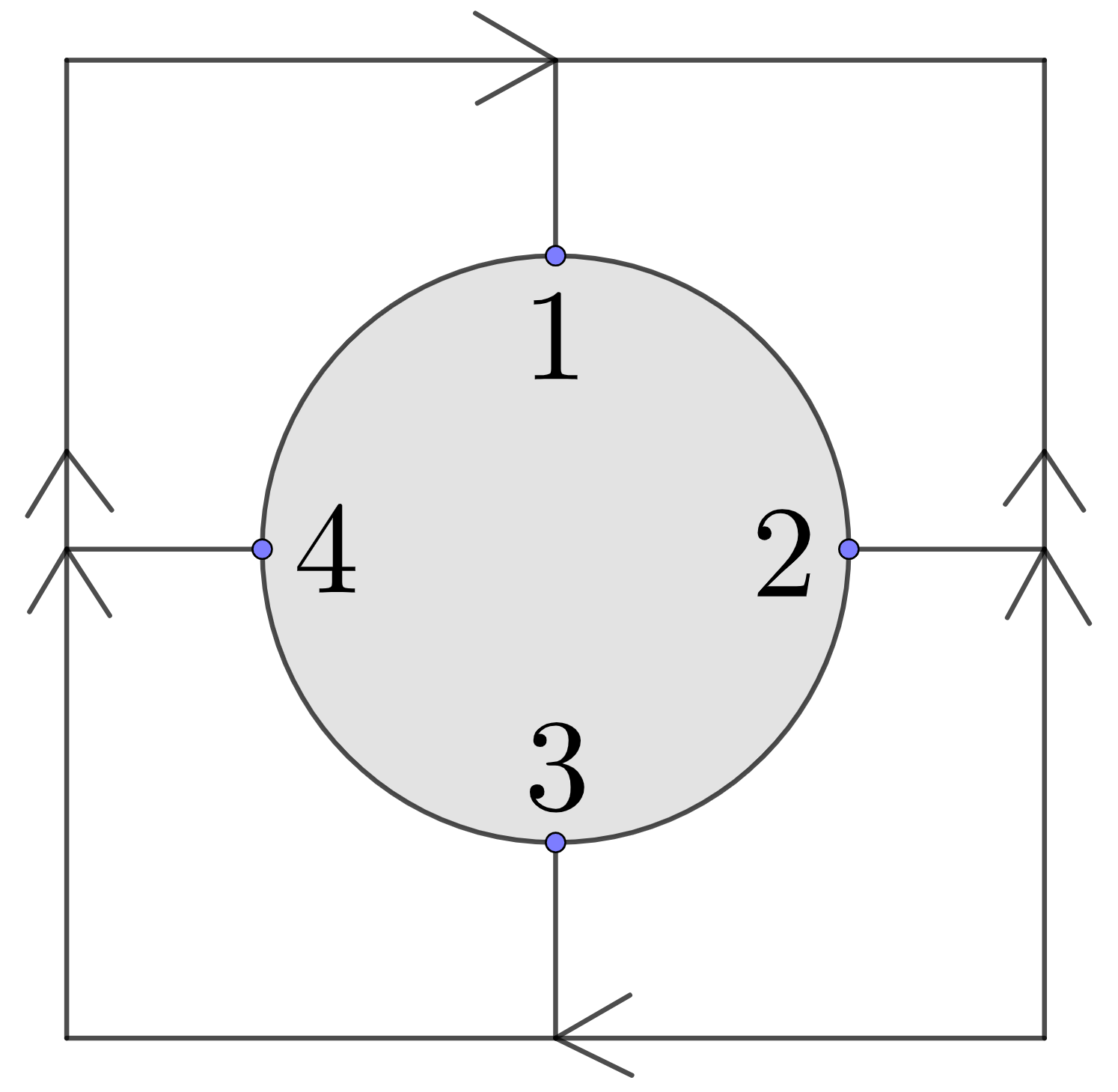}
        \caption{Klein bottle ribbon graph.} \label{fig9a}
    \end{subfigure}\hfill
    \begin{subfigure}{0.32\textwidth}
        \centering
        \captionsetup{justification=centering}
        \includegraphics[width=0.9\textwidth]{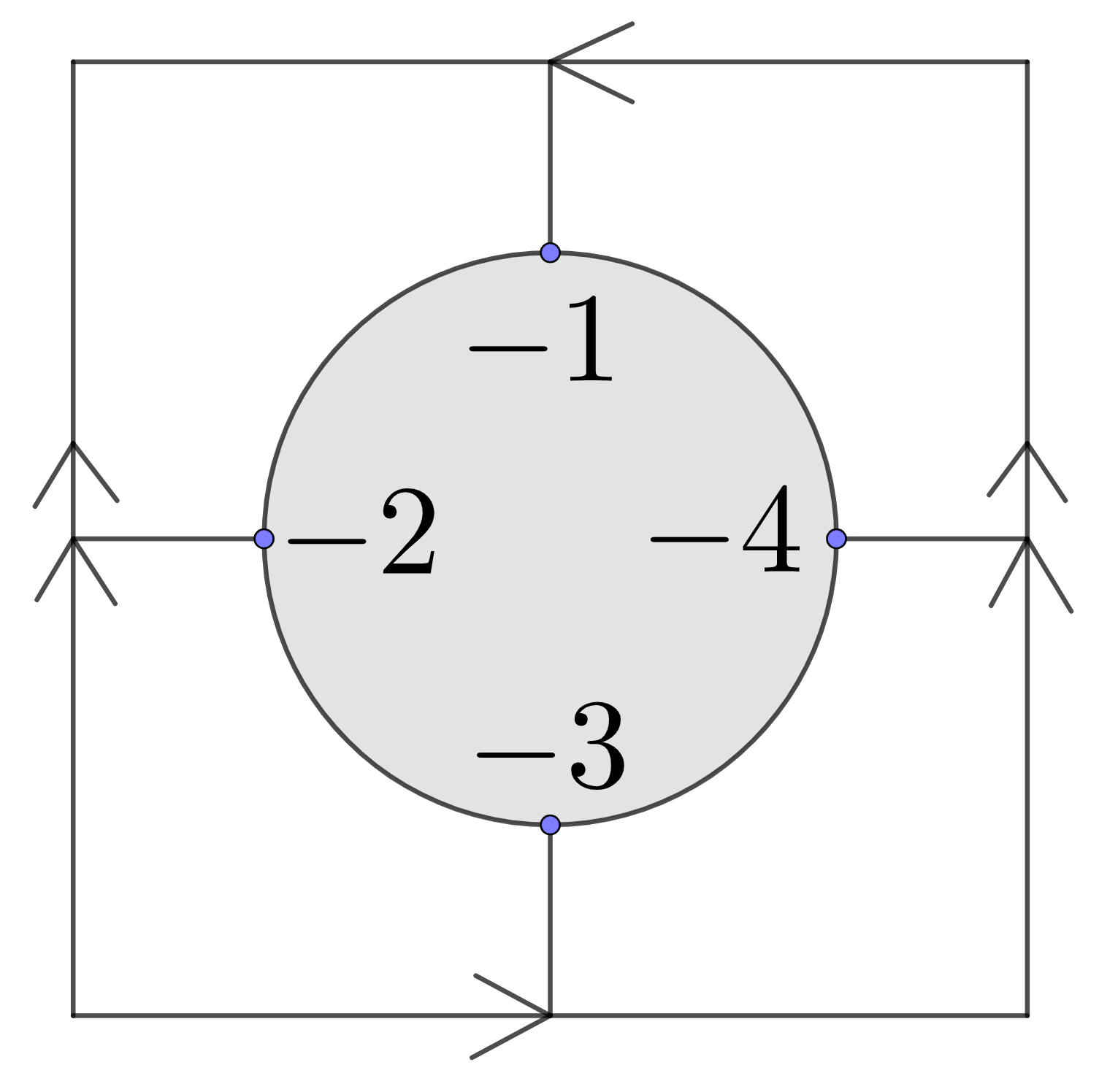}
        \caption{Inverted copy of the graph.} \label{fig9b}
    \end{subfigure}\hfill
    \begin{subfigure}{0.32\textwidth}
        \centering
        \captionsetup{justification=centering}
        \includegraphics[width=0.9\textwidth]{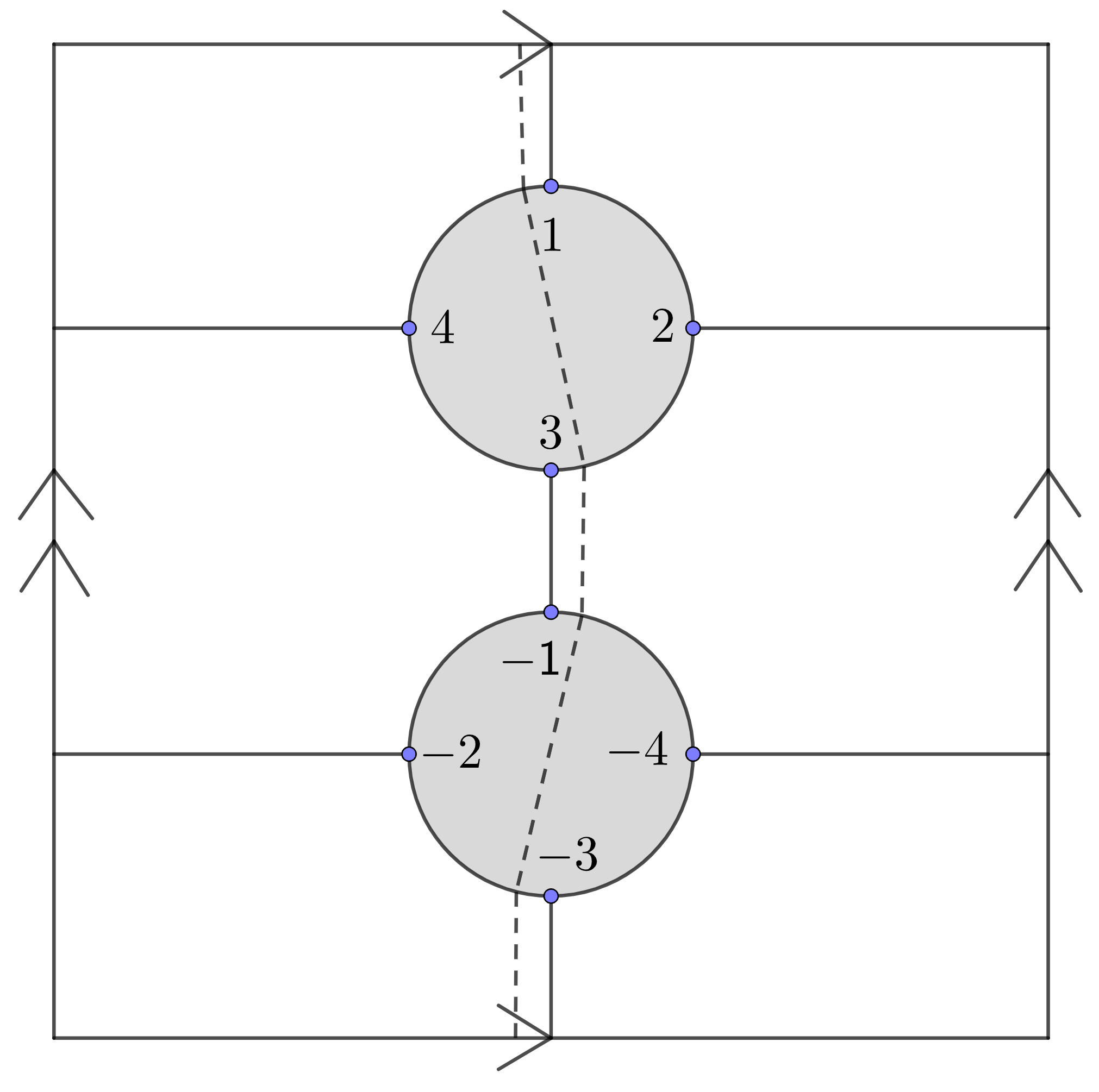}
        \caption{Double cover on the torus.} \label{fig9c}
    \end{subfigure}

    \begin{subfigure}{0.49\textwidth}
        \centering
        \captionsetup{justification=centering}
        \includegraphics[width=0.8\textwidth]{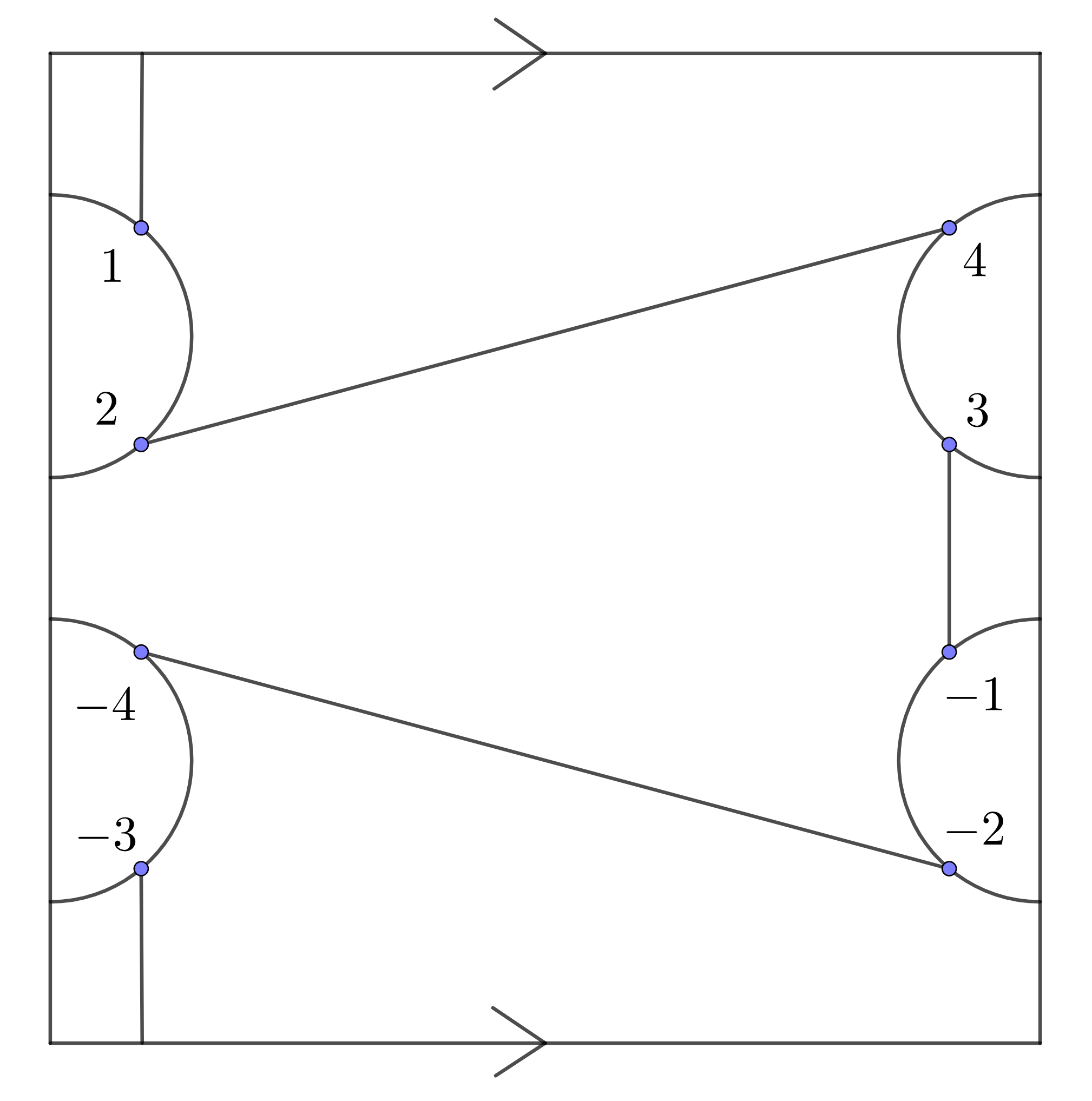}
        \caption{After cutting and regluing.} \label{fig9d}
    \end{subfigure}\hfill
    \begin{subfigure}{0.49\textwidth}
        \centering
        \captionsetup{justification=centering}
        \includegraphics[width=0.8\textwidth]{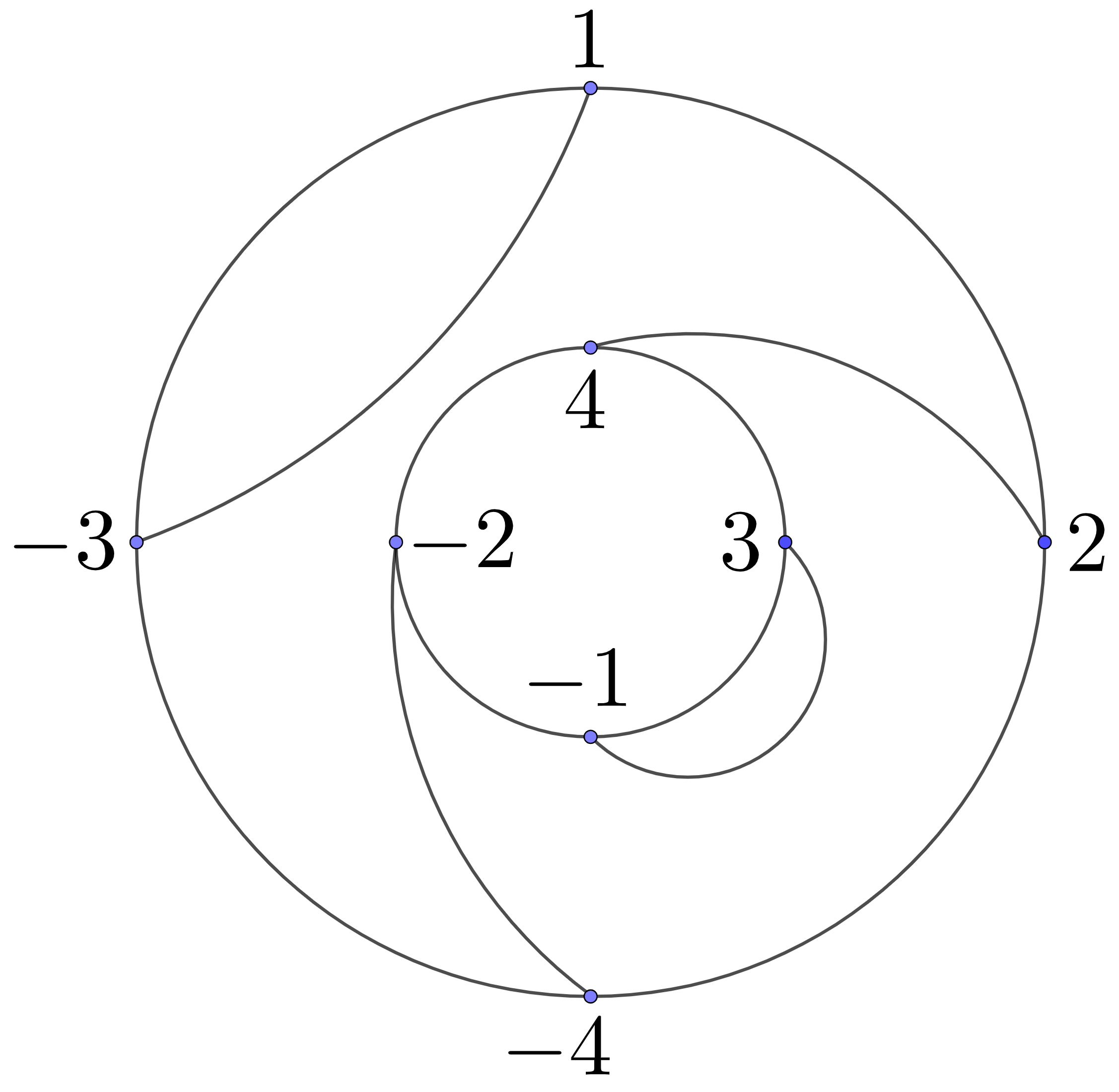}
        \caption{The non-crossing annular pairing.} \label{fig9e}
    \end{subfigure}
    \caption{Given $\tau_1\in b_2(4)$, consider its restriction to the positively-labelled half-edges obtained by quotienting the set of quarter-edges $\pm[4]$ by the action of $\tau_0$; this is (a). To obtain a symmetric non-crossing annular pairing from this graph, proceed by drawing an inverted (reflected) copy of it with labels replaced by their negatives, (b), then glue the bottom boundary of (a) to the top boundary of (b) to produce (c), a two-vertex ribbon graph on a torus. Cutting along the dotted line as prescribed above and gluing along the left and right sides produces (d). Finally, gluing the top and bottom together and excising the ribbon graph vertices yields the annular pairing (e), which is non-crossing with respect to $\widetilde{1}_{4,1,3}=(1,2,-4,-3)(3,4,-2,-1)$.}
\end{figure}

We do not give a topological description of the inverse mapping and instead move onto the combinatorial statement of our bijection.

\begin{proposition}[Bijection for $b_2(n)$]\label{prop7}
Let $n\in\N$ be even and recall the specifications of $\NC_2^{\rm K}(n)$ and $b_k(n)$ given in definitions \ref{def9} and \ref{def7}. Then, the mapping
\begin{align*}
\varphi_2:b_2(n)&\to\NC_2^{\rm K}(n),
\\\tau_1&\mapsto\tau_1\tau_0
\end{align*}
is a bijection, where we recall that $\tau_0=(1,-1)(2,-2)\cdots(n,-n)$.
\end{proposition}

\begin{proof}
We first show that the image of $\varphi_2$ indeed lies within $\NC_2^{\rm K}(n)$. Thus, we check the relevant conditions of definitions \ref{def4}, \ref{def5}, and \ref{def9}:
\begin{enumerate}
\item $\tau_1\tau_0$ is a pairing because, by Definition \ref{def7}, it is fixed-point free and $\tau_0\tau_1=\tau_1\tau_0$ implies that $(\tau_1\tau_0)^2$ is the identity, so $\varphi_2(\tau_1)=\tau_1\tau_0$ is an involution.
\item There does not exist $r\in[n]$ such that $(-r,r)$ is a cycle of $\varphi_2(\tau_1)=\tau_1\tau_0$ because if there were such a cycle, $\tau_1$ would have $r$ and $-r$ as fixed points, which would be in contradiction with it being a pairing.
\item For every $r,s\in\pm[n]$ such that $(r,s)$ is a cycle of $\varphi_2(\tau_1)=\tau_1\tau_0$, we must have that
\begin{equation*}
\tau_1(-r)=\tau_1\tau_0(r)=s\,\textrm{ and }\tau_1(-s)=\tau_1\tau_0(s)=r,
\end{equation*}
so by Definition \ref{def7},
\begin{equation*}
\tau_1\tau_0(-r)=\tau_0\tau_1(-r)=-s\,\textrm{ and }\tau_1\tau_0(-s)=\tau_0\tau_1(-s)=-r.
\end{equation*}
Hence, $(-s,-r)$ is also a cycle of $\varphi_2(\tau_1)=\tau_1\tau_0$.
\item By the definition of $b_2(n)$, there exists a minimal $u\in[n]$ such that $\tau_1(u)\in[n]$, which we denote $v$. Then, by equations \eqref{tau0}, \eqref{tau2}, $\widetilde{1}_n=\tau_2\tau_0$, hence
\begin{equation*}
\widetilde{1}_{n,u,v}=(u,1-v)(v,1-u)\tau_2\tau_0
\end{equation*}
and we have that
\begin{align}
\#(\varphi_2(\tau_1)^{-1}\widetilde{1}_{n,u,v})&=\#(\tau_0\tau_1(u,1-v)(v,1-u)\tau_2\tau_0) \nonumber
\\&=\#(\tau_1(u,1-v)(v,1-u)\tau_2) \nonumber
\\&=\#((u,1-v)(v,1-u)\tau_2\tau_1), \label{eq25}
\end{align}
where we have used the fact that the conjugation of a permutation has the same number of cycles as the original permutation and that $\tau_0,\tau_1$ are involutions. Next, observe that $\tau_2\tau_1(v)=\tau_2(u)=1-u$, so $\tau_2\tau_1$ contains a cycle, say $\sigma$, such that $v,1-u\in\sigma$, so $(v,1-u)\sigma$ has two cycles and
\begin{equation*}
\#((v,1-u)\tau_2\tau_1)=\#(\tau_2\tau_1)+1.
\end{equation*}
In a similar fashion, we see that
\begin{align*}
(v,1-u)\tau_2\tau_1(u)&=(v,1-u)\tau_2(v)
\\&=1-v
\\ \implies \#((u,1-v)(v,1-u)\tau_2\tau_1)&=\#((v,1-u)\tau_2\tau_1)+1
\\&=\#(\tau_2\tau_1)+2.
\end{align*}
Returning to equation \eqref{eq25}, we see that 
\begin{equation*}\#((\varphi_2(\tau_1)^{-1}\widetilde{1}_{n,u,v}))=\#(\tau_2\tau_1)+2=n-2+2=n.
\end{equation*}
As $\varphi_2(\tau_1)=\tau_1\tau_0$ is a pairing on $\pm[n]$, we thus have
\begin{equation*}
\#(\varphi_2(\tau_1))+\#(\varphi_2(\tau_1)^{-1}\widetilde{1}_{n,u,v})+\#(\widetilde{1}_{n,u,v})=2n+2.
\end{equation*}
\item As $\tau_1\tau_0=\tau_0\tau_1$ and this permutation is an involution, we have that
\begin{equation*}
(\varphi_2(\tau_1))(u)=\tau_1\tau_0(u)=\tau_0\tau_1(u)=\tau_0(v)=-v,
\end{equation*}
so $(u,-v)\in \varphi_2(\tau_1)=\tau_1\tau_0$.

\item By the definition of $u$, $\tau_1(u')<0$ for all $u'\in [u-1]$. Hence, for any such $u'$,
\begin{equation*}
(\varphi_2(\tau_1))(u')=\tau_1\tau_0(u')=\tau_0\tau_1(u')>0.
\end{equation*}

\item By Definition \ref{def4}, $\varphi_2(\tau_1)$ is crossing with respect to $\widetilde{1}_n$ because, with $u,v$ as above, $(u,-v)\in\varphi_2(\tau_1)\implies\varphi_2(\tau_1)\vee\widetilde{1}_n=1_{\pm[n]}$ and
\begin{align}
&\#(\varphi_2(\tau_1)^{-1}\widetilde{1}_n)=\#(\tau_2\tau_0\tau_0\tau_1)=\#(\tau_2\tau_1)=n-2 \label{equ:7first}
\\\implies&\#(\varphi_2(\tau_1))+\#(\varphi_2(\tau_1)^{-1}\widetilde{1}_n)+\#(\widetilde{1}_n)=2n\ne2n+2. \nonumber
\end{align}
Thus, defining
\begin{align*}
A&:=[n]\setminus\{u,\ldots,v\}\cup\{-v-1,\ldots,1-u\},
\\B&:=-[n]\setminus\{-v,\ldots,-u\}\cup\{u+1,\ldots,v-1\}
\end{align*}
to be the blocks of $\widetilde{1}_{n,u,v}\vert_{\pm[n]\setminus\{u,v,-u,-v\}}$, if there exists some $u'\in A$ such that $(u',\tau_1\tau_0(u'))$ crosses $(u,-v)$ or $(-u,v)$ with respect to $\widetilde{1}_n$, it must be that $\tau_1\tau_0(u')\in B$, so $\varphi_2(\tau_1)\vee\widetilde{1}_{n,u,v}=1_{\pm[n]}$. This conclusion is our goal, so we henceforth assume for the sake of contradiction that $\varphi_2(\tau_1)$ is crossing with respect to $\widetilde{1}_n$ but $(u,-v)$ and $(-u,v)$ do not cross with any cycles. Now, define
\begin{align*}
\pi_A&:=\varphi_2(\tau_1)\vert_{A},
\\\pi_B&:=\varphi_2(\tau_1)\vert_{B},
\\\sigma_A&:=(v+1,\ldots,n,1,\ldots,u-1,1-v,\ldots,-u-1),
\\\sigma_B&:=(u+1,\ldots,v-1,1-u,\ldots,-1,-n,\ldots,1-v)
\end{align*}
so that $\varphi_2(\tau_1)=\pi_A\pi_B(u,-v)(v,-u)$ and one may check that
\begin{equation} \label{equ:4fixed}
(u,1-v)(v,1-u)(-u,v+1)(-v,u+1)\widetilde{1}_n=\sigma_A\sigma_B(u,-v)(v,-u),
\end{equation}
where we again write $(v,0)\equiv(v,-n)$. Then, by equation \eqref{equ:7first}, we observe that
\begin{align}
n-2&=\#(\varphi_2(\tau_1)^{-1}\widetilde{1}_n) \nonumber
\\&=\#((u,-v)(v,-u)\pi_B^{-1}\pi_A^{-1}\widetilde{1}_n) \nonumber
\\&=\#((-v,u+1)(-u,v+1)(v,1-u)(u,1-v)\sigma_A\sigma_B\pi_B^{-1}\pi_A^{-1}), \label{equ:sigmaAB}
\end{align}
where the last line follows from substituting in equation \eqref{equ:4fixed} and using the fact that the number of cycles of a permutation conjugated by $(u,-v)(v,-u)$ is the same as the number of cycles of the original permutation. As $u,v,-u,-v$ are fixed points of $\sigma_A\sigma_B\pi_B^{-1}\pi_A^{-1}$, it has cycles $(u),(v),(-u),(-v)$ and we must have
\begin{equation*}
\#((-v,u+1)(-u,v+1)(v,1-u)(u,1-v)\sigma_A\sigma_B\pi_B^{-1}\pi_A^{-1})=\#(\sigma_A\sigma_B\pi_B^{-1}\pi_A^{-1})-4,
\end{equation*}
so we see from equation \eqref{equ:sigmaAB} that
\begin{equation*}
\#(\sigma_A\sigma_B\pi_B^{-1}\pi_A^{-1})=n+2.
\end{equation*}
Since $A,B$ are symmetric and disjoint and we have the symmetries
\begin{align*}
\pi_A&=\tau_0\pi_B\tau_0,
\\ \sigma_A&=\tau_0\sigma_B^{-1}\tau_0,
\end{align*}
we thus have that $\#(\pi_A^{-1}\sigma_A)=n/2+1$. Now, our hypothesis is that $\varphi_2(\tau_1)$ is crossing with respect to $\widetilde{1}_n$ but $(u,-v)$ and $(-u,v)$ do not cross with any cycles, so we must conclude that $\pi_A$ (and by symmetry, $\pi_B$) is crossing with respect to $\sigma_A$ ($\sigma_B$). Letting $(a,b)$ be a cycle of $\pi_A$ crossing with respect to $\sigma_A$, we may repurpose equation \eqref{equ:torusid} to obtain
\begin{equation*}
\#(\pi_A^{-1}\widehat{\sigma}_A)=\#(\pi^{-1}\sigma_A)+1=n/2+2,\quad \widehat{\sigma}_A:=\sigma_A(a-1,b).
\end{equation*}
Since $\pi_A\vee\widehat{\sigma}_A=1_A$, this is in contradiction with the inequality \cite[Eq.~(2.10)]{mingo2004annular}
\begin{equation*}
\#(\pi_A)+\#(\pi_A^{-1}\widehat{\sigma}_A)+\#(\widehat{\sigma}_A)\le |A|+2\#(\pi_A\vee\widehat{\sigma}_A);
\end{equation*}
the left-hand side simplifies as $|A|/2+(n/2+2)+2=n+3$, while the right-hand side reads $|A|+2=n$. Therefore, we finally conclude that $(u,-v)$ or $(-u,v)$ must be crossing with respect to $\widetilde{1}_n$ and $\varphi_2(\tau_1)\vee\widetilde{1}_{n,u,v}=1_{\pm[n]}$. The above inequality is a well known generalisation of the second condition of Definition \ref{def4}, where in the parlance of Remark \ref{rmk4}, the term $\#(\pi_A\vee\widehat{\sigma}_A)$ counts the number of connected components of the hypermap at hand. It can be proven via induction, with the cycles of $\pi_A$ being interpreted as non-crossing edges on the annulus corresponding to $\widehat{\sigma}_A$ --- the induction step simplifies the cell decomposition of said annulus by removing one edge, hence face, leaving the Euler characteristic unchanged.
\end{enumerate}

Conditions (4) and (7) ensure that $\tau_1\tau_0$ is non-crossing with respect to $\widetilde{1}_{n,u,v}$, so is an element of $\NC(\widetilde{1}_{n,u,v})$. Then, conditions (2) and (3) place $\tau_1\tau_0$ in $\NC^\delta(\widetilde{1}_{n,u,v})$, while condition (1) refines this to $\varphi_1(\tau_1)=\tau_1\tau_0\in\NC_2^\delta(\widetilde{1}_{n,u,v})$. Finally, conditions (5) and (6) imply that $\varphi_1(\tau_1)=\tau_1\tau_0\in \NC_{2,u,v}^{\rm K}(n) \subset \NC_2^{\rm K}(n)$, as desired.

Since $\tau_0$ is an involution, $\varphi_2$ is injective, so it remains to prove that it is also surjective. Thus, let $\pi\in\NC_2^{\rm K}(n)$ and check that $\varphi_2^{-1}(\pi)=\pi\tau_0$ satisfies the conditions of $b_2(n)$ given in Definition \ref{def7}:
\begin{enumerate}
\item By exactly the same arguments as in the proof of Proposition \ref{prop5}, $\varphi^{-1}(\pi)\in\cP_2(\pm n)$, $\tau_0\varphi_2^{-1}(\pi)=\varphi_2^{-1}(\pi)\tau_0$, and $\tau_0\varphi_2^{-1}(\pi)$ is fixed-point free.
\item By the definition of $\NC_2^{\rm K}(n)$, there exist $1\leq u<v<n$ such that $\pi\in\NC_{2,u,v}^{\rm K}(n)\subset\NC_2^\delta(\widetilde{1}_{n,u,v})$. Thus, by definitions \ref{def4} and \ref{def5}, we have that
\begin{equation*}
\#(\pi)+\#(\pi^{-1}\widetilde{1}_{n,u,v})=2n.
\end{equation*}
Recalling that $\pi$ is a pairing and $\widetilde{1}_{n,u,v}=(u,1-v)(v,1-u)\tau_2\tau_0$ then shows that
\begin{equation}\label{equ:1}
\#((u,1-v)(v,1-u)\tau_2\tau_0\pi)=\#(\pi^{-1}(u,1-v)(v,1-u)\tau_2\tau_0)=n.
\end{equation}
As $\pi(v)=-u$,
\begin{align*}
\tau_2\tau_0\pi(v)&=\tau_2\tau_0(-u)=\tau_2(u)=1-u,
\end{align*}
so $v$ and $1-u$ are in the same cycle of $\tau_2\tau_0\pi$, hence
\begin{equation}\label{equ:2}
\#((v,1-u)\tau_2\tau_0\pi)=\#(\tau_2\tau_0\pi)+1.
\end{equation}
Similarly,
\begin{align*}
(v,1-u)\tau_2\tau_0\pi(u)&=(v,1-u)\tau_2\tau_0(-v)=(v,1-u)\tau_2(v)=1-v,
\end{align*}
so $u$ and $1-v$ are in the same cycle of $(v,1-u)\tau_2\tau_0\pi$ and we have
\begin{equation*}
\#((u,1-v)(v,1-u)\tau_2\tau_0\pi)=\#((v,1-u)\tau_2\tau_0\pi)+1= \#(\tau_2\tau_0\pi)+2,
\end{equation*}
where the last equality follows from equation \eqref{equ:2}. Therefore, substituting this into equation \eqref{equ:1} yields
\begin{equation*}
\#(\tau_2\tau_0\pi)=\#((u,1-v)(v,1-u)\tau_2\tau_0\pi)-2=n-2.
\end{equation*}
Finally, we observe that $\varphi_2(\tau_1)=\tau_1\tau_0=\tau_0\tau_1\implies\varphi_2^{-1}(\pi)=\pi\tau_0=\tau_0\pi$, thus
\begin{equation*}
\#(\tau_2\varphi_2^{-1}(\pi))=\#(\tau_2\tau_0\pi)=n-2.
\end{equation*}

\item By the definition of $\NC_{2,u,v}^{\rm K}(n)$, our choice of $u,v\in[n]$ above is such that $(u,-v)\in\pi$, so
\begin{equation*}
\varphi_2^{-1}(\pi)(v)=\pi\tau_0(v)=\pi(-v)=u.
\end{equation*}
Hence, $\varphi_2^{-1}(\pi)$ contains the cycle $(u,v)\in\cP_2(n)$.
\end{enumerate}
Having shown that $\varphi_2$ is injective and surjective, we are done.
\end{proof}

\section{Bipartite Pairings, Permutations, and Moments of the Laguerre Ensembles} \label{s3}
The theory of Section \ref{s2} transfers readily to the setting of the equivalent Laguerre ensembles. Indeed, by propositions \ref{prop3} and \ref{prop4}, the moments of the LUE and LOE (with $M=cN$, recall) are respectively given by the polynomial expansions
\begin{align*}
m_n^{(\mathrm{LUE})}&=\left(\sum_{p=1}^nc^p|\widetilde{a}_{0,p}(n)|\right)N^{n+1}+\left(\sum_{p=1}^nc^p|\widetilde{a}_{1,p}(n)|\right)N^{n-1}+\mathrm{O}(N^{n-3}),
\\ m_n^{(\mathrm{LOE})}&=\left(\frac{1}{2^n}\sum_{p=1}^nc^p|\widetilde{a}_{0,p}(n)|\right)N^{n+1}+\left(\frac{1}{2^n}\sum_{p=1}^nc^p|\widetilde{b}_{1,p}(n)|\right)N^n
\\&\quad+\left(\frac{1}{2^n}\sum_{p=1}^nc^p\left(|\widetilde{a}_{1,p}(n)|+|\widetilde{b}_{2,p}(n)|\right)\right)N^{n-1}+\mathrm{O}(N^{n-2}),
\end{align*}
where the coefficients enumerate the same ribbon graphs as in the Gaussian case subject to the extra condition that the ribbon graphs be bipartite. Thus, we give a combinatorial definition of bipartite ribbon graphs in analogy with Definition \ref{def7}.

\begin{definition}[Bipartite ribbon graphs] \label{def10}
Let $n\in\N$. The set of genus $g$ orientable, bipartite ribbon graphs with $2n$ half-edges and $p$ white boundaries introduced in Proposition \ref{prop3} is
\begin{equation*}
\widetilde{a}_{g,p}(n)=\left\{\pi\in a_g(2n)\,\middle\vert\,\begin{array}{l}\textrm{for all }a\in2[n],\,\pi(a)\in2[n]-1\\\textrm{and }\#\left((\pi^{-1}1_{2n})\vert_{2[n]-1}\right)=p\end{array}\right\},
\end{equation*}
where we write $2[n]:=\{2,4,\ldots,2n\}$ and $2[n]-1:=\{1,3,\ldots,2n-1\}$.

The set of non-orientable, bipartite ribbon graphs with $n$ half-edges of Euler genus $k$ with $p$ white boundaries introduced in Proposition \ref{prop4} is
\begin{equation*}
\widetilde{b}_{k,p}(n)=\left\{\tau_1\in b_k(2n)\,\middle\vert\,\begin{array}{l}\textrm{for all }a\in B(n),\tau_1(a)\in B(n)\\\textrm{and }\#\left((\tau_2\tau_1)\vert_{W(n)}\right)=2p\end{array}\right\},
\end{equation*}
where we define
\begin{align*}
B(n)&:=2[n]-1\cup-2[n]=\{1,3,\ldots,2n-1,-2n,2-2n,\ldots,-2\},
\\W(n)&:=2[n]\cup1-2[n]=\{2,4,\ldots,2n,1-2n,3-2n,\ldots,-1\}
\end{align*}
to be the sets of black and white quarter-edges, respectively.
\end{definition}

Equivalent to the theory of bipartite ribbon graphs is that of combinatorial hypermaps or bipartite maps. Topologically, the idea is to merge together each quarter-edge $u\in1-2[n]=\{1-2n,\ldots,-3,-1\}$ with $\tau_2(u)$ so that the area bounded by the white boundaries of these quarter-edges shrink to white vertices. In the orientable case, one may further identify $u$ with $\tau_0(u)$ since there is no need to track orientation. Combinatorially, we have the following equivalence.

\begin{definition}[Combinatorial hypermaps] \label{def11}
Let $n,g,k\in\N$ and define the following sets of combinatorial hypermaps:
\begin{align*}
\widehat{a}_{g,p}(n)&=\left\{\pi\in\PS_n\,\middle\vert\,\#(\pi)=p\,\textrm{ and }\#(\pi^{-1}1_n)=n-p+1-2g\right\},
\\\widehat{b}_{k,p}(n)&=\left\{\tau_1\in\PS_{\pm n}\,\middle\vert\,\begin{array}{l}\#(\tau_1)=2p,\,\#(\widetilde{1}_n\tau_1)=2(n-p+1-k),\\\tau_0\tau_1\tau_0=\tau_1^{-1},\,\tau_0\tau_1\textrm{ is fixed-point free, and}\\\textrm{there exists }a\in[n]\textrm{ such that }\tau_1(a)\in-[n]\end{array}\right\}.
\end{align*}
\end{definition}

\begin{lemma}[Equivalence of bipartite ribbon graphs and combinatorial hypermaps] \label{lemma3}
For $n,g,k\in\N$, we have the bijections
\begin{align*}
\widetilde{a}_{g,p}(n)&\cong\widehat{a}_{g,p}(n),
\\\widetilde{b}_{k,p}(n)&\cong\widehat{b}_{k,p}(n).
\end{align*}
\end{lemma}
\begin{proof}
In the orientable case, we identify each half-edge $u\in2[n]-1$ of $\pi\in\widetilde{a}_{g,p}(n)$ with $u+1$. Then, the corresponding $\pi'\in\widehat{a}_{g,p}(n)$ is such that, for each $u'\in[n]$,
\begin{equation*}
\pi'(u')=\frac{\pi(2u')+1}{2}.
\end{equation*}
In the non-orientable case, we first identify each quarter-edge $u\in[2n]$ of $\tau_1\in\widetilde{b}_{k,p}(n)$ with $-u$, as our ribbon graphs being bipartite makes $\tau_0$ redundant. Then, similar to the orientable case, we glue each half-edge $u\in2[n]-1$ to $u+1$ along their white boundaries (but do not identify, as we still need to keep track of twists) so that these half-edges now serve as quarter-edges. We relabel every $u\in2[n]$ by $-u/2$ and every $u\in2[n]-1$ by $(u+1)/2$. Then, with $f$ denoting this relabelling, the $\tau_1'\in\widehat{b}_{k,p}(n)$ corresponding to $\tau_1$ is $f\circ(\tau_2\tau_1)\circ f^{-1}$, i.e., the cycles of $\tau_2\tau_1$ are the hyperedges (white vertices) of our hypermaps (bipartite maps).
\end{proof}

\begin{figure}
    \centering
    \includegraphics[width=0.16\textwidth]{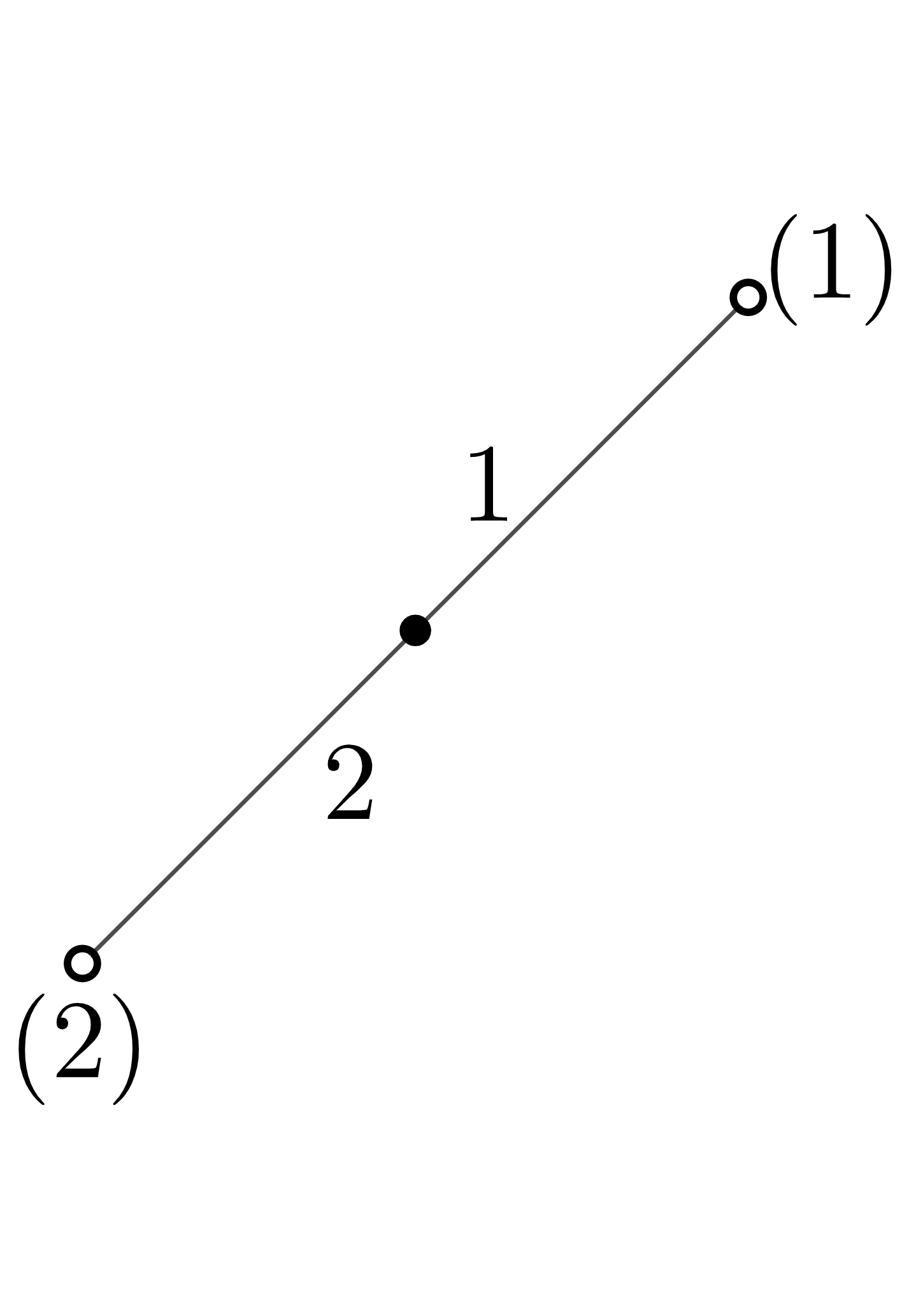}\qquad \includegraphics[width=0.14\textwidth]{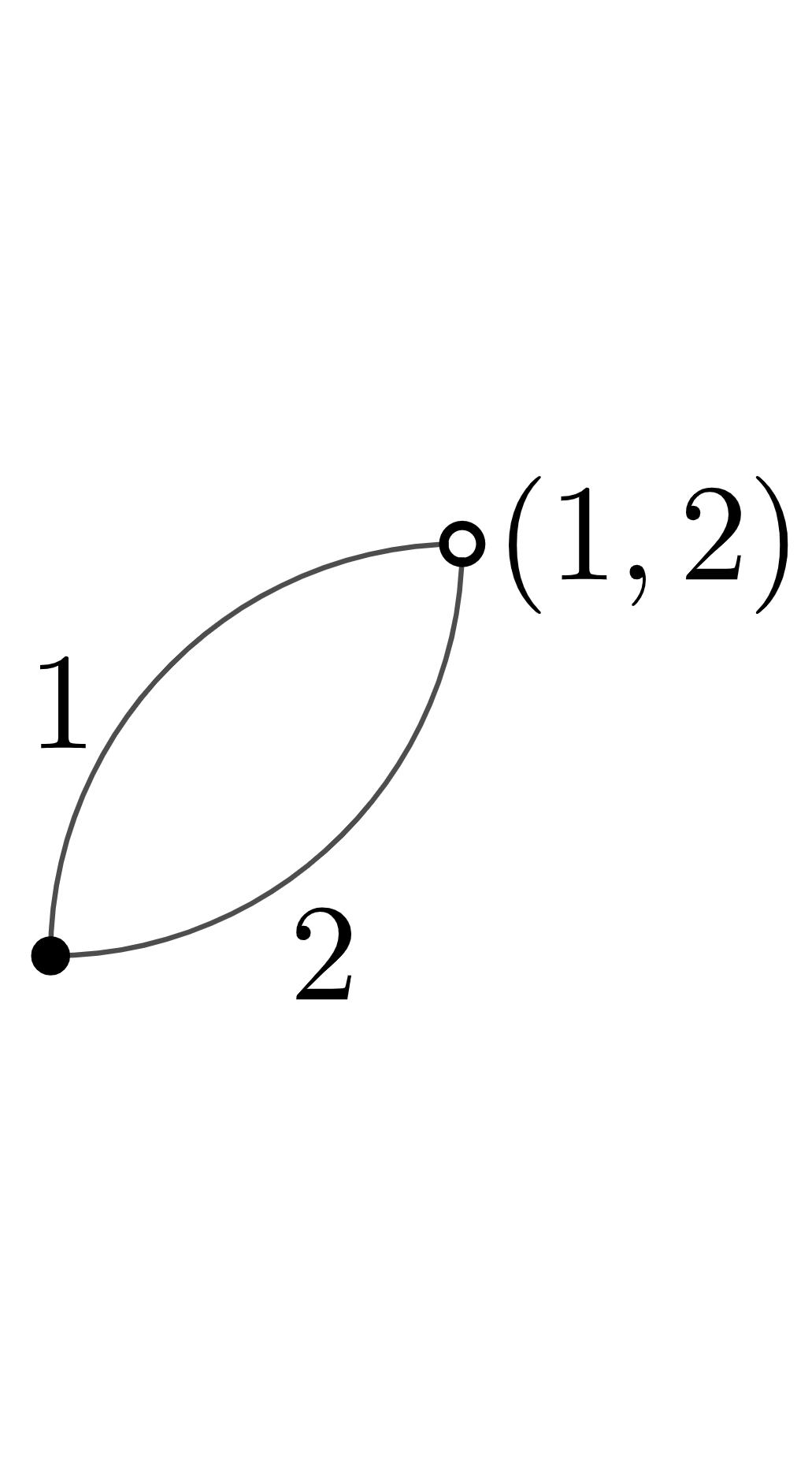}\qquad \includegraphics[width=0.42\textwidth]{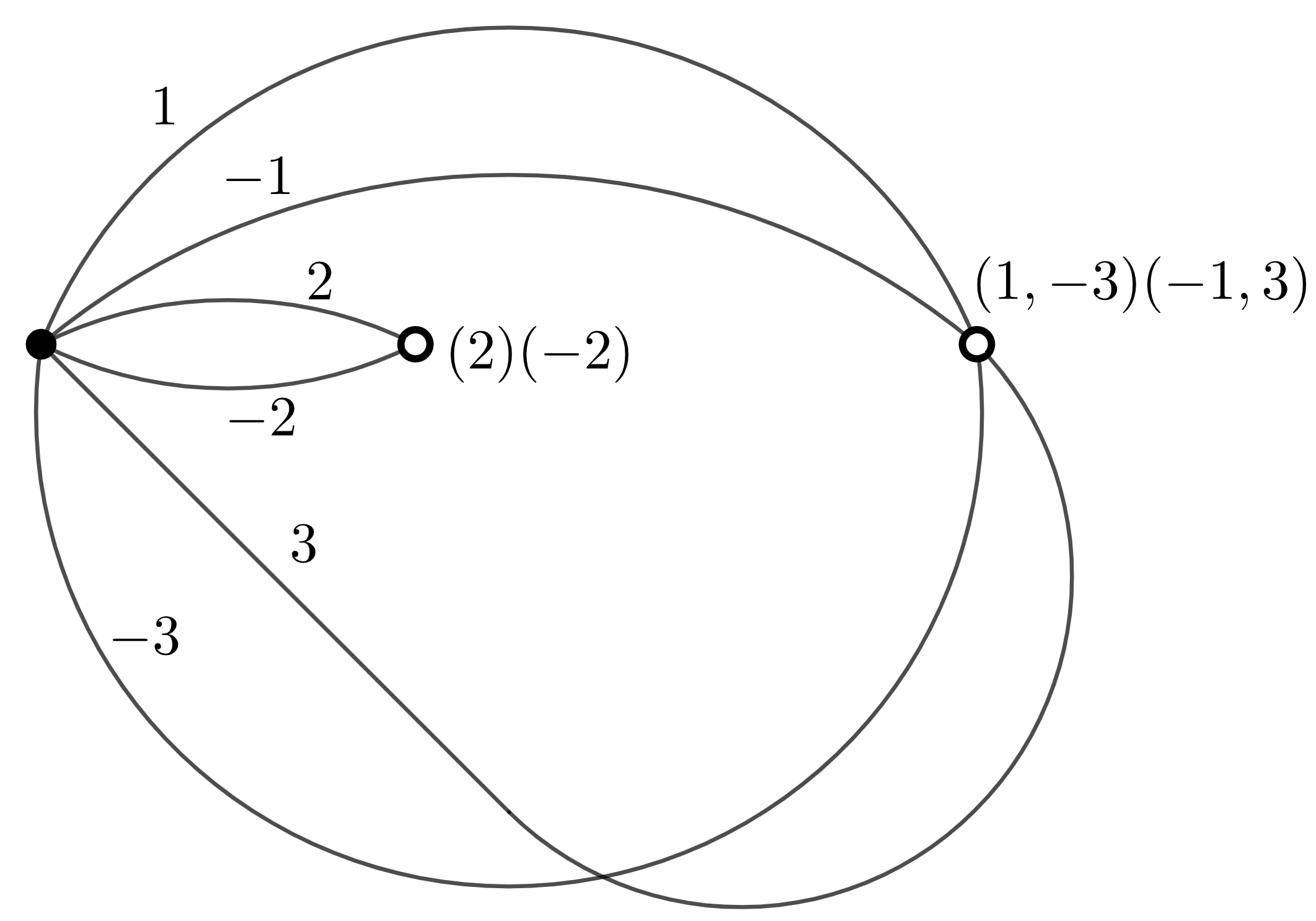}
    \caption{The left and middle orientable combinatorial hypermaps correspond to the ribbon graphs of Figure \ref{fig3}. The non-orientable combinatorial hypermap on the right corresponds to the ribbon graph $\tau_1=(1,5)(-1,-5)(2,6)(-2,-6)(3,-4)(-3,4)$ of Figure \ref{fig7a}, which is bipartite since elements of $B(3)$ are mapped to $B(3)$ by $\tau_1$.}
    \label{fig10}
\end{figure}

We now show how the results of Section \ref{s2} extend to the Laguerre ensembles.

\subsection{Laguerre Ensemble Moments in Terms of Bipartite Ribbon Graphs and Pairings}

We begin by defining the relevant bipartite annular pairings.
\begin{definition}[Bipartite non-crossing annular pairings]
Let $n\in\N$ be even, let $p\in[n/2]$, let $1\le u<v<n$, and recall the sets of definitions \ref{def5}, \ref{def8}, and \ref{def9}. Recall also the sets $B(n)$ and $W(n)$ of Definition \ref{def10}. Then, we define
\begin{align*}
\widetilde{\NC}_{2}^{\delta,p}(n,-n)&:=\left\{\pi\in\NC_2^\delta(n,-n)\,\middle\vert\,\begin{array}{l}\textrm{for all }a\in W(n/2),\,\pi(a)\in B(n/2)\\\textrm{and }\#\left((\pi^{-1}\widetilde{1}_n)\vert_{B(n/2)}\right)=2p\end{array}\right\},
\\\widetilde{\NC}_{2,u,v}^{{\rm T},p}(n)&:=\left\{\pi\in\NC_{2,u,v}^{{\rm T},p}(n)\,\middle\vert\,\begin{array}{l}\textrm{for all }a\in[n]\textrm{ even, }\pi(a)\textrm{ is odd}\\\textrm{and }\#\left((\pi^{-1}1_{n})\vert_{2[n/2]-1}\right)=p\end{array}\right\},
\\\widetilde{\NC}_2^{{\rm T},p}(n)&:=\bigcup_{\substack{1\le u<v<n,\\v-u\textrm{ odd}}}\widetilde{\NC}_{2,u,v}^{{\rm T},p}(n),
\\\widetilde{\NC}_{2,u,v}^{{\rm K},p}(n)&:=\left\{\pi\in\NC_{2,u,v}^{{\rm K},p}(n)\,\middle\vert\,\begin{array}{l}\textrm{for all }a\in W(n/2),\,\pi(a)\in B(n/2)\\\textrm{and }\#\left((\pi^{-1}\widetilde{1}_n)\vert_{B(n/2)}\right)=2p\end{array}\right\},
\\\widetilde{\NC}_2^{{\rm K},p}(n)&:=\bigcup_{\substack{1\le u<v<n,\\v-u\textrm{ even}}}\widetilde{\NC}_{2,u,v}^{{\rm K},p}(n),
\end{align*}
\end{definition}

With these definitions in hand, it is relatively straightforward to define bijections between them and the bipartite ribbon graphs of Definition \ref{def10}.
\begin{proposition}[Bijections for $\widetilde{a}_{1,p}(n)$, $\widetilde{b}_{1,p}(n)$, and $\widetilde{b}_{2,p}(n)$]
Let $n\in\N$. We have that $\widetilde{a}_{1,p}(n)=\widetilde{\NC}_2^{{\rm T},p}(2n)$ and that the following maps are bijections:
\begin{align*}
\widetilde{\varphi}_1:\widetilde{b}_{1,p}(n)&\to\widetilde{\NC}_2^{\delta,p}(2n,-2n),
\\\tau_1&\mapsto\tau_1\tau_0,
\\\widetilde{\varphi}_2:\widetilde{b}_{2,p}(n)&\to\widetilde{\NC}_2^{{\rm K},p}(2n),
\\\tau_1&\mapsto\tau_1\tau_0.
\end{align*}
\end{proposition}
\begin{proof}
The sets at hand are subsets of the sets considered in Section \ref{s2} with additional constraints stemming from the bipartite nature of the ribbon graphs and non-crossing annular pairings pertaining to the Laguerre ensembles. In Section \ref{s2}, we already proved the bijections
\begin{align*}
b_1(2n)&\cong\NC_2^\delta(2n,-2n),
\\a_1(2n)&\cong\NC_2^{\rm T}(n),
\\b_2(2n)&\cong\NC_2^{\rm K}(n).
\end{align*}
Thus, it remains only to check the additional constraints in the definitions of the sets that we claim are in bijection. However, the additional constraints on $\widetilde{a}_{1,p}(n)$, $\widetilde{b}_{1,p}(n)$, and $\widetilde{b}_{2,p}(n)$ are respectively the same as for $\widetilde{\NC}_2^{{\rm T},p}(2n)$, $\widetilde{\NC}_2^{\delta,p}(2n,-2n)$, and $\widetilde{\NC}_2^{{\rm K},p}(2n)$, other than the fact that in the non-orientable cases, the role of $\tau_0$ is to pair every white label with a black one.
\end{proof}
\begin{remark}
One may surmise from \cite{mingob2025} that $|\widetilde{b}_{1,p}(n)|=|\widetilde{\NC}_2^{\delta,p}(2n,-2n)|$.
\end{remark}

\begin{figure}[H]
    \centering
    \begin{subfigure}{0.4\textwidth}
    \centering
    \captionsetup{justification=centering}
    \includegraphics[width=0.7\textwidth]{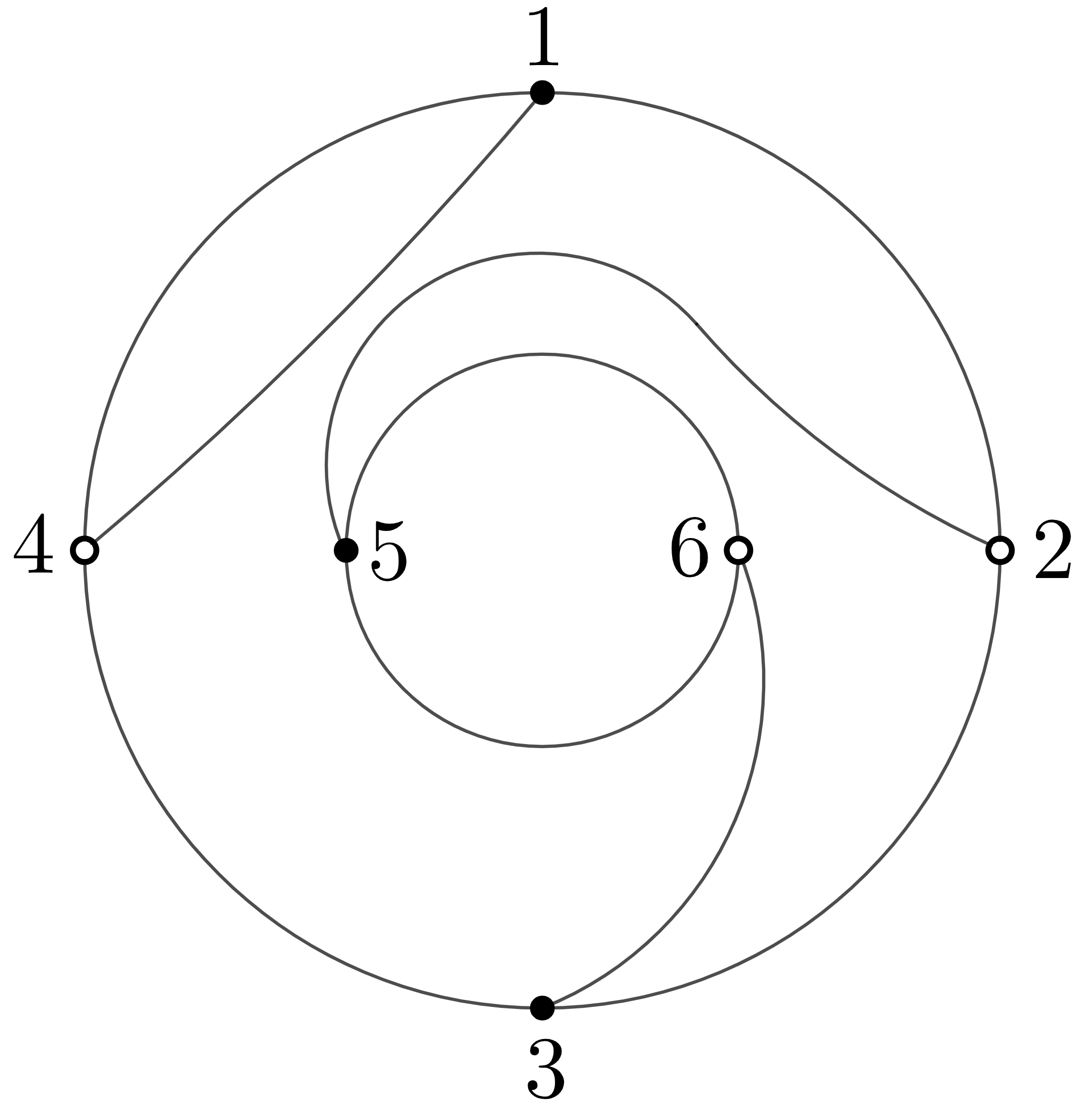}
    \caption{Bipartite, toroidal, non-crossing annular pairing.} \label{fig11a}
    \end{subfigure}\qquad\qquad
    \begin{subfigure}{0.4\textwidth}
    \centering
    \captionsetup{justification=centering}
    \includegraphics[width=0.7\textwidth]{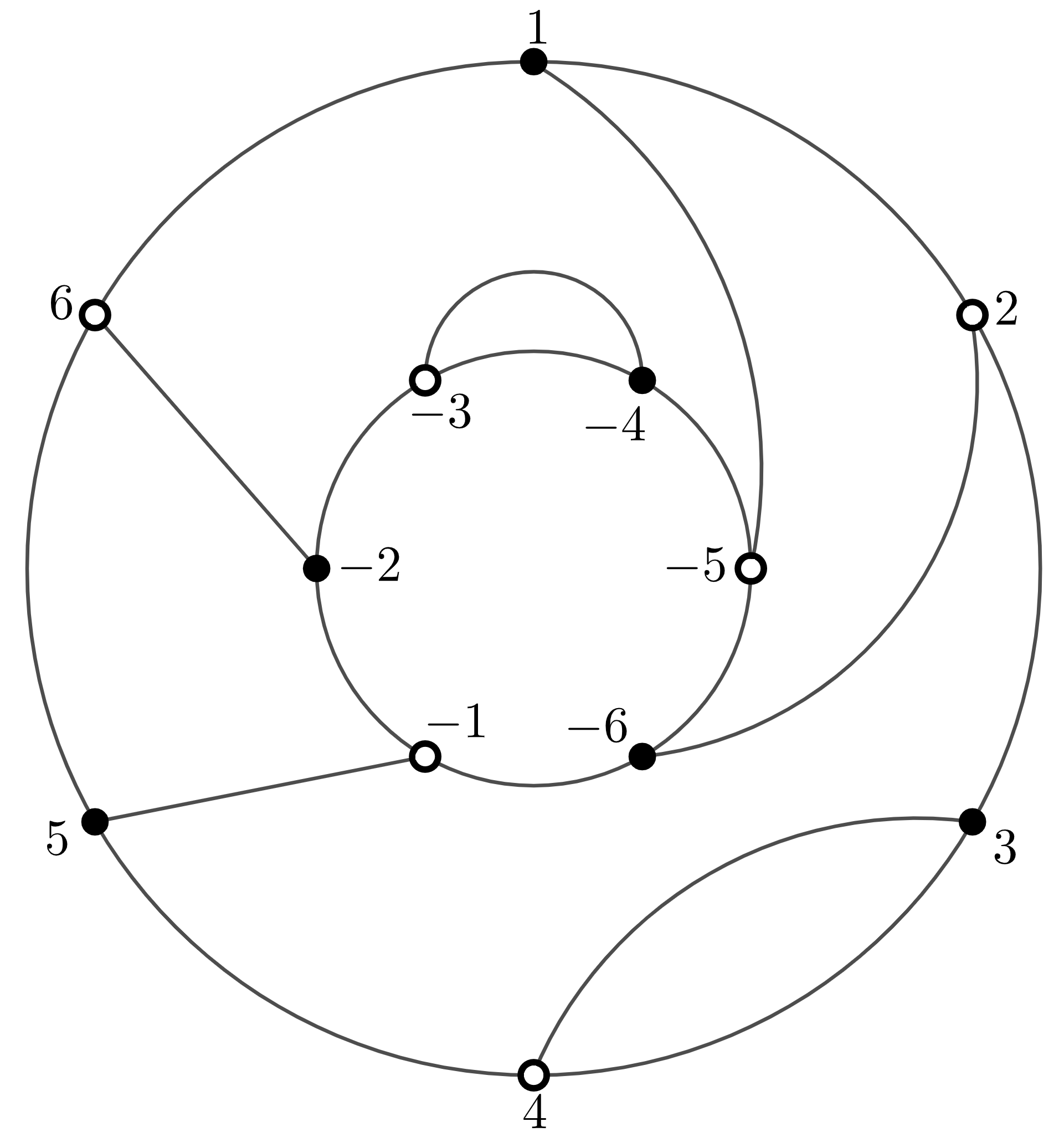}
    \caption{Bipartite, symmetric, non-crossing annular pairing.} \label{fig11b}
    \end{subfigure}
    \caption{The bipartite non-crossing annular pairing of (a) pairs odd labels coloured black with even labels coloured white; it corresponds to the toroidal ribbon graph $(1,4)(2,5)(3,6)$. In (b), we continue with the example of Figure \ref{fig7} by assigning the colours black and white to the labels of the non-crossing annular pairing of Figure \ref{fig7d} in accordance with the definitions of $B(3),W(3)$; this non-crossing annular pairing is bipartite as every black label is paired with a white one.}
    \label{fig11}
\end{figure}

\subsection{Laguerre Ensemble Moments in Terms of Non-Crossing Annular Permutations}

We conclude this paper with some comments on the the relation between the corrections to the LOE and LUE moments and non-crossing annular permutations.

\begin{definition}[Non-crossing annular permutations]
Let $n\in\N$, $p\in[n]$, and recall the sets of Definition \ref{def5}. We define
\begin{align*}
\NC^{\delta,p}(n,-n)&:=\{\pi\in\NC^\delta(n,-n)\,\vert\,\#(\pi)=2p\},
\\\NC^{{\rm T},p}(n)&:=\bigcup_{1\le u<v<n}\left\{\pi\in\NC\left(1_n(u-1,v)\right)\,\middle\vert\,\begin{array}{l}\pi^{-1}(u)=v,\,\#(\pi)=p,
\\\textrm{and for all }a\in[u-1],
\\\pi^{-1}(a)\in[n]\setminus\{u,\ldots,v\}\end{array}\right\},
\\\NC^{{\rm K},p}(n)&:=\bigcup_{1\le u<v<n}\left\{\pi\in\NC^\delta(\widetilde{1}_{n,u,v})\,\middle\vert\,\begin{array}{l}\pi(u)=-v,\,\#(\pi)=2p,\,\textrm{and}
\\\textrm{for all }a\in[u-1],\,\pi^{-1}(a)\in[n]\end{array}\right\}.
\end{align*}
\end{definition}
\begin{proposition}[Bijections for $\widehat{a}_{1,p}(n)$, $\widehat{b}_{1,p}(n)$, and $\widehat{b}_{2,p}(n)$]
Let $n\in\N$ and $p\in[n]$. Then, we have $\widehat{a}_{1,p}(n)=\NC^{{\rm T},p}(n)$ and that the following maps are bijections
\begin{align*}
\widehat{\varphi}_1:\widehat{b}_{1,p}(n)&\to\NC^{\delta,p}(n,-n),
\\\tau_1&\mapsto\tau_1\tau_0,
\\\widehat{\varphi}_2:\widehat{b}_{2,p}(n)&\to\NC^{{\rm K},p}(n),
\\\tau_1&\mapsto\tau_1\tau_0.
\end{align*}
\end{proposition}

\begin{remark}
Since the annulus is planar and oriented and our permutations are non-crossing, we do not need to remember the ordering of the cycles of our permutations, as this (anticlockwise) ordering is always equivalent to that inherited from the structure of the annulus itself. Thus, our non-crossing annular permutations are precisely equivalent to non-crossing annular partitions.
\end{remark}

We do not give a formal algebraic proof of this proposition, opting instead to present topological arguments that are perhaps more instructive; see \cite{mingob2025} for $\widehat{\varphi}_1$. In general, there are two approaches:
\begin{enumerate}
\item One may draw a ribbon graph on the appropriate fundamental polygon, checking that it is bipartite, then use the arguments of Section \ref{s2} to transform it into a bipartite non-crossing annular pairing. Then, identifying positive odd labels $a$ with $a+1$ and negative odd labels $b$ with $b-1$ causes the areas bounded between the relevant edges to shrink to (white) vertices, thereby resulting in non-crossing annular permutations.
\item On the other hand, one may perform half-edge identification at the ribbon graph stage to form combinatorial hypermaps, then use topological arguments similar to those of Section \ref{s2} to form the desired non-crossing annular permutations.
\end{enumerate}

\begin{figure}[H]
    \centering
    \begin{subfigure}{0.33\textwidth}
    \centering
    \captionsetup{justification=centering}
    \includegraphics[width=0.8\textwidth]{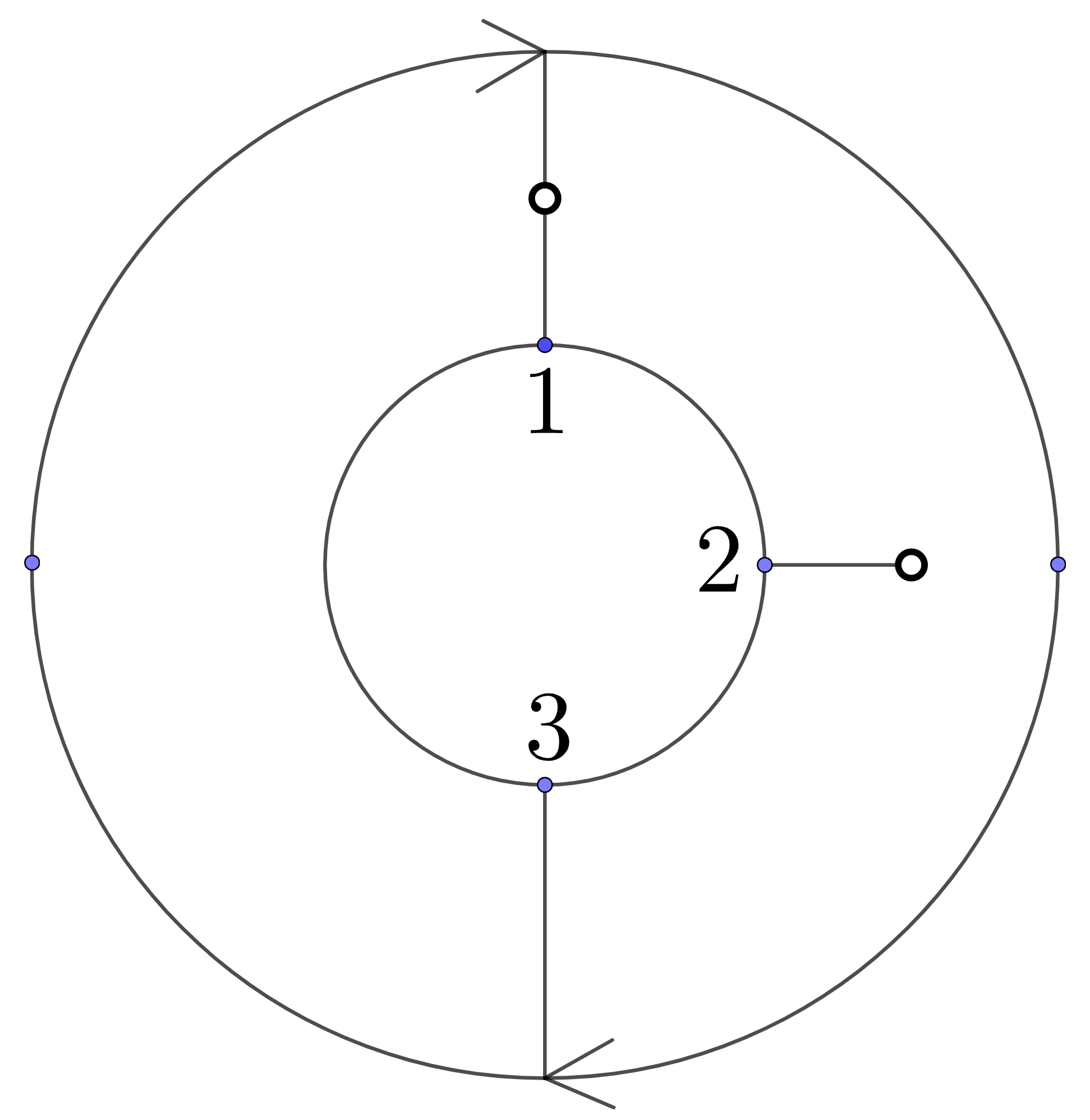}
    \caption{Hypermap on the real projective plane.} \label{fig12a}
    \end{subfigure}\qquad\qquad
    \begin{subfigure}{0.33\textwidth}
    \centering
    \captionsetup{justification=centering}
    \includegraphics[width=0.8\textwidth]{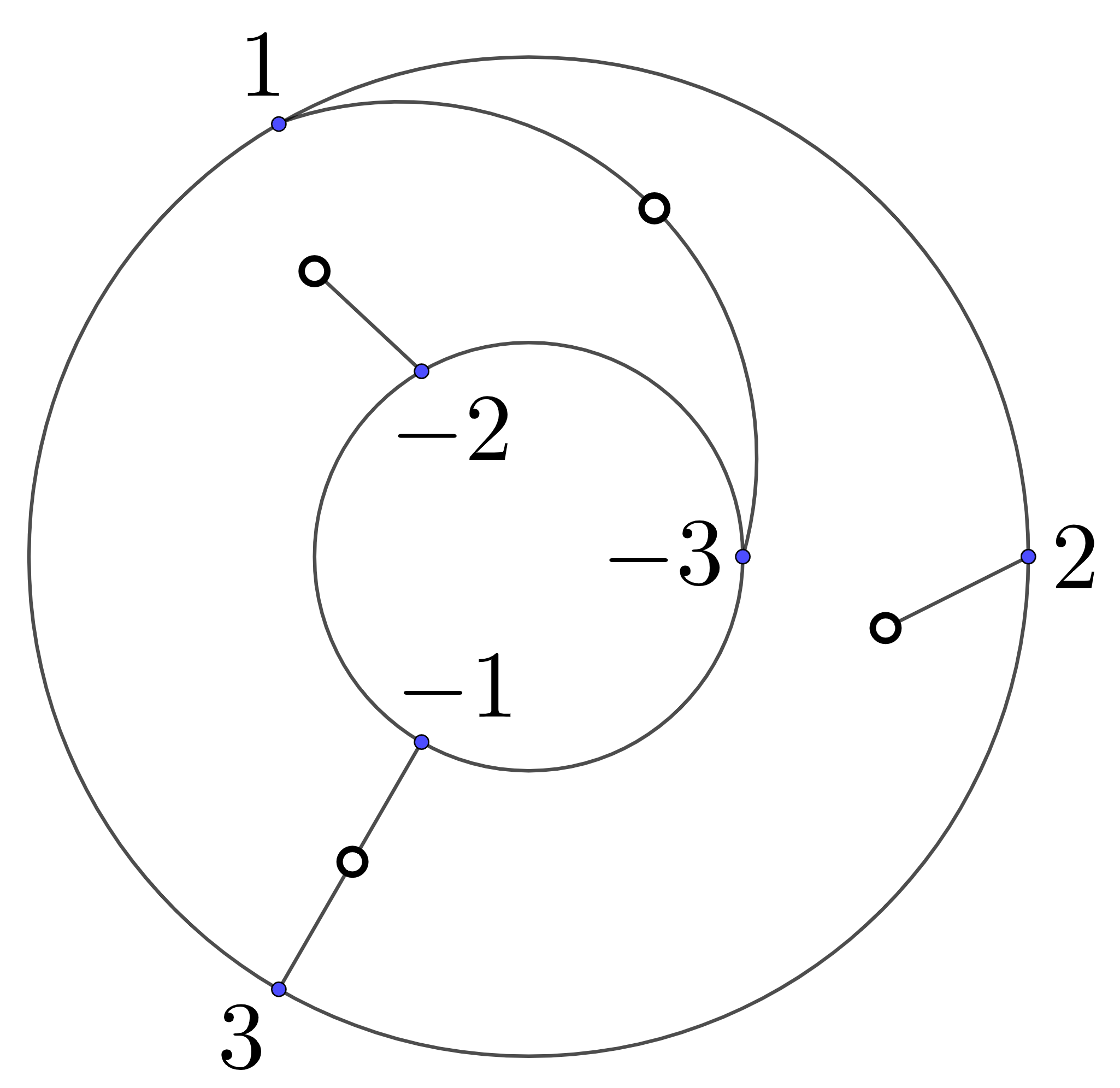}
    \caption{Symmetric non-crossing annular permutation.} \label{fig12b}
    \end{subfigure}
    \caption{Continuing once again with the example of Figure \ref{fig7}, we identify the half-edges of the ribbon graph in Figure \ref{fig7a} and relabel them so that $1\equiv2\mapsto1$, $3\equiv4\mapsto2$, and $5\equiv6\mapsto3$ in order to form (a). Then, gluing this hypermap to an inverted copy of itself yields the symmetric non-crossing annular permutation in (b). Alternatively, we may begin at Figure \ref{fig11b} and identify half-edges there to obtain (b) above.}
    \label{fig12}
\end{figure}

To give a little more detail, let us first consider the case of bipartite ribbon graphs on the real projective plane, i.e., those belonging to $\widetilde{b}_{1,p}(n)$. Given such a ribbon graph, we first draw its restriction to $[2n]\cong\pm[2n]/\tau_0$ on the fundamental polygon of the real projective plane, as in Figure \ref{fig7a}. Then, as the ribbon graph is bipartite, we are able to consistently colour the areas enclosed by the edges such that for each $u\in[2n]$ odd, the area between the half-edges labelled $u$ and $u+1$ is coloured white, while that between those labelled $u$ and $u-1$ (with $0\equiv 2n$) is coloured black. Doing so and then identifying all half-edges labelled by odd $u\in[2n]$ with those labelled $u+1$ shrinks the white areas down to vertices of the corresponding hypermap of $\widehat{b}_{1,p}(n)$, still drawn on the fundamental polygon of the real projective plane; see Figure \ref{fig12a}. Then, we may proceed as in \S\ref{s2.1} and glue said ribbon graph to its inverted copy to form its double cover and realise the result as a hypermap on an annulus, i.e., a non-crossing annular permutation; see Figure \ref{fig12b}.

Alternative to the above procedure, one may instead follow the algorithm of \S\ref{s2.1} entirely to arrive at a symmetric non-crossing annular pairing. Then, one is once again able to bicolour and identify half-edges in the manner described above to produce the desired non-crossing annular permutation.

\begin{figure}[H]
    \centering
    \begin{subfigure}{0.49\textwidth}
        \centering
        \captionsetup{justification=centering}
        \includegraphics[width=0.5\textwidth]{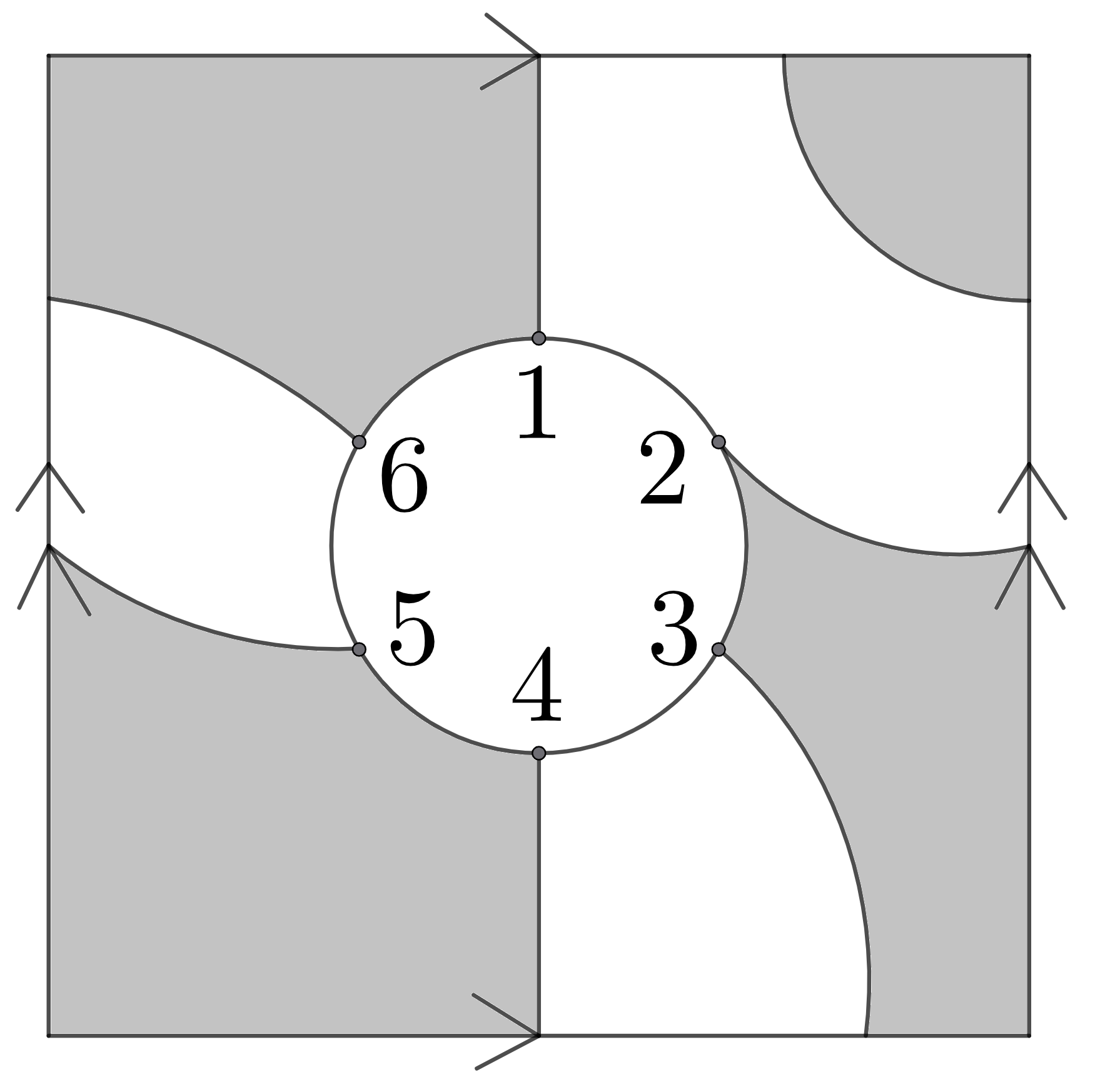}
        \caption{Bipartite, toroidal ribbon graph.} \label{fig13a}
    \end{subfigure}\hfill
    \begin{subfigure}{0.49\textwidth}
        \centering
        \captionsetup{justification=centering}
        \includegraphics[width=0.5\textwidth]{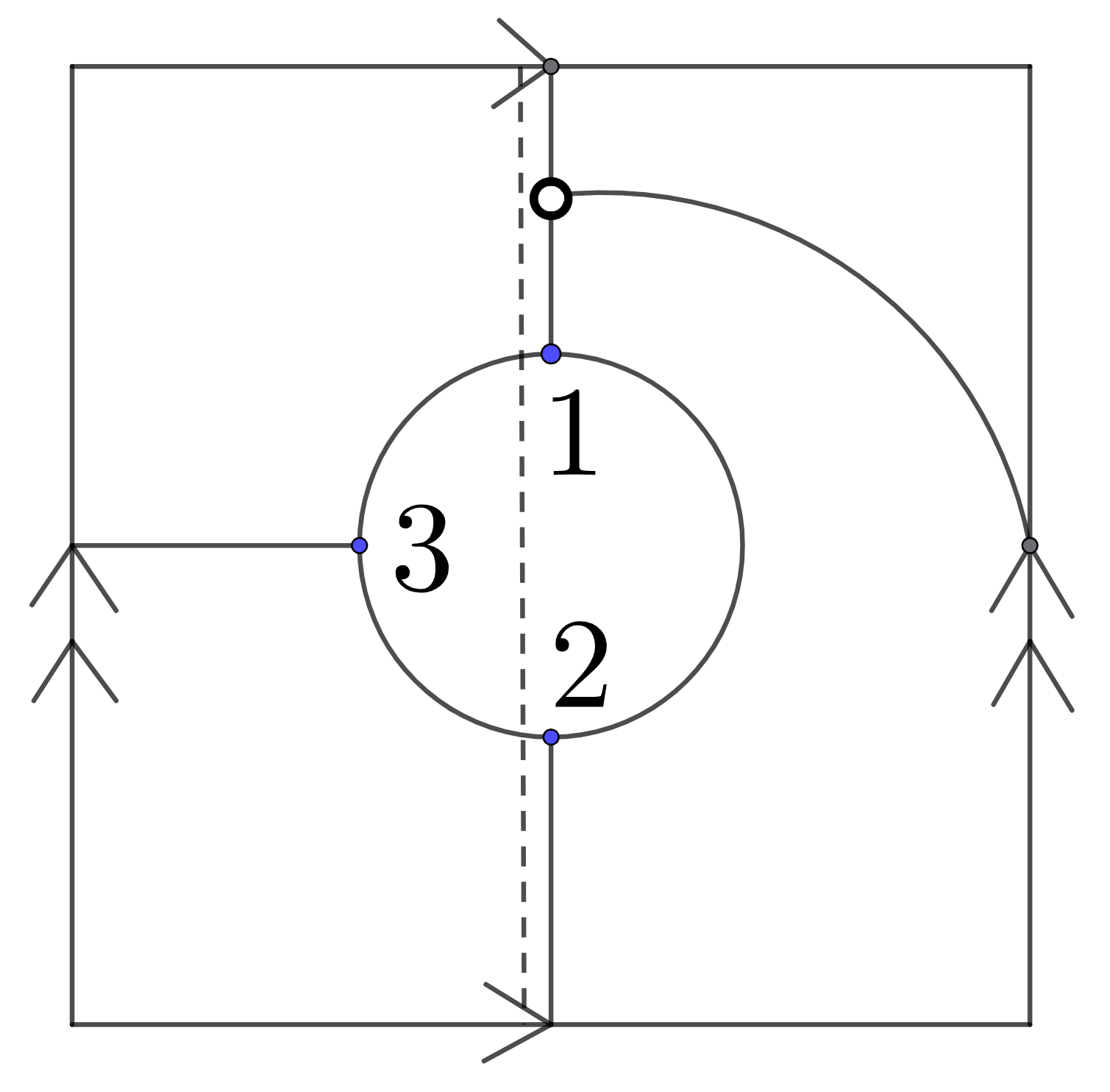}
        \caption{Toroidal hypermap.} \label{fig13b}
    \end{subfigure}

    \begin{subfigure}{0.49\textwidth}
        \centering
        \captionsetup{justification=centering}
        \includegraphics[width=0.6\textwidth]{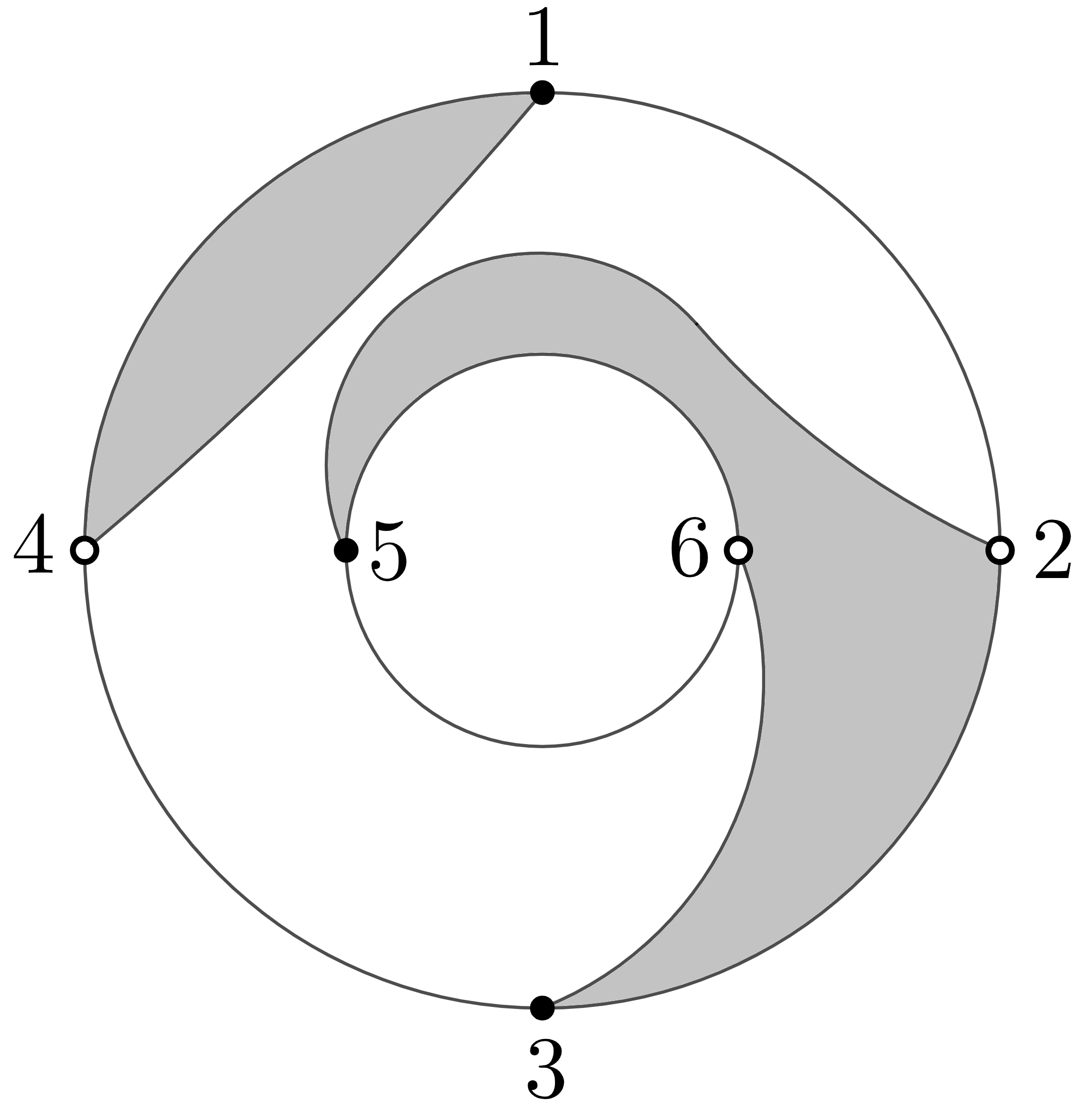}
        \caption{Bipartite non-crossing annular pairing.} \label{fig13c}
    \end{subfigure}\hfill
    \begin{subfigure}{0.49\textwidth}
        \centering
        \captionsetup{justification=centering}
        \includegraphics[width=0.5\textwidth]{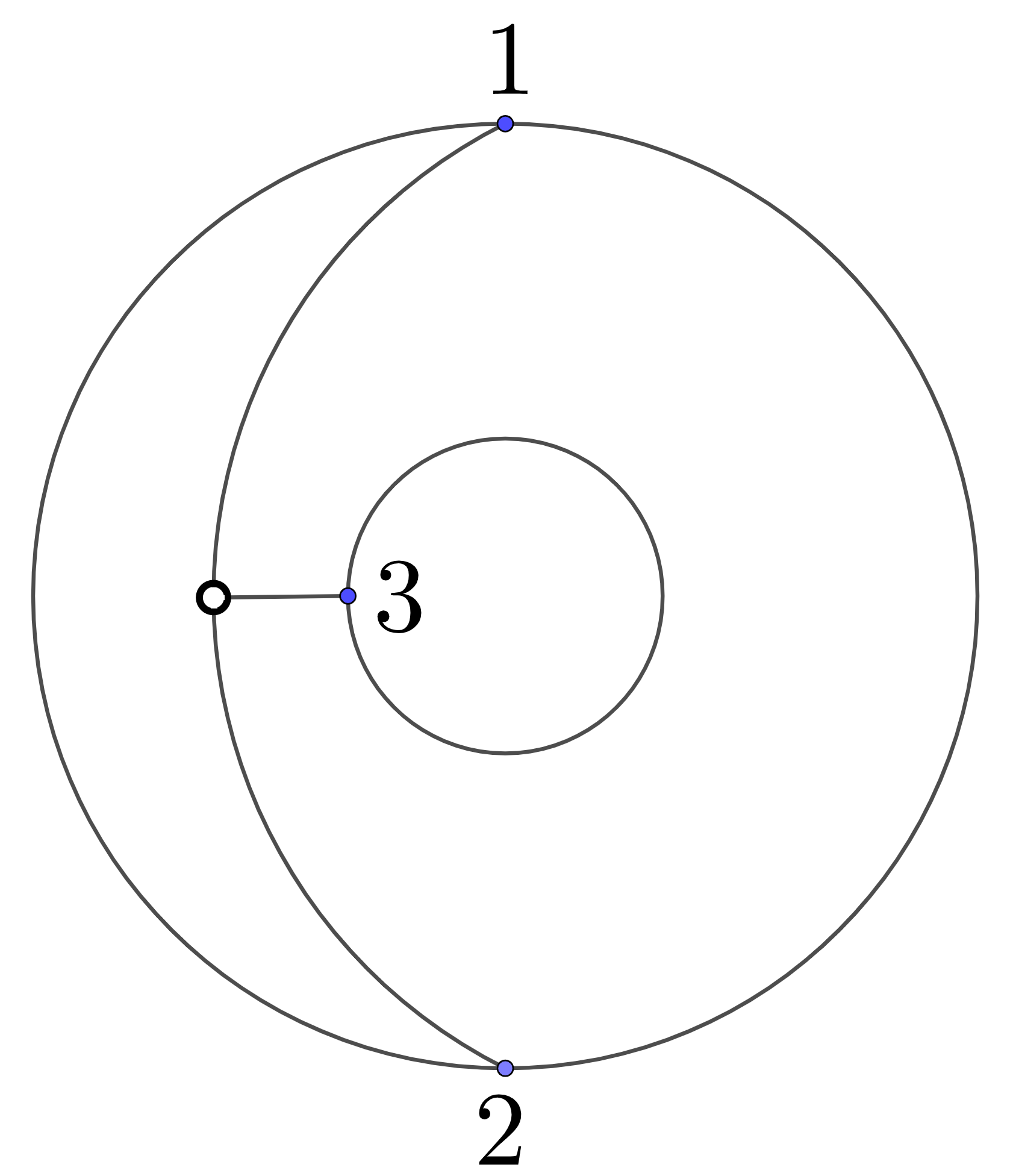}
        \caption{Non-crossing annular permutation.} \label{fig13d}
    \end{subfigure}
    \caption{Here, we make explicit in (a) the bipartite ribbon graph $(1,4)(2,5)(3,6)$ referred to in Figure \ref{fig11}. Identifying and relabelling the half-edges according to $1\equiv2\mapsto1$, $3\equiv4\mapsto2$, $5\equiv6\mapsto3$ shrinks the white areas to vertices and produces the hypermap (b). Cutting this along the dotted line and identifying the left and right sides produces (d). Alternatively, we may arrive at (d) by bicolouring the non-crossing annular pairing of Figure \ref{fig11a} to obtain (c) and then identify half-edges.} \label{fig13}
\end{figure}

Moving on, let us now consider bipartite ribbon graphs on the torus, i.e., those belonging to $\widetilde{a}_{1,p}(n)$. As in the case of $\widetilde{b}_{1,p}(n)$, we are able to identify half-edges of the ribbon graph to form a hypermap on the torus, then cut and reglue as prescribed in \S\ref{s2.2} to form the desired non-crossing annular permutation. Moreover, we are able to instead begin at the non-crossing annular pairing produced by following the algorithm of \S\ref{s2.2} and then identify half-edges to get the same annular permutation. However, there are two closely related subtleties to consider, those being of where to cut the hypermap on the torus and of the fact that if the minimal label, say $u$, on the outer circle of a non-crossing annular pairing is even, we would need to understand how to identify this half-edge with $u-1$, which lies on the inner circle. Both of these issues can be resolved by defining $u\in[2n]$ to be the minimal \textit{odd} label such that $(u,\pi(u))$ is a crossing cycle of our ribbon graph $\pi\in\widetilde{a}_{1,p}(n)$ with $\pi(u)>u$ --- indeed, $u$ was only chosen to be minimal in \S\ref{s2.2} to ensure our mapping was a bijection, but the choice of $u$ could be made unambiguous in many ways. We explain momentarily why there must exist an odd $u$ as described above, but for now, we proceed by assuming existence of said $u$. Then, drawing our ribbon graph on the torus with the edge $(u,\pi(u))$ being vertical and cutting along the left of said edge produces a non-crossing annular pairing that is amenable to being transformed into a non-crossing annular permutation through the half-edge identification mechanism described above. Furthermore, if we opt to first identify half-edges of our ribbon graph to form a hypermap on the torus, we simply need to cut along the left of the half-edge corresponding to $u$ (now labelled $(u+1)/2$) and continue along (passing the white vertex and going through the boundary of the fundamental polygon of the torus) until arriving at the half-edge corresponding to $\pi(u)$ (now labelled $\pi(u)/2$). We give an example in Figure \ref{fig13}.

In the above, we assumed that given $\pi\in\widetilde{a}_{1,p}(n)$, there always exists a minimal odd $u\in[2n]$ such that $(u,\pi(u))$ is crossing and $\pi(u)>u$. We now prove the veracity of this assumption. First, note that by the definition of $\widetilde{a}_{1,p}(n)$, one has by the same argument as in the first observation in the proof of Proposition \ref{prop6} that there exist $1\le a<a'<b<b'\le 2n$ such that $(a,b),(a',b')\in\pi$. Let $a,a'$ be minimal such that this property holds. If either $a,a'$ are odd, then the set of odd $u'\in[2n]$ such that $(u',\pi(u'))$ is crossing with $\pi(u')>u'$ is non-empty and we may set $u$ as the minimal element of this set. Otherwise, we have a parity argument: Suppose that $a,a'$ are both even. Then, since $\pi$ is bipartite, $b,b'$ must both be odd. Then, the set $\{a+1,\ldots,b-1\}\setminus\{a'\}$ contains one more odd label than even label. As odd labels must pair with even ones, at least one of the odd labels in $\{a+1,\ldots,b-1\}\setminus\{a'\}$, say $c$, must be such that $(c,\pi(c))$ crosses with $(a,b)$. Since $a\in[2n]$ is the minimal half-edge participating in a crossing edge, we must have that $\pi(c)>b>c$, and we are done.

Finally, let us consider the set $\widetilde{b}_{2,p}(n)$ of bipartite ribbon graphs on the Klein bottle. As this case is treated by combining the ideas used in the previous two settings, it should come as no surprise that the algorithm for obtaining a non-crossing annular permutation from a ribbon graph $\tau_1\in\widetilde{b}_{2,p}(n)$ is much the same as just discussed. Thus, we draw the restriction of $\tau_1$ to $[2n]\cong\pm[2n]/\tau_0$ on the fundamental polygon of the Klein bottle, as in Figure \ref{fig14a}, and then proceed in one of two ways:
\begin{enumerate}
\item Glue the ribbon graph to its inverted copy (along the bottom boundary of the original and the top boundary of the copy) to form a graph on a torus, cut along an appropriate (vertical) edge and reglue along the left and right boundaries to form a bipartite non-crossing annular pairing, and then pairwise identify half-edges to produce the desired non-crossing annular permutation.
\item Identify half-edges to form a hypermap on the Klein bottle, glue this hypermap to its inverted copy to form a hypermap on the torus, and then cut along the appropriate half-edges and reglue as needed to arrive yet again at the desired non-crossing annular permutation.
\end{enumerate}

The key point is that the subtlety of the choice of half-edges to cut along that arose in the toroidal case has an analogue on the Klein bottle. This is resolved by taking $u$ to be the minimal odd element of $[2n]$ such that $\tau_1(u)\in[2n]$ and cutting the ribbon graph on the torus along the half-edges $\pm u,\pm\tau_1(u)$, with the cuts placed between $\pm u$ and $\pm(1-u)$ and between $\pm\tau_1(u)$ and $\pm(\tau_1(u)-1)$; for cutting the hypermap on the torus, we cut along $(u+1)/2$ and follow the permutation in the natural way. We work through one final example in the following figure.

\begin{figure}[H]
    \centering
    \begin{subfigure}{0.32\textwidth}
        \centering
        \captionsetup{justification=centering}
        \includegraphics[width=0.9\textwidth]{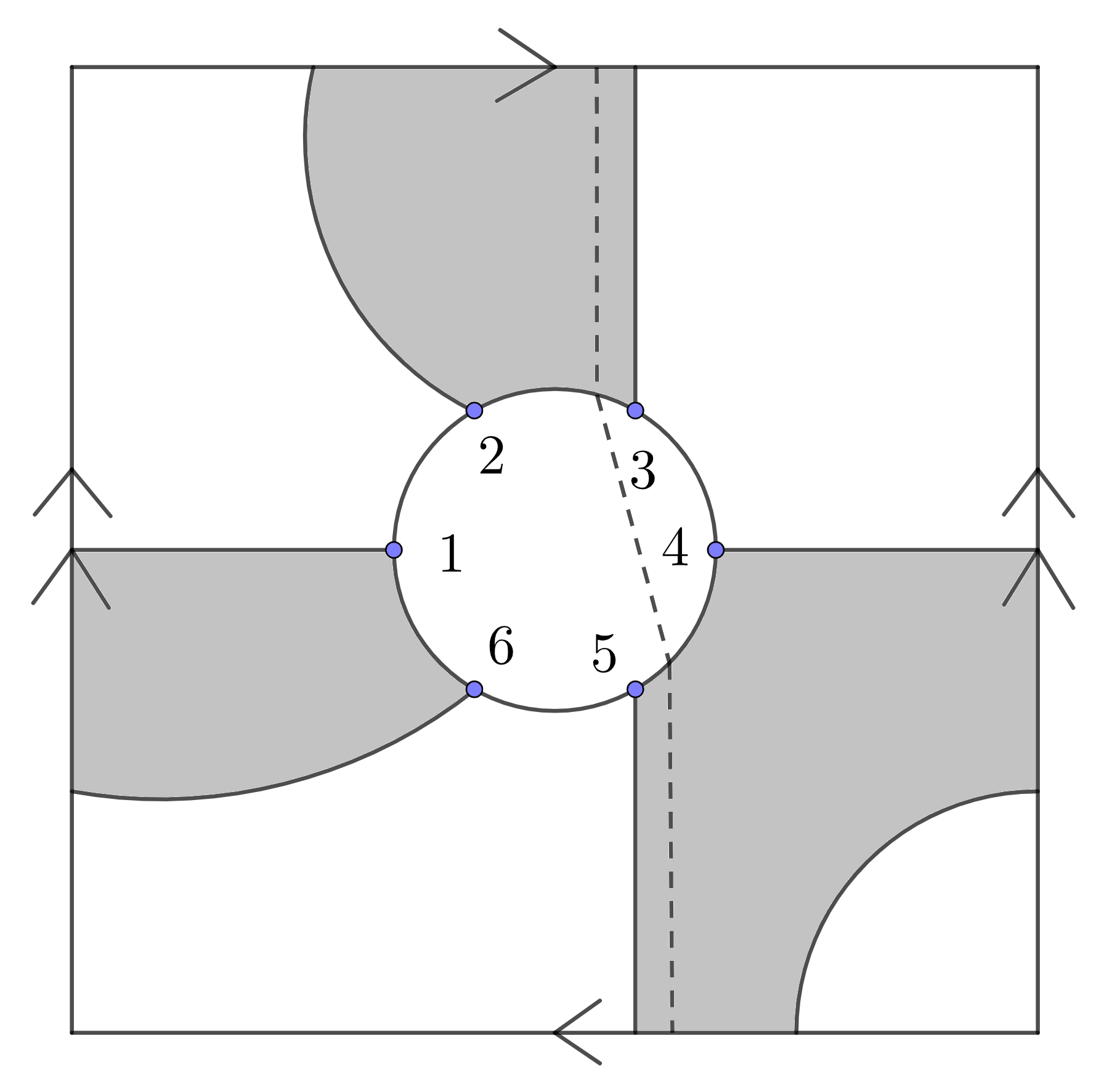}
        \caption{Bipartite ribbon graph on the Klein bottle.} \label{fig14a}
    \end{subfigure}\hfill
    \begin{subfigure}{0.32\textwidth}
        \centering
        \captionsetup{justification=centering}
        \includegraphics[width=0.9\textwidth]{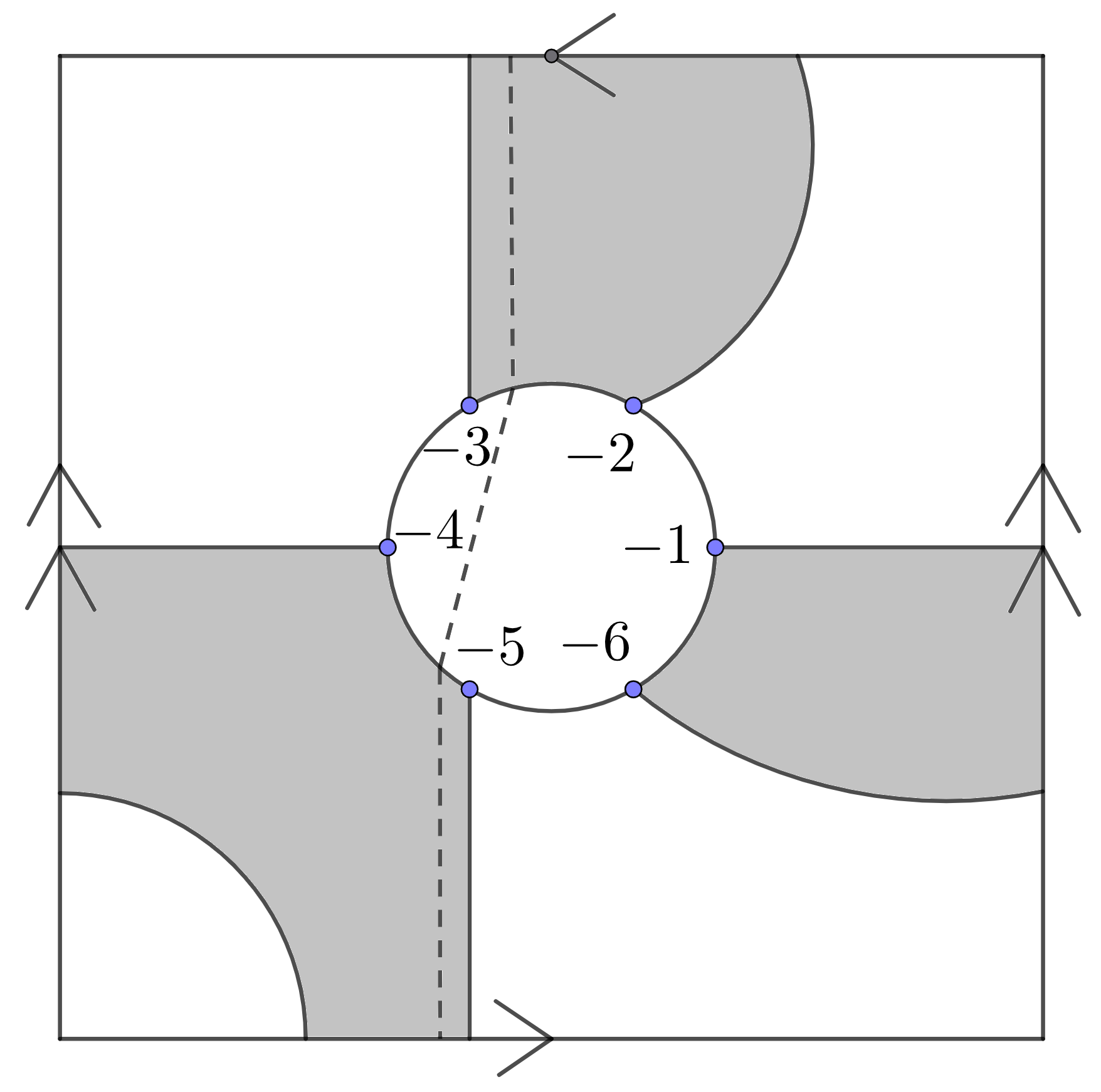}
        \caption{Inverted copy of the bipartite ribbon graph.} \label{fig14b}
    \end{subfigure}\hfill
    \begin{subfigure}{0.32\textwidth}
        \centering
        \captionsetup{justification=centering}
        \includegraphics[width=0.9\textwidth]{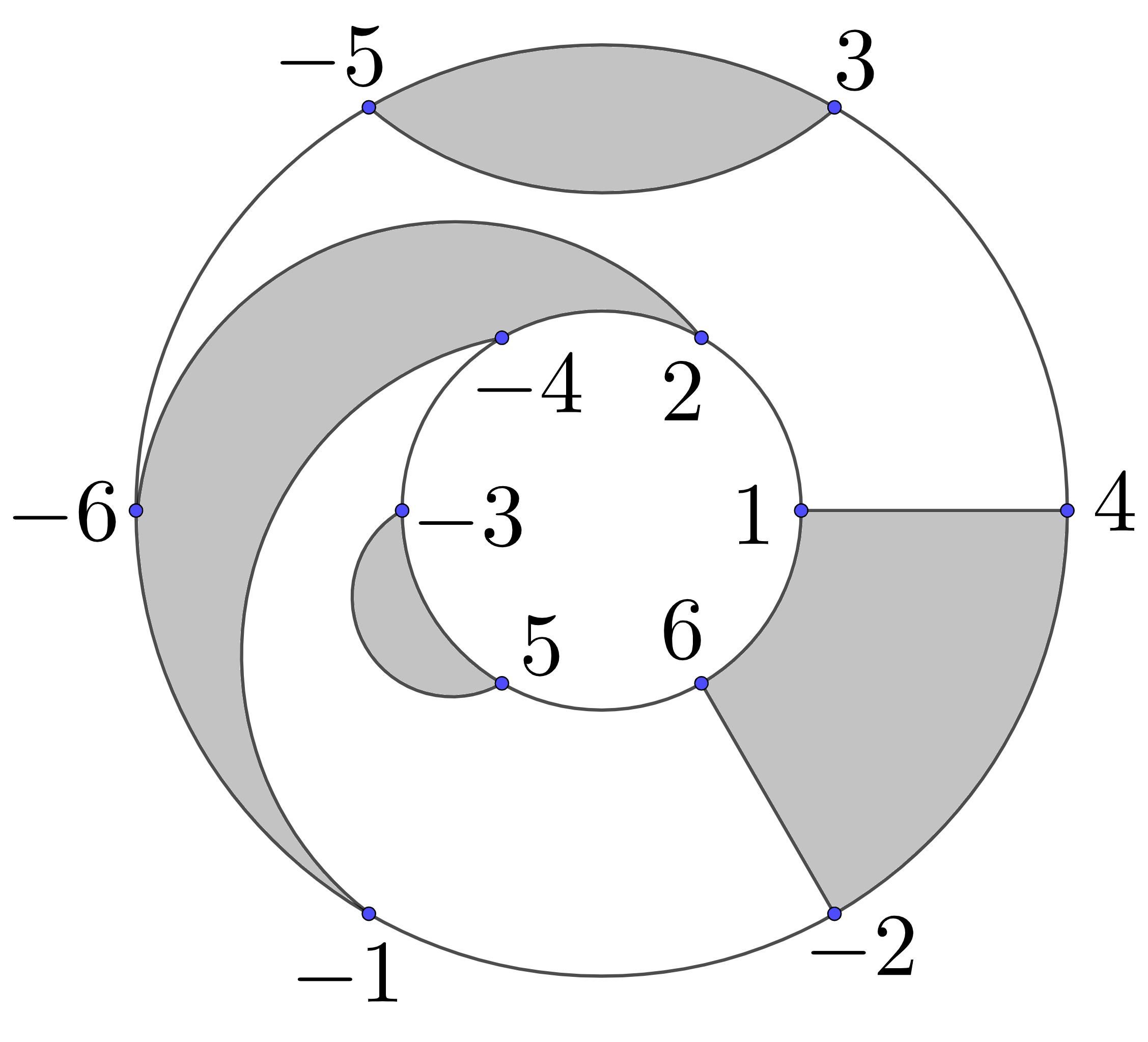}
        \caption{Bipartite non-crossing annular pairing.} \label{fig14c}
    \end{subfigure}

    \begin{subfigure}{0.32\textwidth}
        \centering
        \captionsetup{justification=centering}
        \includegraphics[width=0.9\textwidth]{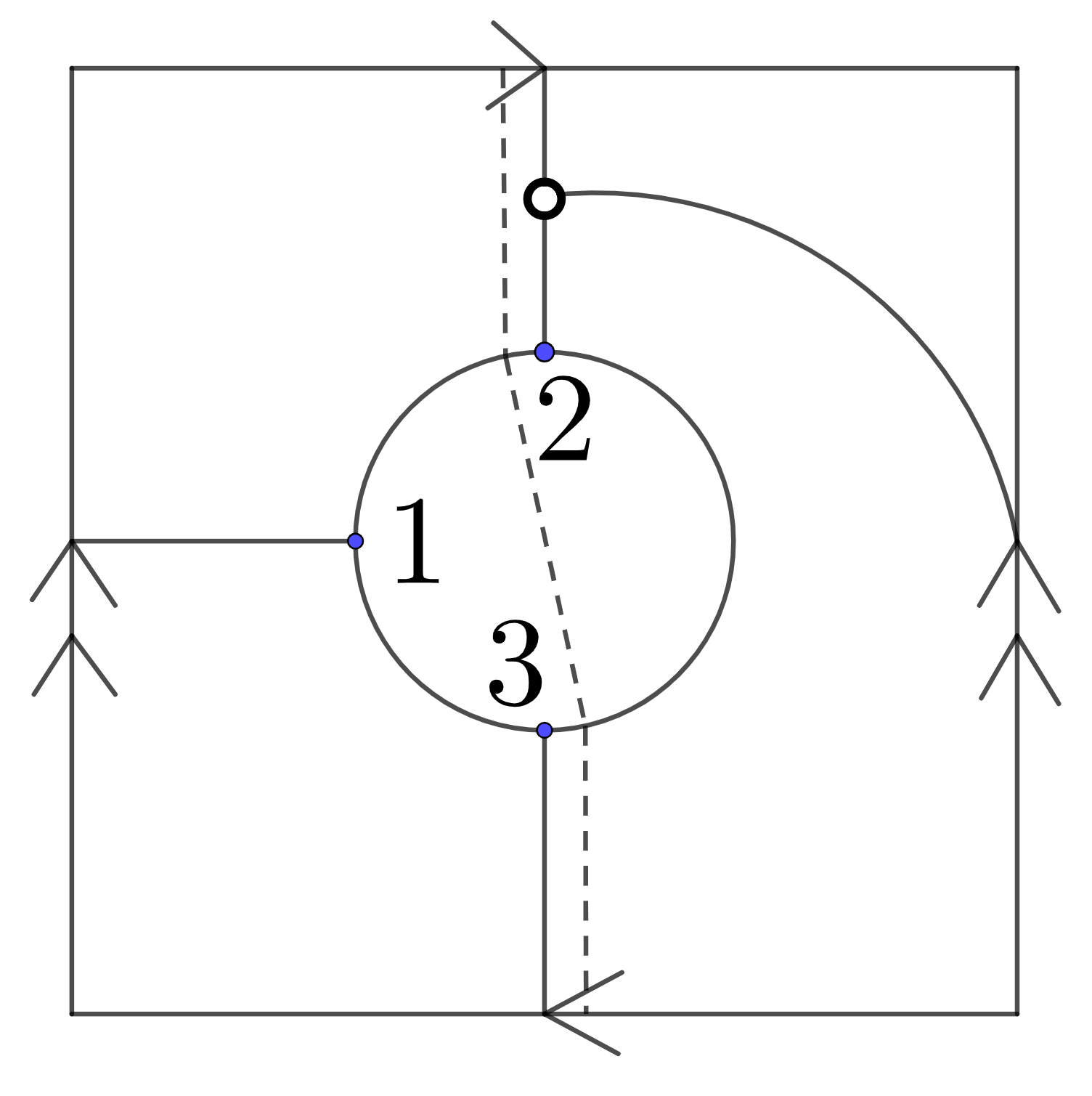}
        \caption{Hypermap on the Klein bottle.} \label{fig14d}
    \end{subfigure}\hfill
    \begin{subfigure}{0.32\textwidth}
        \centering
        \captionsetup{justification=centering}
        \includegraphics[width=0.9\textwidth]{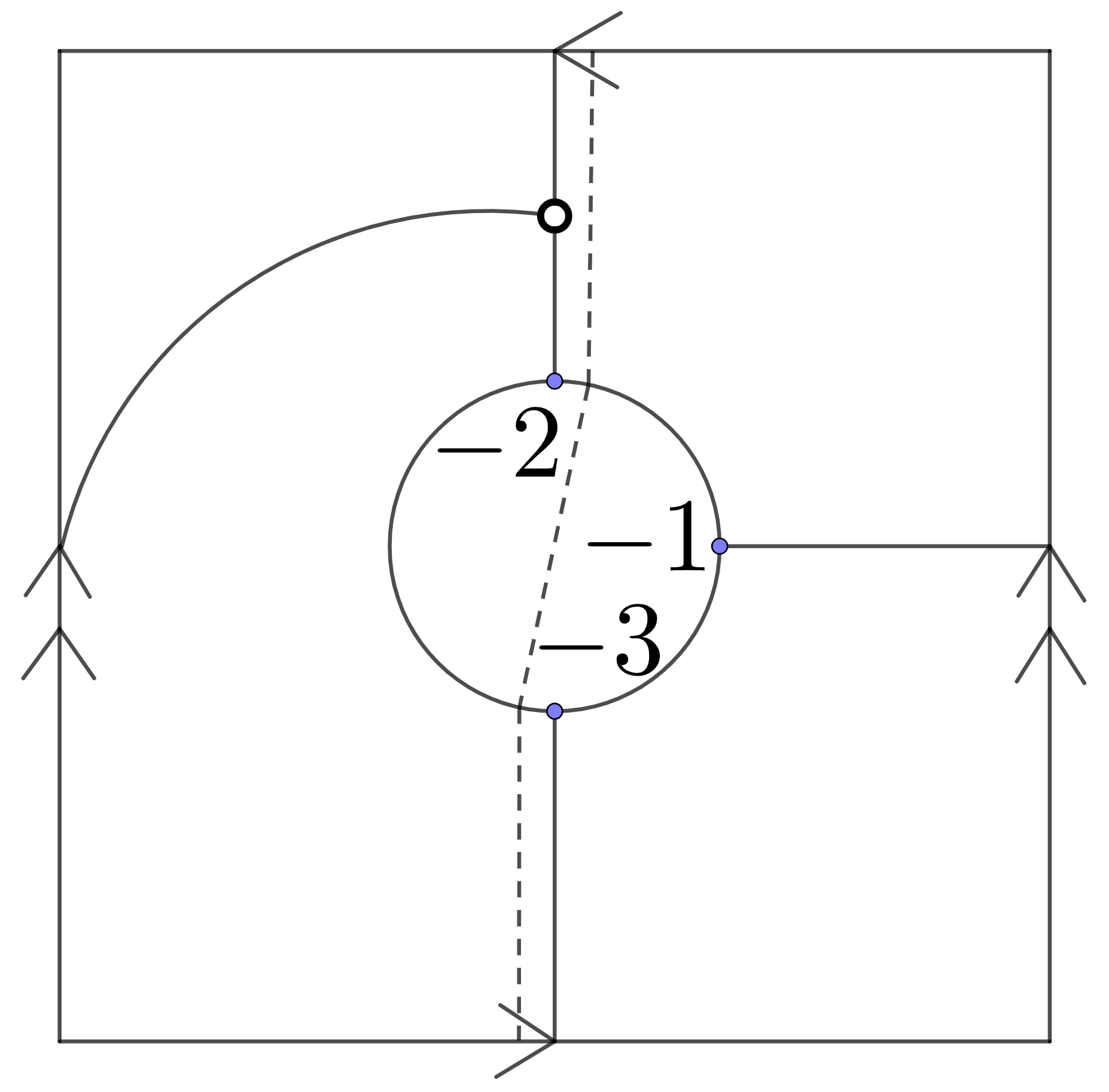}
        \caption{Inverted copy of the hypermap.} \label{fig14e}
    \end{subfigure}\hfill
    \begin{subfigure}{0.32\textwidth}
        \centering
        \captionsetup{justification=centering}
        \includegraphics[width=0.9\textwidth]{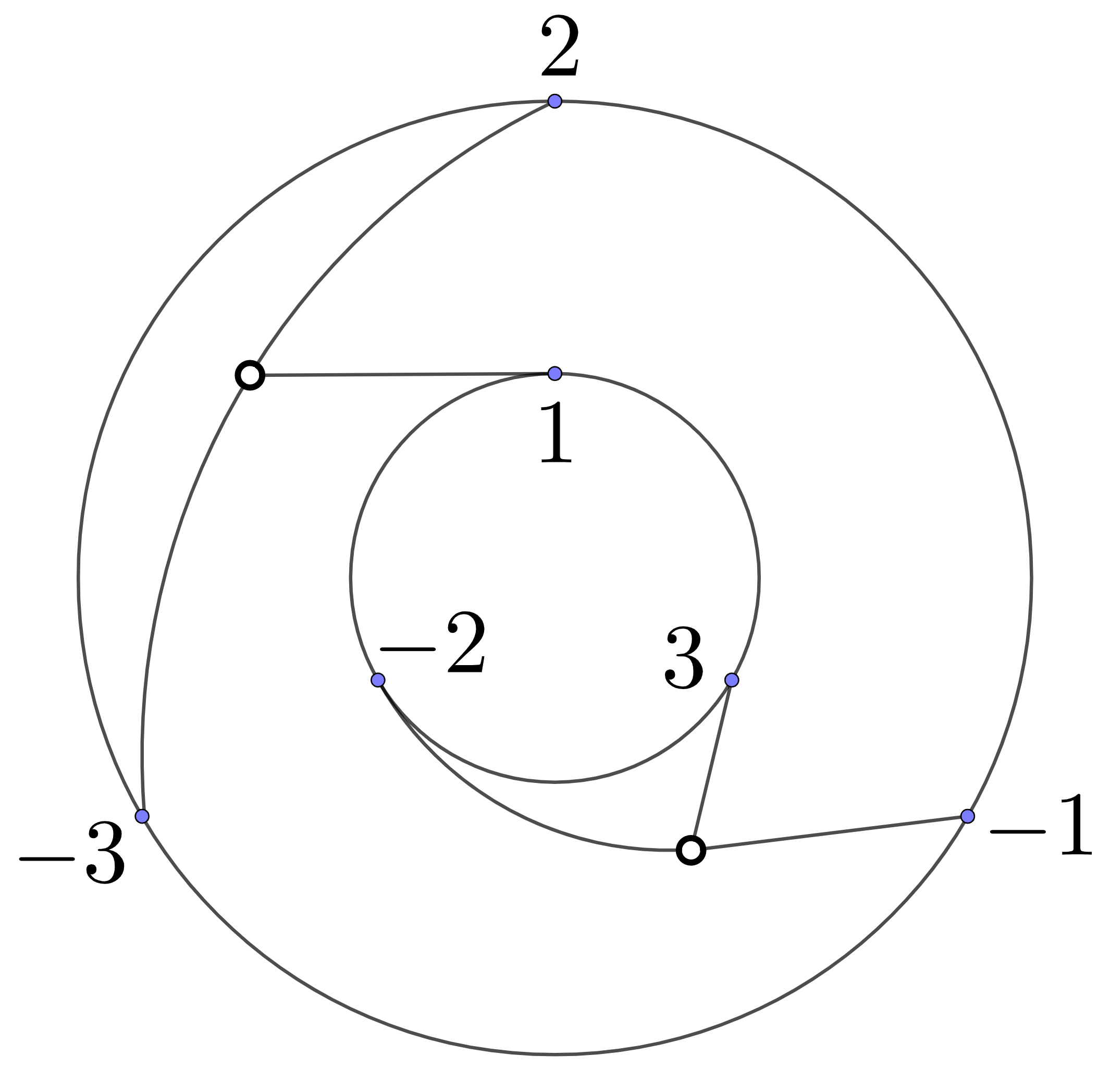}
        \caption{Non-crossing annular permutation.} \label{fig14f}
    \end{subfigure}
    \caption{Although $2\in[6]$ is minimal such that the edge of (a) starting at it traverses the upper boundary, indicating that $\tau_1(2)=6>0$, we will cut along $u=3$, as it is the minimal \textit{odd} element of $[6]$ that maps to $[6]$ under $\tau_1$. Gluing (a) to its inverted copy (b), then cutting along $\pm u,\pm\tau_1(u)$ produces (c). Identifying half-edges then yields (f). Alternatively, we may identify half-edges in (a) to produce (d), glue it to its inverted copy (e), and then cut along the indicated line to form (f).} \label{fig14}
\end{figure}

It remains to justify our assumption in the above that there always exists an odd $u\in[2n]$ such that $\tau_1(u)\in[2n]$. Thus, return to the definition of $\widetilde{b}_{2,p}(n)$ and note that there must exist some $u\in[2n]$ such that $\tau_1(u)\in[2n]$; take $u$ henceforth to be the minimal element of $[2n]$ with this property. If $u$ is odd, we are done. Otherwise, suppose for the sake of contradiction that every odd element $u'$ of $[2n]$ is such that $\tau_1(u')\in-[2n]$. Then, as $\tau_1$ is bipartite, each of these $u'$ must in fact map to an element of $-2[2n]$ under $\tau_1$ so that $\tau_1$ pairs every positive element of $B(n)$ with a negative one. This is in contradiction with the definition of $u$, since the bipartiteness condition means that $\tau_1(u)\in1-2[2n]$, hence the symmetry $\tau_1\tau_0=\tau_0\tau_1$ requires that $\tau_1$ pairs together the two negative elements $-u,-\tau_1(u)\in B(n)$. Thus, there must exist some odd $u'\in[2n]$ such that $\tau_1(u')\in[2n]$ and we are able to have well-defined cuts in the algorithm outlined above.

\begin{acknowledgement}
The authors would like to thank Norman Do for valuable discussion on this project. The first two authors were supported by Hong Kong GRF 16304724 and 17304225. The research of the third author is supported by a Discovery Grant from the Natural Sciences and Engineering Research Council of Canada. This work was initiated during  the event \emph{Log-gases in Caeli Australi: Recent Developments in and around Random Matrix Theory} held at the MATRIX Institute during August 2025. Rahman also benefited from a visit to the University of Melbourne funded by the MATRIX--Simons Young Scholarship program.
\end{acknowledgement}

\end{document}